\documentclass{article}
\usepackage[T1]{fontenc}
\usepackage[utf8]{inputenc}
\usepackage[american]{babel}
\usepackage{csquotes}
\usepackage{amsthm,amsmath,mathtools,mathrsfs,latexsym}
\usepackage{amssymb}
\usepackage{dsfont}
\usepackage{comment}
\usepackage{color}
\usepackage{enumitem}
\usepackage[colorlinks=true,linkcolor=blue,citecolor=blue]{hyperref}
\usepackage{titletoc}

\usepackage[a4paper, top=2.5cm, bottom=2.5cm, left=2.5cm, right=2.5cm]{geometry}

\usepackage[toc,page]{appendix}

\numberwithin{equation}{section}
\mathtoolsset{showonlyrefs}

\usepackage[backend=biber,style=alphabetic,giveninits=true,sorting=nyt,natbib=true,url=false]{biblatex}
\DeclareNameAlias{default}{family-given}
\DeclareNolabel{\nolabel{\regexp{[\p{P}\p{S}\p{C}\p{Z}]+}}}
\bibliography{Biblio, Biblio_Fluido}

\newcommand{\N}{\mathbb{N}}
\newcommand{\R}{\mathbb{R}}
\newcommand{\E}{\mathbb{E}}
\newcommand{\F}{\mathcal{F}}
\newcommand{\C}{\mathbb{C}}

\newcommand{\id}{\mathbb{I}}
\newcommand{\Tr}{\operatorname{Tr}}
\newcommand{\supp}{\operatorname{supp}}

\newcommand{\1}{\mathds{1}}
\newcommand{\lan}{\langle}
\newcommand{\ran}{\rangle}
\newcommand{\doublehookrightarrow}{\lhook\joinrel\relbar\mspace{-10mu}\hookrightarrow}
\newcommand{\ds}{\displaystyle}
\newcommand{\itref}[1]{$\textit{\ref{#1}}$}

\newcommand\dela[1]{}

\renewcommand{\P}{\mathbb{P}}
\renewcommand{\d}{\,\mathrm{d}}
\renewcommand{\div}{\operatorname{div}}
\renewcommand{\ss}{\scriptscriptstyle}
\DeclareMathOperator*{\esssup}{ess\,sup}

\def\a{{\alpha}}
\def\b{{\beta}}
\def\de{{\delta}}
\def\la{{\lambda}}

\def\s{{\sigma}}

\def\g{{\gamma}}
\def\e{{\varepsilon}}

\newtheorem{theorem}{Theorem}[section]
\newtheorem{proposition}[theorem]{Proposition}
\newtheorem{lemma}[theorem]{Lemma}
\newtheorem{corollary}[theorem]{Corollary}

\theoremstyle{definition}
\newtheorem{definition}[theorem]{Definition}

\newtheorem{notation}[theorem]{Notation}

\theoremstyle{remark}
\newtheorem{remark}[theorem]{Remark}

\title{\vskip -1cm  Inviscid Limit of the Stochastic Hyperviscous Navier-Stokes Equations and Invariant Measures for the Euler Equations in $\R^2$} 
\author{
    Zdzis{\l}aw Brze\'{z}niak\footnote{Corresponding author.} \\
    {\footnotesize Dept.~of Mathematics, University of York}\\
    {\footnotesize \texttt{zdzislaw.brzezniak@york.ac.uk}}\\
    {\footnotesize \url{https://orcid.org/0000-0001-8731-6523}}
\and 
    Matteo Ferrari \\
    {\footnotesize Dept.~of Mathematics $``$Felice Casorati$"$, University of Pavia}\\
    {\footnotesize Dept.~of Mathematics and Applications, University of Milano Bicocca}\\
    {\footnotesize \texttt{matteo.ferrari@unimib.it}}\\
    {\footnotesize \url{https://orcid.org/0009-0000-6585-2210}}}
\date{June 25, 2026}

\titlecontents{section}
  [0em] 
  {} 
  {\thecontentslabel.\quad} 
  {} 
  {\titlerule*[.5pc]{.}\contentspage} 

\begin{document}

\maketitle

\begin{abstract}
    We prove the existence and some moment estimates for an invariant measure $\mu$ for the two-dimensional ($2$D) deterministic Euler equations on the unbounded domain $\R^2$ and with highly regular initial data. 
    The result is achieved by first showing the existence of Markov (marginally) stationary processes which solve the hyperviscous $2$D Navier-Stokes equations with kinematic viscosity $\nu>0$ and an additive stochastic noise scaling as $\sqrt \nu$. 
    We then study the inviscid limit and prove that, as $\nu$ tends to $0$, these processes converge, in an appropriate trajectory space, to a pathwise solution to the Euler equations. 
    Its law is the sought invariant measure $\mu$.
\\[3mm]
\\[3mm]

\end{abstract}


\tableofcontents
\newpage

\section{Introduction}
The most famous equations in fluid dynamics, at least from a mathematical point of view, are the Navier-Stokes equations for homogeneous incompressible fluids.
\begin{equation}
\label{EQ:proto_NSE}
    \begin{cases}
        \partial^{}_tu-\nu\Delta u+(u\cdot\nabla)u =f-\frac{1}{\rho}\nabla p,\\
        \div u=0.\\
    \end{cases}
\end{equation}
They describe time and space evolution of the flow-velocity vector field $u$ and the pressure scalar field $p$ of any fluid, given constant parameters $\nu,\rho>0$ and a vector field $f$. 
The parameter $\nu$ is the kinematic viscosity, $\rho$ is the density of the fluid, while $f$ is interpreted as an external force per unit of mass acting on the fluid and can be either deterministic or stochastic. 

By formally setting the kinematic viscosity $\nu$ to zero, one obtains the Euler equations for homogeneous incompressible fluids.
\begin{equation}
\label{EQ:proto_EE}
    \begin{cases}
        \partial^{}_tu+(u\cdot\nabla)u =f-\frac{1}{\rho}\nabla p,\\
        \div u=0.\\
    \end{cases}
\end{equation}
The main physical difference between the two systems is that the Navier–Stokes equations describe the motion of viscous Newtonian fluids, whereas the Euler equations model the behaviour of ideal (inviscid) fluids, in which the effects of viscosity are neglected.
Although truly inviscid flows do not exist in nature, the Euler equations provide a fundamental model for high–Reynolds-number regimes in which viscous effects are negligible compared to inertial and pressure forces. They play a central role in both theoretical fluid mechanics and numerical simulations, for instance in the modelling of atmospheric and oceanic flows, or in gas dynamics of stellar interiors where viscosity is extremely small (see, \textit{e.g.}, \cite{Edelmann+Horst+Berberich+Andrassy+Higl+Leidi+Klingenberg+Ropke_2021_Well-balanced_treatment,Landau+Lifshitz_1987_Fluid_Mechanics, Toro_2009_Riemann_Solvers, LeVeque_2002_Finite_volume,Majda+Bertozzi_2002_Vorticity_incompressible_flow}).

The problems of the well-posedness, the existence and the uniqueness of solutions for both equations have been widely investigated in literature and are completely solved in the two-dimensional ($2$D) case, see, for instance, \cite{Wolibner_1933_Theoreme_existence_mouvement_plan_fluide_parfait_homogene_incompressible_pendant_temps_infiniment_long, Ladyzhenskaya_1969_Mathematical_Theory, Temam_2001_Navier-Stokes_equations, Judovic_1963_Non-stationary_flows_ideal_incompressible_fluid, Majda+Bertozzi_2002_Vorticity_incompressible_flow}.
Further interesting properties, both from a mathematical viewpoint and for the practical implications in understanding the long-term behaviour and statistical properties of turbulent flows, concern the invariant measures for the two equations.
Formally speaking, invariant measures are spatial distributions of the fluid, which remain stationary as time evolves. 
Rigorous definitions will be given in Sections~\ref{SEC:Notations} and \ref{SEC:markov_invariant_NS}.

A classical method for proving the existence of invariant measures for partial differential equations, both in the deterministic and stochastic setting, is the Krylov-Bogoliubov method; see, for instance, \cite{DaPrato+Zabczyk_2014_stochastic_equations_infinite, DaPrato+Zabczyk_1996_ergodicity_infinite, Brzezniak+Gatarek_1999_Martingale_solutions_invariant_measures_stochastic_evolution_equations_Banach_spaces}. 
It has been successfully applied to the stochastic Navier-Stokes equations, starting from the celebrated paper \cite{flandoli_1994_dissipativity}.
Instead, concerning the deterministic Euler equations, earlier works have identified invariant measures, in the case of bounded domains and periodic boundary conditions,  often employing the Gibbs measure \cite{Albeverio+DeFaria+Hoegh-Krohn_1979_stationary, Albeverio+Hoegh-Krohn_1989_Stochastic_flows_stationary_distribution_two-dimensional_inviscid_fluids, Albeverio+Cruzeiro_1990_Global_flows_invariant, Cipriano_1999_two-dimensional_Euler_equation_statisticalstudy, Biryuk_2006_invariantmeasures_2D_Euler_equation}. 
Kuksin's influential studies have demonstrated that stationary solutions for the Navier-Stokes equations with stochastic forcing converge to non-trivial stationary solutions of the deterministic Euler equations as viscosity vanishes and the noise vanishes as well \cite{Kuksin_2004_Eulerian_limit_2D_statisticalhydrodynamics, Kuksin_2006_Randomly_forced_nonlinear_PDEs_statistical_hydrodynamics_2_space_dimensions, Kuksin_2007_Rigorous_results_conjectures_stationary_space-periodic_2D_turbulence, Kuksin_2008_distribution_energy_vorticity_solutions_2D_Navier-Stokes_equation_small_viscosity}.

The approach in this paper builds upon foundational works, including \cite{Ferrario_2023_Eulerian_limits_Kuksin} and \cite{latocca_2023_construction_high_regularity}, which investigate the Eulerian limits and their connection to the stationary solutions of the $2$D hypoviscous and hyperviscous Navier-Stokes equations as viscosity vanishes, in the case of bounded domains with periodic boundary conditions. 
Instead, we focus on the case of the unbounded domain $\mathbb{R}^2$, inspired by recent advances in stochastic partial differential equations (SPDEs) on unbounded domains, which involve the continuity of the solution flow with respect to weak topologies, a crucial aspect for handling the lack of compact embeddings in unbounded domains.

Specifically, we start from the stochastic hyperviscous $2$D Navier-Stokes equations on the space of diverge-free, square integrable vector fields over $\R^2$:
\begin{equation}
\label{EQ:stoch_NS_proto}
    \d X_t + \big[\nu A^\a X_t + B(X_t)\big]\d t =\b_\nu \d W_t, \qquad \textit{for  } t> 0,
\end{equation}
where $\a>1$, $W$ is a cylindrical Wiener noise, $\b_\nu>0$ is a coefficient to be appropriately chosen, $A$ denotes the Stokes operator, and $B$ the Stokes nonlinearity. 
The higher-order dissipation terms given by the power $\a>1$ of the Stokes operator regularise the solutions, and allow us to derive powerful \textit{a priori} estimates. 
The scaling factor $\b_\nu$ in front of the external forcing is of paramount importance. 
The following theorem gathers known results from the literature which clarify the role this factor plays in the deterministic setting on bounded domains.
\begin{theorem}[\text{\cite[Theorem $10.1$]{Temam_1995_Navier-Stokes_equations_nonlinear_functional_analysis}}]
\label{TH:Temam}
    Assume that $\mathcal D\subset \R^2$ is a bounded domain with periodic or Dirichlet boundary conditions and let $V$ denote the space of divergence-free vector fields in $H^1(\mathcal D;\R^2)$. 
    If $\nu>0$ and $f\in V'$, then there exists $u\in V$ such that 
    \begin{equation}
    \label{EQ:weak_stationary_NS}
        \nu\lan Au,v\ran^{}_{\ss{\!V'\!\times\!V}}+b(u,u,v)={\lan}f,v\ran^{}_{\ss{\!V'\!\times\!V}}, \qquad \forall\, v\in V.
    \end{equation}
    Moreover, if there exists $c>0$ depending only on $\mathcal D$, such that 
    \[
        \|f\|^{}_{V'}<c\nu^2,
    \]
    then the $u\in V$ from above is unique.
\end{theorem}
\begin{remark}
\label{REM:temam}
    The last theorem motivates the following claim: if the coefficient $\b_\nu$ in equation~\eqref{EQ:stoch_NS_proto} converges sufficiently rapidly to $0$, as $\nu\to 0$, then the solutions of the Navier-Stokes equations vanish in the limit $\nu\to 0$ and converge to zero solution to the Euler equations. 
    We will make this claim more precise.
    
    Let us consider the setting of Theorem~\ref{TH:Temam}.
    Assume that $\tilde f\in V'$.
    For all $\nu>0$, let $\b_\nu>0$, and let $u^\nu\in V$ satisfy equation~\eqref{EQ:weak_stationary_NS} with $f=\b_\nu\tilde f$.
    Then, in the limit as $\nu\to 0$, the following statements hold:
    \begin{itemize}
        \item 
            If $\b_\nu=o(1)$, then $\nu u^\nu\longrightarrow 0$ in $V$.
            In particular $B(u^\nu)\longrightarrow 0$ in $V'$.
        \item 
            If $\b_\nu=o(\nu)$, then $u^\nu\longrightarrow 0$ in $V$.
    \end{itemize}

    Indeed, by testing the equation~\eqref{EQ:weak_stationary_NS} with $f=\b_\nu \tilde f$ against $u^\nu\in V$, and recalling that $b(u^\nu,u^\nu,u^\nu)=0$, we get:
    \[
        \nu \|u^\nu\|^2_{V}=\b_\nu{\lan}\tilde f,u^\nu\ran^{}_{\ss{\!V'\!\times \!V}}\leq \b_\nu \|\tilde f\|^{}_{V'}\|u^\nu\|^{}_V.
    \]
    This implies
    \[
        \nu \|u^\nu\|^{}_{V}\leq \b_\nu \|\tilde f\|^{}_{V'},
    \]
    from which the assertions follow.
\end{remark}
Therefore, if we want to construct an invariant measure for the Euler equations by studying the inviscid limit of the Navier-Stokes equations, then we need $\b_\nu$ to converge to $0$ less rapidly than $\nu$. 

\begin{remark}
Let us now consider the setting of our equation~\eqref{EQ:stoch_NS_proto} with $\a>1$, and hint at the reason for our choice of the constant $\b_\nu$, see Theorem~\ref{TH:moment_estimates_mu_nu} for the details.
Assume that $(H, \|\,\cdot\,\|)$ is the space of square-integrable divergence-free vector fields on $\R^2$. 
If we suppose that the $H$-valued process $X^\nu$ is a solution to the equation~\eqref{EQ:stoch_NS_proto}, then the It\^o formula applied to the process $\|X^\nu\|^2$ gives:
\begin{equation}
    \E\|X^\nu_t\|^2
    +2\nu \E\int_0^t\|A^{\frac\a2}X^\nu_s\|^2\d s=\E\|X^\nu_0\|^2
    + {\b_\nu^2}Ct, \qquad \forall\, t\geq 0,
\end{equation}
for a constant $C>0$ that depends on the Wiener process.
If, in addition, $X^\nu$ is marginally stationary, \textit{i.e.} its law is an invariant measure,  hence constant in time, then the first terms in both sides of the equality cancel out, and the integral in time is easily computed. 
We are left with:
\[
    \nu \, t\bigg(2\E\|A^{\frac\a2}X^\nu_0\|^2-\frac{\b_\nu^2}{\nu}C\bigg)=0, \qquad \forall\, t\geq 0.
\]
The last equality, together with the marginal stationarity of the law, implies:
\[
    \E\big\|A^{\frac\a2}X^\nu_s\big\|^2=\E\big\|A^{\frac\a2}X^\nu_0\big\|^2=\frac{\b_\nu^2}{\nu}C,\qquad \forall\, s\geq 0,
\]
which becomes a uniform estimate with respect to the kinematic viscosity $\nu$ if $\b_\nu=\sqrt \nu$, and thus gives a powerful property to study the inviscid limit.
Moreover, it is shown in \cite{Kuksin_2004_Eulerian_limit_2D_statisticalhydrodynamics} that this scaling is the only one which leads to non-trivial limiting measures which are invariant for the 2D Euler equations. 
This scaling limit was also considered in the important paper \cite{Glatt-Holtz+Sverak+Vicol_2015_On_inviscid_limits}, and was
also used outside of the context of invariant measures for a Kato-type boundary layer problem in \cite{Goodair+Crisan_2025_The_zero_viscosity_limit_stochastic, Luongo_2024_Inviscid_limit_stochastic}.
In fact in \cite{Goodair+Crisan_2025_The_zero_viscosity_limit_stochastic} separate arguments are used to justify an ‘optimality’ of the scaling by $\sqrt\nu$.
\end{remark}
The last remark, inspired by \cite[Proposition $4.2$]{latocca_2023_construction_high_regularity}, motivates the choice of the factor $\b_\nu=\sqrt \nu$ in our equation~\eqref{EQ:stoch_NS_proto}.

Once the well-posedness of the equation~\eqref{EQ:stoch_NS_proto} with $\b_\nu:=\sqrt \nu$ has been investigated, we establish the existence of invariant measures, by employing a version of the Krylov–Bogoliubov method tailored to weak topologies.
This technique was introduced in \cite{Maslowski+Seidler_1999_On_sequentially_weakly_Feller}, and successfully applied to the stochastic nonlinear beam and wave equations \cite{Brzezniak+Ondrejat+Seidler_2016_Invariant_measures_stochastic_nonlinear_beam_wave_equations}, the Navier-Stokes equations in unbounded domains \cite{Brzezniak+Motyl+Ondrejat_2017_Invariant_Measure_Stochastic_N-S, Brzezniak+Ferrario_2019_stationary_solutions}, the stochastic Landau-Lifshitz-Bloch equation \cite{Brzezniak+Goldys+Le_2020_Existence_unique_solution_invariant_measures_stochastic_Landau-Lifshitz-Bloch_equation}, the stochastic damped Euler equation \cite{Bessaih+Ferrario_2020_Invariant_measures_stochastic_damped_2D_Euler_equations}, and the stochastic nonlinear and damped Schr\"odinger equation \cite{Brzezniak+Ferrario+Zanella_2024_Invariant_measures_stochastic_nonlinear_damped_2D_Schrodinger_equation}.
The lack of compactness has been addressed also in \cite{Brzezniak+Li_2006_Asymptotic_Compactness_Absorbing_Sets, Brzezniak+Li_2004_Asymptotic_behaviour_solutions_2D_stochastic_Navier-Stokes_equations_unbounded_domains_new_developments}, where the first named author and Li established the existence of a compact absorbing set for the $2$D stochastic Navier-Stokes equations with additive noise in a certain class of unbounded domains. 

Later, we use the found invariant measure to construct (\textit{marginally}) stationary solutions to the stochastic hyperviscous $2$D Navier-Stokes equations, and finally perform the inviscid limit by means of the Jakubowski's version \cite[Theorem $2$]{Jakubowski_1997_almost_sure_skorokhod} of the Skorokhod Theorem.
To this end, we find a sufficiently large space in which the laws of the viscous sequence are tight. 
This space has also to be small enough so that, after using the Jakubowski theorem, the convergence in that space is strong enough for the sequences of viscous solutions, as well as some auxiliary processes, to be convergent.
A similar technique has been applied for finite-dimensional approximations to SPDEs in \cite{Brzezniak+Ondrejat_2013_stochastic_geometric_wave, Brzezniak+Motyl_2013_Existence_Martingale_Solution, Brzezniak+Motyl+Ondrejat_2017_Invariant_Measure_Stochastic_N-S}. 

Finally, the process resulting from the inviscid limit is proved to be \textit{marginally} stationary and its law is the sought invariant measure for the Euler equations, which inherits the moment estimates valid for the approximating sequence.

The structure of the paper is as follows.
In Section~\ref{SEC:Notations}, we present some notations and preliminaries on frequently used functional spaces, operators and stochastic processes. 
Section~\ref{SEC:stoch_hyp_2D_NSE} formulates and solves the stochastic hyperviscous $2$D  Navier-Stokes equations, whose Markov property and invariant measure are later studied in Section~\ref{SEC:markov_invariant_NS}.
In Section~\ref{SEC:stationary_solution_EE}, we rigorously pass the stochastic Navier-Stokes equations to the inviscid limit. 
Hence, we derive an invariant measure for the deterministic Euler equations.
Finally, the appendices gather some auxiliary results and technical lemmas used throughout the paper.

\section{Notations and preliminaries}
\label{SEC:Notations}
\subsection{General notations}
\begin{notation}
\label{NOT:sequential_spaces}
    Let $\mathcal X,\mathcal Y$ be topological spaces. 
    We denote by $C(\mathcal X;\mathcal Y)$ the vector space of continuous functions $f:\mathcal X\to\mathcal Y$.
    If $\mathcal Y=\R$, we simply write $C(\mathcal X)$ instead of $C(\mathcal X;\R)$. 
    The subspace of $C(\mathcal X)$ consisting of bounded functions is denoted by $C_b(\mathcal X)$.
    A function $f:\mathcal X\to \mathcal Y$ is sequentially continuous, and we write $f\in SC(\mathcal X;\mathcal Y)$,  if for any converging sequence $\{x_n\}_{n\in\N}$ in $\mathcal X$, $x\in\mathcal X$ being its limit, the sequence $\{f(x_n)\}_{n\in\N}$ converges to $f(x)$  in $\mathcal Y$.
    If $\mathcal Y=\R$ we write $SC(\mathcal X)$ instead of $SC(\mathcal X;\R)$.
    
    If $\mathcal X$ is compact and $(\mathcal Y,\|\,\cdot\,\|)$ is a normed vector space, then $C(\mathcal X;\mathcal Y)$ is endowed with the usual supremum norm, \textit{i.e.} $\|f\|^{}_{C(\mathcal X;\mathcal Y)}:=\sup_{x\in \mathcal X}\|f(x)\|$, for $f\in C(\mathcal X;\mathcal Y)$.
\end{notation}

\begin{notation}
\label{NOT:Borel_probabilities}
    Assume that $\mathcal X$ is a topological space.
    We denote by $\mathscr B_{\mathcal X}$ its Borel $\sigma$-algebra and by $\mathscr P(\mathcal X)$ the family of probability measures on $(\mathcal X,\mathscr B_{\mathcal X})$.
    Moreover, $\mathcal B_b(\mathcal X)$ will denote the Banach space of measurable and bounded functions $\varphi:(\mathcal X,\mathscr B_{\mathcal X})\to(\R,\mathscr B_\R)$, endowed with the sup norm, \textit{i.e.}, $\sup_{x\in\mathcal X}|\varphi(x)|$, for $\varphi\in\mathcal B_b(\mathcal X)$.
    
    Let $\mu, \mu_n\in \mathscr P(\mathcal X)$, $n\in\N$.
    We say that $\mu_n\longrightarrow \mu$ in the weak sense in $\mathscr P(\mathcal X)$, or simply in $\mathscr P(\mathcal X)$, if $\lim_{n\to\infty}\int_\mathcal Xf\d\mu_n=\int_\mathcal Xf\d\mu$ for all $f\in C_b(\mathcal X)$.
    Assume that $\mu\in\mathscr P(\mathcal X)$ and that $\mathcal Y$ is another topological space. 
    If $\Phi:\mathcal X\to\mathcal Y$ is a Borel measurable function, we denote the pushforward measure of $\mu$ via $\Phi$ by:
    \[
        \Phi_\ast\mu:\mathscr B_{\mathcal Y}\ni E\mapsto \mu\big(\Phi^{-1}(E)\big)\in[0,1],
    \]
    which is a probability measure on $(\mathcal Y,\mathscr B_{\mathcal Y})$.
\end{notation}

\begin{notation}
    A filtration $\{\F_t\}_{t\geq 0}$ on a probability space $(\Omega,\F,\P)$ is said to be augmented or to satisfy the usual conditions if it is right-continuous and complete.
\end{notation}

\begin{notation}
\label{NOT:embeddings}
    Assume that $\mathcal X$, $\mathcal Y$ are topological spaces.
    We say that $\mathcal X$ is continuously embedded into $\mathcal Y$, and we write $\mathcal X\hookrightarrow \mathcal Y$, if $\mathcal X\subset \mathcal Y$ and the map $\iota:\mathcal X\ni x\mapsto x\in\mathcal Y$ is continuous.
    The map $\iota$ will be referred to as the natural embedding. 

    Assume now that $\mathcal X$ and $\mathcal Y$ are Banach spaces.
    We say that $\mathcal X$ is compactly embedded into $\mathcal Y$, and we write $\mathcal X\doublehookrightarrow \mathcal Y$, if $\mathcal X\hookrightarrow \mathcal Y$ and the natural embedding is a compact linear operator. 
\end{notation}

\begin{notation}
\label{REM:linear_operators}
    Assume that $(\mathcal X,\|\,\cdot\,\|)$ is a Banach space.
    We denote by $\mathcal L(\mathcal X)$ the Banach space of linear and bounded operators $T:\mathcal X\to\mathcal X$, endowed with the usual operator norm  $\|T\|^{}_{\mathcal L(\mathcal X)}:=\sup_{x\in\mathcal X\setminus\{0\}}\frac{\|Tx\|}{\|x\|}$, for $T\in\mathcal L(\mathcal X)$.
    
    Assume that $\big(\mathcal H,\lan\,\cdot\,,\,\cdot\,\ran\big)$ is a separable Hilbert space, with induced norm $\|\,\cdot\,\|$.
    The symbol $\mathcal L_1(\mathcal H)$ will denote the subspace of $\mathcal L(\mathcal H)$ consisting of trace-class operators, see, \textit{e.g.}, \cite[Section $4.4$]{Moretti_2018_spectral_theory_quantum_mechanics}.
    If $T\in\mathcal L_1(\mathcal H)$, then its trace is denoted by $\Tr[T]\in\R$.
\end{notation}

\begin{notation}
\label{NOT:weak_topologies}
    Assume that $\mathcal X$ is a Banach space. 
    We denote by $\mathcal X_w$ the topological space $(\mathcal X, \tau_{\mathcal X}^w)$, where the weak topology $\tau_{\mathcal X}^w$ is the smallest topology on $\mathcal X$ with respect to which every linear $f:\mathcal X\to \R$ is continuous. 
    If $\tau_{\mathcal X}^s$ denotes the strong topology on $\mathcal X$ (\textit{i.e.}, the natural topology induced by the norm), we will, with a slight abuse of notation, simply denote the topological space $(\mathcal{X}, \tau_{\mathcal X}^s)$ by $\mathcal{X}$.
    We denote by $\mathcal X_{bw}$ the topological space $(\mathcal X,\tau_{\mathcal X}^{bw})$, where the bounded weak topology $\tau_{\mathcal X}^{bw}$ is the largest topology on $\mathcal X$ that coincides with the weak topology on norm-bounded sets, \textit{i.e.} $C\subset \mathcal X_{bw}$ is closed if and only if $C\cap \bar B^\mathcal X$ is closed in $\mathcal X_w$ for every closed ball $\bar B^\mathcal X\subset \mathcal X$. 
    We refer to \cite[Section II.$5$]{Day_1973_Normed_linear_spaces_3rd_edition} for details about the bounded weak topology.
\end{notation}

\begin{notation}
\label{NOT:Fourier_trasnform}
    Assume that $\big(\mathcal H, \lan\,\cdot\,,\,\cdot\,\ran^{}_{\mathcal H}\big)$ is a separable real Hilbert space. 
    If $d\in\N$, we denote the Fourier transform of a Bochner-integrable function $f:\R^d\to \mathcal H$ by:
    \[
        \mathscr F[f]:\R^d\ni \xi\mapsto \mathscr F [f](\xi):=
        (2\pi)^{-d/2}
        \int_{\R} e^{-i\xi\cdot x}f(x)\d x\in\mathcal H.
    \]  
\end{notation}

We recall the definition of the classical Hilbert-valued fractional Sobolev spaces. 
See \cite[Chapter $3$, Section $2.3$]{Temam_2001_Navier-Stokes_equations}.

\begin{notation}
\label{NOT:ordinary_Sobolev_spaces}
    For $s\in\R$, we set:
    \[
        H^s(\R^d;\mathcal H):=\Big\{f\in \mathcal S'(\R^d;\mathcal H) \ : \ (1+|x|^{2})^{s/2}\mathscr F[f](x)\in L^2(\R^d;\mathcal H)\Big\},
    \]
    where $|x|$, for $x\in\R^d$, denotes the standard Euclidean norm in $\R^d$, while $(1+|x|^{2})^{s/2}\mathscr{F}[u](x)$ is the $\mathcal H$-valued distribution given by the product between the smooth function $\R^d\ni x \mapsto (1+|x|^{2})^{s/2}\in\R$ and the distribution $\mathscr F[u]\in\mathcal S'(\R^d;\mathcal H)$.
    $H^s(\R^d;\mathcal H)$ is a Hilbert space if endowed with the inner product:
    \begin{gather*}
        \int_{\R^d}(1+|x|^{2})^{s}\big\lan \mathscr F[f](x), \mathscr F[g](x)\big\ran^{}_{\mathcal H}\d x,
        \qquad \forall\, f,g\in H^s(\R^d;\mathcal H),
    \end{gather*}
    and with the induced norm $\|\, \cdot \, \|^{}_{H^s(\R^d;{\mathcal H})}$.
\end{notation}  

\subsection{Functional setting for fluid dynamics equations}

A variant of the classical Sobolev spaces, tailored to deal with the equations of incompressible fluid dynamics, is given in the following definition.

\begin{definition}
\label{DEF:divergencefree_Sobolev_spaces}
    For $s\in \R$, we define the Sobolev space of divergence-free vector fields as follows:
    \[
        H^s:=\big\{u\in H^s(\R^2;\R^2)\ : \ \div u=0\big\}.
    \]
    We recall that, for $u\in \mathcal S'(\R^2;\R^2)$, the distribution $\div u\in\mathcal S'(\R^2;\R^2)$ is defined as:
    \[
        \div u:= i \mathscr F^{-1}\big[x \cdot  \mathscr{F}[u](x) \big], 
        \qquad 
        \textit{for }x\in\R^2,
    \]
    where $\cdot$ is the euclidean inner product in $\R^2$ and $x\cdot \mathscr{F}[u](x)$ is the real-valued distribution given by the euclidean product in $\R^2$ between the smooth function $\R^2\ni x \mapsto x\in\R^2$ and the distribution $\mathscr F[u]\in\mathcal S'(\R^2;\R^2)$.
    The linear space $H^s$ is endowed with the inner product:
    \begin{gather*}
        \lan u,v\ran^{}_{\!H^s}:=\int_{\R^2}(1+|x|^2)^{s}\mathscr F[u](x)\cdot\mathscr F[v](x)\d x,
        \qquad \forall\, u,v\in H^s,
    \end{gather*}
    which makes $H^s$ a separable Hilbert space.
    The induced norm is denoted as $\|\,\cdot\,\|^{}_{H^s}$. 
\end{definition}

\begin{notation}
    If $s>0$, the topological  dual space of $H^s$ is $H^{-s}$. 
    The duality product is denoted by:
    \[
        {\big\lan}u,v\big\ran^{}_{\ss{\!H^{-s}\!\times\!H^s}}
        =\int_{\R^2}\mathscr F[u](x)\cdot \mathscr F[v](x)\d x, \qquad \forall\, u\in H^{-s}, \ v\in H^s.
    \]
    If $s=0$ we use the following notation:
    \begin{gather}
        H:=H^0,\\
        \|u\|:=\|u\|^{}_{H^0}=\bigg(\int_{\R^2}|u(x)|^2\d x\bigg)^{1/2}, \qquad \forall\, u\in H,\\
        \lan u,v\ran :=\lan u,v\ran^{}_{\!H^0}= \int_{\R^2}u(x)\cdot v(x)\d x, \qquad \forall\, u,v\in H,
    \end{gather}
     where the equalities hold thanks to the properties \eqref{EQ:unitary_fourier_1} and \eqref{EQ:unitary_fourier_2}.
\end{notation}

\begin{remark}
\label{REM:duality=inner_product}
    If $u\in H$ and $v\in H^s$ for some $s>0$, then, by direct inspection, since $H^{s}\hookrightarrow H\hookrightarrow H^{-s}$,
    \[
        \lan u,v\ran={\big\lan}u,v\big\ran^{}_{\ss{\!H^{-s}\!\times\!H^s}}.
    \]
\end{remark}

\begin{definition}
\label{DEF:leray_projection}
    The Hilbert space $H$ is a closed subspace of $H^0(\R^2;\R^2)=L^2(\R^2;\R^2)$, thus an orthogonal projection
    \[
        \Pi:L^2(\R^2;\R^2)\to H,
    \]
    called Leray projector, is well-defined, see, \textit{e.g.}, \cite[Proposition $5.3$, Corollary $5.4$]{brezis_2010_functional_analysis_Sobolev}.
    It is known that the operator $\Pi$ can be extended in such a way that, for any $s\in\R$, $\Pi\big(H^s(\R^2;\R^2)\big)=H^s$.
\end{definition}

We derive the following trivial lemma from the definition of the divergence-free Sobolev spaces and the H\"older inequality.

\begin{lemma}[Interpolation inequality]
\label{LEM:sobolev_interpolation_inequality} 
    Assume that $p,q\in\R$ with $p<q$ and $\la\in(0,1)$.
    If $u\in H^q$, then 
    \begin{equation}
        \|u\|^{}_{H^r}\leq \|u\|^{\la}_{H^p}\|u\|^{1-\la}_{H^q}, \qquad   r:=\la p + (1-\la)q.
     \end{equation}
\end{lemma}

The Laplacian on the Sobolev spaces of divergence-free vector fields is a paramount operator in fluid dynamics that has been extensively studied in the literature. 
We give here the definition of its powers, that will allow us to study the hyperviscous Navier-Stokes equations.

\begin{definition}
\label{DEF:stokes_operator}
    For $\a\geq 1$, we define:
    \[
        A^\a u:=\mathscr F^{-1}\big[(1+|x|^2)^\a\mathscr F[u](x)\big], \qquad \forall\, u\in \mathcal S'(\R^2;\R^2),
    \]
    where $(1+|x|^2)^\a\mathscr F[u](x)$ is the $\R^2$-valued distribution given by the product between the smooth function $\R^2\ni x\mapsto (1+|x|^2)^\a\in\R$ and the distribution $\mathscr F[u]\in\mathcal S'(\R^2;\R^2)$.
    If $s\in\R$, the operator ${A:=A^1:H^{2+s}\to H^s}$ is commonly called the Stokes operator.
    Whereas, for $\a>1$, the operator $A^\a:H^{2\a+s}\to H^s$, for $s\in\R$, will be referred to as the hyperviscous Stokes operator.
\end{definition}

In the next lemma, we present the main properties of the (hyperviscous) Stokes operator.
These are mainly classical, so we postpone the proof to Appendix~\ref{APP:aux}.

\begin{lemma}
\label{LEM:stokes_operator_semigroup}
    Let $s\in\R$ and $\a\ge 1$. 
    The linear operator $A^\a:H^{2\a+s}\to H^s$ is well-defined, bijective and bounded.
    Moreover, for every $u,v\in H^{2\a+s}$,
    \begin{equation}
    \label{EQ:LEM:stokes_operator_semigroup:graph_norm_Aa}
        \|A^\a u\|_{H^s}=\|u\|_{H^{2\a+s}},\qquad 
        \lan u,v\ran_{H^{2\a+s}}
        =
        \lan A^\a u,A^\a v\ran_{H^s}.
    \end{equation}
    If $\a=1$, then for every $u\in H^{2+s}$,
    \begin{equation}
    \label{EQ:LEM:stokes_operator_semigroup:A_Delta}
        Au=u-\Delta u
        \qquad\text{in }H^s.
    \end{equation}
    Letting $D_{H^s}(A^\a)$ denote the dense subspace $\big(H^{2\a+s},\|\,\cdot\,\|^{}_{H^s}\big)$ of $H^s$, then, for $\nu>0$, the operator $-\nu A^\a:D_{H^s}(A^\a)\subset H^s\to H^s$ is negative definite, self-adjoint, has spectrum $(-\infty,-\nu]$, and generates a strongly continuous contraction analytic semigroup $\{e^{-\nu tA^\a}\}_{t\ge0}$ on $H^s$.
    In particular, for every $\theta\ge0$, there exists a constant
    $C=C(\nu,\a,\theta)>0$ such that
    \begin{equation}
    \label{EQ:LEM:stokes_operator_semigroup:smoothing_estimate}
        \|A^\theta e^{-\nu tA^\a}x\|_{H^s}
        \le C\, t^{-\theta/\a}\|x\|_{H^s},
        \qquad \forall\,x\in H^s,\ \forall\,t>0.
    \end{equation}
\end{lemma}

Eventually, one needs to formally introduce the nonlinear term that appears in both the Navier-Stokes and the Euler equations.
The proof can be found in Appendix~\ref{APP:aux}.

\begin{lemma}
\label{LEM:b}
    Assume that $p, q, r\geq 0$ satisfy the assumption:
    \begin{equation}
        \begin{cases}
            p+q+r\geq 1, \quad &\textit{if }p, q, r\neq 1,\\
            p+q+r>1, &\textit{otherwise},
        \end{cases}
\end{equation}
    then the trilinear form
    \begin{equation}
        b:H^p\times H^{q+1}\times H^{r}\ni (u,v,w)\mapsto
        \int_{\R^2}
        [(u\cdot \nabla)v]
        \cdot w\, \d \mathscr{L}^2
        =\sum_{i,j=1}^2\int_{\R^2}w_j\,u_i\,\partial_iv_j\, \d \mathscr L^2\in\R,
    \end{equation}
    is well-defined and continuous, where we denoted by $\mathscr L^2$ the $2$-dimensional Lebesgue measure on $\R^2$.
    Moreover, whenever the expressions make sense, we have:
    \begin{align}
    \label{EQ:b_antysimmetry_1}
        &b(u,v,w)=-b(u,w,v),\\
    \label{EQ:b_antysimmetry_2}
        &b(u,v,v)=0, \\ 
    \label{EQ:b(x,x,Ay)=0}
        &b(u,u,Au)=0.
    \end{align}
\end{lemma}

\begin{definition}
\label{DEF:B}
    Let $p,q,r$ satisfy the assumption of Lemma~\ref{LEM:b}, then we define the bilinear continuous operator
    \[
        B:H^p\times H^{q+1}\ni(u,v)\mapsto \Pi
        [(u\cdot\nabla)v]
        =\Pi\div(u\otimes v)\in H^{-r}.
    \]
    With a slight abuse of notation, by $B$ we will also denote the corresponding quadratic map
    \[
        B:H^{p\vee (q+1)}\ni u\mapsto B(u):=B(u,u)\in H^{-r},
    \]
    called Navier-Stokes nonlinearity.
\end{definition}
In the following lemma we summarise some frequently used estimates on the operators $b$ and $B$.
The proof is in Appendix~\ref{APP:aux}.
\begin{lemma}
\label{LEM:estimates_B}
    Assume that $\s>2$ and $\e>0$. 
    There exists a finite constant $c>0$ such that:
    \begin{align}
    \label{EQ:estim_B_1}
        &\|B(u,v)\|^{}_{H^{\s-1}}\leq c\|u\|^{}_{H^{\s-1}}\|v\|^{}_{H^{\s}}, && \forall\, u\in H^{\s-1}, \ v\in H^{\s},\\
    \label{EQ:estim_B_2}
        &\|B(u,v)\|\leq c\|u\|^{}_{H^{1}}\|v\|^{}_{H^{1+\e}}, && \forall\, u\in H^{1}, \ v\in H^{1+\e},\\
    \label{EQ:estim_B_3}
         &\|B(u,v)\|^{}_{H^{-\e}}\leq c\|u\|^{}_{H^{1}}\|v\|^{}_{H^{1}}, && \forall\, u\in H^{1}, \ v\in H^{1},\\
    \label{EQ:estim_B_4}
         &|b(u,v,w)|\leq c\|u\|^{}_{H^{1}}\|v\|^{}_{H^{1}}\|w\|^{}_{H^\e}, && \forall\, u\in H^{1}, \ v\in H^{1}, \ w\in H^{\e}.
    \end{align}
In particular, the quadratic maps $B:H^{\s}\to H^{\s-1}$, $B:H^{\s-1}\to H^{2-\s}$ are continuous.
\end{lemma}

\subsection{The 2D incompressible Euler Equation}

We project the Euler equations \eqref{EQ:proto_EE} on the space $H$ of divergence-free square-integrable vector fields, by means of the Leray operator.
This allows us to incorporate the incompressibility condition in the functional setting and to cancel out the gradient of the pressure. 
The equations thus reduce to the following ($2$D incompressible) Euler Equation:
\begin{equation}
    \label{EQ:euler}
            u'+B(u)=0, \qquad \textit{on }(0,+\infty).
\end{equation}  
We give the definition of a solution to this equation and recall an existence and uniqueness result from the literature.
\begin{definition}
\label{DEF:euler_solution}
    Assume that $\s>2$.
    We say that a function $u\in C(\R_+;H^\s)$ is a solution to the Euler Equation~\eqref{EQ:euler} if in $H^{\s}$:
    \begin{equation}
    \label{EQ:integral_euler}
        u(t)+\int_0^tB\big(u(s)\big)\d s =u(0), \qquad \forall\, t\geq 0.
    \end{equation}
\end{definition}
\begin{remark}
    If $t>0$ and $u\in C(\R_+;H^\s)$, then the integral in equation~\eqref{EQ:integral_euler} converges in $H^{\s-1}$ by Lemma~\ref{LEM:estimates_B}, hence the equation~\eqref{EQ:integral_euler} holds in $H^{\s-1}$. 
    However, since both $u(t)$ and $u(0)$ belong to $H^\s$, the equation defines the integral in $H^\s$ and the equality is thus true in $H^{\s}$.
\end{remark}

\begin{theorem}[\text{\cite[Theorem III]{Kato+Ponce_1986_well_posedness_Euler_N-S}, \cite[Theorem II]{Kato+Ponce_1987_nonstationary_flows_viscous_ideal_fluids}}]
\label{TH:solution_euler}
    Assume that $\s>2$. 
    If $x\in H^{\s}$, then there exists a unique solution $u^x$ to the Euler Equation~\eqref{EQ:euler} such that $u^x(0)=x$.
    Moreover, the Euler flow map
    \[
        \Phi:\R_+\times H^\s\ni (t,x)\mapsto \Phi(t,x):=u^x(t)\in H^\s
    \]
    is continuous.
\end{theorem}

We aim to prove the existence of an invariant measure for the Euler Equation, whose definition is given below.

\begin{definition}
\label{DEF:inv_meas_euler}
    Assume that $\s>2$. 
    We say that a probability measure $\mu\in\mathscr P(H^\s)$ is an invariant measure for the Euler Equation~\eqref{EQ:euler} if:
    \begin{equation}
    \label{EQ:def_invariant_measure_euler}
        (\Phi_t)_\ast\mu=\mu\in\mathscr P(H^\s), \qquad \forall\, t\geq 0,
    \end{equation}
    where, for $t\geq 0$, $\Phi_t: H^\s\ni x\mapsto \Phi(t,x)\in H^\s$, and $\Phi$ has been introduced in Theorem~\ref{TH:solution_euler}.
\end{definition}

\subsection{The Ornstein-Uhlenbeck process}

As explained in the introduction, the invariant measure for the Euler Equation~\eqref{EQ:euler} will be obtained as the law of marginally stationary solutions, which are constructed as the limit of marginally stationary solutions for appropriate Navier-Stokes equations, as the kinematic viscosity tends to $0$.
In order to obtain good estimates for the invariant measures of the Navier-Stokes equation, \textit{i.e.}, uniform with respect to the kinematic viscosity, we add a white noise to the equation that vanishes in the inviscid limit together with the diffusion term.

A usual tool to study the solution to stochastic partial differential equations with additive noise exploits the Ornstein-Uhlenbeck process, as briefly explained in Remark~\ref{REM:additive_noise}, see \cite[Chapter $5$]{DaPrato+Zabczyk_2014_stochastic_equations_infinite} for a complete dissertation.
Therefore, we now consider the equation that defines this stochastic process, adapted to our setting, and recall its main properties. 
All the assertions follow from the general theory for linear equations with additive noise from \cite[Chapter $5$]{DaPrato+Zabczyk_2014_stochastic_equations_infinite}, hence the proof is postponed to Appendix~\ref{APP:aux}.

\begin{theorem}
\label{TH:z}
    Assume that $\b\geq 0$ and let $(\Omega, \F, \{\F_t\}_{t\geq 0}, \P)$ be an augmented filtered probability space with an adapted {$H^\b$-valued Wiener process} $W$. 
    Assume also that $\a>1$ and $\nu>0$. 
    Then, there exists a pathwise unique $H^\b$-valued predictable process $Z$ with regularity 
    \begin{align}
    \label{EQ:regularity_Z}
        &Z\in L^2\big(\Omega\times[0,T];H^{\b+\a}\big) \cap L^2\big(\Omega;C\big([0,T];H^{\b}\big)\big),
        \qquad \forall\, T>0,
    \end{align}
    that satisfies the weak formulation for the equation 
    \begin{equation}
    \label{EQ:SPDE_Z}
        \begin{cases}
            \d Z_t +\nu A^{\a}Z_t\d t =\sqrt\nu  \d W_t, \qquad \textit{for }t>0,\\
             Z_0=0,
        \end{cases}
    \end{equation}
    \textit{i.e.} if $\phi\in H^{\b+2\a}$, then $\P-a.s.$
    \begin{equation}
    \label{EQ:weakz}
        \lan Z_t, \phi\ran^{}_{\! H^\b} +\nu\int_0^t\lan Z_s, A^{\a}\phi\ran^{}_{\! H^\b}\d s =\sqrt\nu\lan W_t,\phi\ran^{}_{\! H^\b}, \qquad \forall\, t\geq 0.
    \end{equation}
    Furthermore, $Z$ satisfies the mild formulation for equation~\eqref{EQ:SPDE_Z}, \textit{i.e.} $\P-a.s.$ in $H^{\b}$
    \begin{equation}
    \label{EQ:mild_Z}
        Z_t=\sqrt \nu\int_0^te^{-\nu (t-s)A^{\a}}\d W_s,\qquad \forall\,t\geq0,
    \end{equation}
    and the strong formulation, \textit{i.e.} $\P-a.s.$ in $H^{\b-\a}$
    \begin{equation}
    \label{EQ:strong_Z}
        Z_t+\nu\int_0^tA^{\a}Z_s \d s = \sqrt \nu \, W_t,\qquad \forall \, t\geq 0,
    \end{equation}
    where a continuous modification for both integral processes in equations \eqref{EQ:mild_Z}, \eqref{EQ:strong_Z} is employed.
\end{theorem}

\section{The stochastic hyperviscous  Navier-Stokes Equation}
\label{SEC:stoch_hyp_2D_NSE}

This section is devoted to the study of the $2$D hyperviscous Navier-Stokes equations, projected onto the space of square-integrable divergence-free vector fields, with an additive white noise, and coupled with an initial condition that can be either deterministic or random.
First, we give the definition of solution for the equation and specify its uniqueness property.
Next, we construct a solution that satisfies these properties and obtain some further regularities.

\begin{definition}
\label{DEF:solution_NSnu}
    Assume that $\nu>0$ and $\a>1$.
    We say that the Stochastic Hyperviscous Navier-Stokes Equation with kinematic viscosity $\nu$ and hyperviscous power $\a$ (henceforth contracted SHNS$_{\nu,\a}$ Equation)
    \begin{equation}
    \label{EQ:NSnu}
            \d X_t+\nu A^{\a}X_t\d t +B(X_t)\d t=\sqrt \nu \d W_t, \qquad \textit{for }t>0,
    \end{equation}
    has a solution if there exists an augmented filtered probability space $(\Omega,\F,\{\F_t\}_{t\geq 0},\P)$ with an adapted $H^{2\a}$-valued Wiener process $W$, and a predictable process $X:\R_+\times\Omega\to H^{2\a}$ such that:
    \begin{itemize}
        \item $X$ has $\P-a.s.$ trajectories in $C(\R_+;H^{2\a})\cap L^2_{loc}(\R_+;H^{3\a})$
        \item The following identity holds $\P-a.s.$ in $H^{\a}$
        \begin{equation}
        \label{EQ:def_solution_NS}
            X_t+\nu\int_0^tA^{\a}X_s\d s +\int_0^tB(X_s)\d s =X_0+\sqrt \nu \, W_t, \qquad \forall\, t\geq 0.
        \end{equation}
    \end{itemize}
    In this case, we say that $(\Omega,\F,\{\F_t\}_{t\geq 0},\P;W;X)$ is a solution to the SHNS$_{\nu,\a}$ Equation~\eqref{EQ:NSnu}.
\end{definition}

The definition of solution for the stochastic hyperviscous Navier-Stokes equation is given in a weak probabilistic sense, \textit{i.e.} with a non-fixed filtered probability space or Wiener process, in anticipation of Sections~\ref{SEC:markov_property} and \ref{SEC:stationary_solutions}, where we will need to work with different probability spaces.

\begin{definition}
\label{DEF:pathwise_uniqueness_NS}
    Assume that $\nu>0$ and $\a>1$.
    We say that the solutions to the SHNS$_{\nu,\a}$ Equation~\eqref{EQ:NSnu} are pathwise unique if, given two solutions $(\Omega,\F,\{\F_t\}_{t\geq 0},\P;W;X^i)$, $i=1,2$, such that $\P(X^1_0=X^2_0)=1$, then 
    \[
        \P\Bigg(\bigcap_{t\geq 0}\big\{X^1_t= X^2_t\big\}\Bigg)=1.
    \]
\end{definition}
\begin{definition}
\label{DEF:uniqueness_law_NS}
    Assume that $\nu>0$ and $\a>1$.
    We say that the solutions to the SHNS$_{\nu,\a}$ Equation~\eqref{EQ:NSnu} are unique in law if, given two solutions $(\Omega^i,\F^i,\{\F^i_t\}_{t\geq 0},\P^i;W^i;X^i)$, $i=1,2$, such that ${(X^1_0)_\ast\P^1=(X^2_0)_\ast\P^2}$ on $(H^{2\a},\mathscr B_{H^{2\a}})$, then for all $T>0$:
    \[
        (X^1)_\ast\P^1=(X^2)_\ast\P^2\in\mathscr P\big(C([0,T];H^{2\a})\big).
    \]
\end{definition}

\begin{remark}
\label{REM:additive_noise}
    Fix $\nu>0$ and $\a>1$.
    In order to study the solvability of the SHNS$_{\nu,\a}$ Equation~\eqref{EQ:NSnu}, we formally differentiate in time the process $V:=X-Z$, where $Z$ is defined in Theorem~\ref{TH:z} as the solution of the stochastic equation~\eqref{EQ:SPDE_Z}.
    We have, without any claim of rigour:
    \begin{align}
        \d V_t
        &=\d X_t-\d Z_t\\
        &=-\nu A^{\a}X_t\d t -B(X_t)\d t +\sqrt \nu \d W_t-\big(\nu A^{\a}Z_t\d t +\sqrt \nu \d W_t\big)\\
        &=-\nu A^{\a}V_t\d t -B(V_t+Z_t)\d t,\\
        V_0&=X_0-Z_0=X_0.
    \end{align}
    In this new equation, the stochastic term is reduced to the term $Z$ inside the nonlinearity and does not explicitly appear in the dynamics for $V$.
    Therefore, we are led to consider the following deterministic equation, for some $x\in H^{2\a}$:
    \begin{equation}
    \label{EQ:v_nu_weak}
        \begin{cases}
            v'+\nu A^{\a}v+B(v+z)=0, \qquad \textit{in }(0,+\infty),\\
            v(0)=x,
        \end{cases}
    \end{equation}
    where $z$ is a deterministic function which plays the role of a trajectory of the Ornstein-Uhlenbeck process, thus having its regularity. 
    Once the existence and uniqueness for the solution to this deterministic equation will have been proved, a map $\mathcal V:z\mapsto v$ will be properly defined, and the process $X$ will be recovered as $X:=\mathcal V(Z)+Z$.
\end{remark}

We now state and prove the main theorem that concerns the existence, uniqueness and dependence on parameters, for the deterministic problem~\eqref{EQ:v_nu_weak}.

\begin{theorem}
\label{TH:exisuniqv}
    Assume that $\nu>0$, $\a>1$, $x\in H^{2\a}$, and $z\in C(\R_+;H^{2\a})$.
    Then, there exists a unique function 
    \[
        v\in C^1(\R_+;H)\cap H^1_{loc}(\R_+;H^\a)\cap C(\R_+;H^{2\a})\cap L^2_{loc}(\R_+;H^{3\a}),
    \]
    that satisfies the problem~\eqref{EQ:v_nu_weak} in the mild sense in $H^{2\a}$,
    or strong sense in $H^{\a}$, see Lemma~\ref{LEM:v_mild_to_strong}. 
    Furthermore, the map
\begin{equation}
    \mathcal V_{\nu,\a}=\mathcal V:H^{2\a} \times  C(\R_+;H^{2\a}) \ni (x, z) \mapsto v \in C(\R_+;H^{2\a})\cap L^2_{loc}(\R_+;H^{3\a})
\end{equation}
is such that $\mathcal V(0,0)=0$ and satisfies the following estimate. 
For every $T>0$ and every $R>0$ there exists $C=C(T,R,\nu,\a)>0$ such that if $\|x_i \|^{}_{H^{2\a}} \leq R$ and $\sup_{t\in[0,T]}\|z_i(t)\|^{}_{H^{2\a}} \leq R$, $i=1,2$, then 
\begin{equation}
\label{EQ2:boundedness_continuity_v(x,z)}
\begin{gathered}
    \sup_{t\in[0,T]}\|\mathcal V(x_1,z_1)(t)-\mathcal V(x_2,z_2)(t)\|_{H^{2\a}}^2+\int_0^T\|\mathcal V(x_1,z_1)(t)-\mathcal V(x_2,z_2)(t)\|_{H^{3\a}}^2\d t \\
    \leq C\Big(\|x_1-x_2\|^2_{H^{2\a}}+\sup_{t\in[0,T]}\|z_1(t)-z_2(t)\|_{H^{2\a}}^2\Big).
\end{gathered}
\end{equation}
\end{theorem}
\begin{proof}
    \textbf{\textit{Step $\mathbf{1}$.}} 
    Let us fix $\nu>0$ and $\a>1$. 
    We introduce
    \begin{equation}
    \label{EQ2:defGamma}
    \Gamma:H^{2\a}\times L^2(\R_+;H^{\a})\to L^2(\R_+;H^{3\a})\cap C_b(\R_+;H^{2\a}),
    \end{equation}
    such that, for $x\in H^{2\a}$ and $f\in L^2(\R_+;H^\a)$,
    \[
        \Gamma(x,f)(t):=e^{-\nu t A^{\a}}x-\int_0^{t} e^{-\nu (t-s)A^{\a}}f(s) \d s, \qquad \forall\, t\geq 0.
    \]
    We prove that $\Gamma$ is a well-defined bilinear operator and that it satisfies, for a constant $C_\nu>0$ and  for all $(x,f)\in H^{2\a}\times L^2(\R_+;H^{\a})$:
    \begin{equation}
    \label{EQ:estim_Gamma}
        \int_{0}^{+\infty}\|\Gamma (x,f)(t)\|^{2}_{H^{3\a}}\d t +\sup_{t\geq 0}\|\Gamma (x,f)(t)\|^{2}_{H^{2\a}}\leq 
        C_\nu\|x\|^{2}_{H^{2\a}}+
        C_\nu\int_{0}^{+\infty}\|f(t)\|^{2}_{H^{\a}}\d t .
        \end{equation}
    This directly implies that, for any $x\in H^{2\a}$ and $T>0$, the restriction
    \begin{equation}
    \label{EQ2:def_Gamma_xT}
        \Gamma_T^x:L^2(0,T;H^{\a})\ni f  \mapsto \Gamma^x_Tf:=\big(\Gamma(x,\tilde f\,)\big)\Big|_{[0,T]}\in C\big([0,T];H^{2\a}\big)\cap L^2(0,T;H^{3\a})
    \end{equation}
    is linear and bounded, where $\tilde f\in L^2(\R_+;H^\a)$ is the trivial extension of $f\in L^2(0,T;H^\a)$, \textit{i.e.} $\tilde f(t)=f(t)$ for $t\in[0,T]$ and  $\tilde f(t)=0$ for $t>T$.

    To this end, let us fix $x\in H^{2\a}$ and $f\in L^2(\R_+;H^\a)$. 
    For $a.e.\ t\geq 0$, we calculate the Fourier transform of $\Gamma (x,f)(t)$,
    denoting it by $\hat u(t,\,\cdot\,)$,  then for all $\xi\in\R^2$ we have:
    \begin{align}
        \hat u (t,\xi)&=e^{-\nu t (1+|\xi|^2)^{\a}}\mathscr F[x](\xi)-\int_0^te^{-\nu(t-s)(1+|\xi|^2)^{\a}}\mathscr F[f(s)](\xi)
        \d s \\
        &=:K_{\xi}(t)\mathscr F[x](\xi)-\Big( (\1_{\R_+}K_\xi)\ast \big(\1_{\R_+}
        \mathscr F[f(\,\cdot\,)](\xi)
        \big)\Big)(t).
    \end{align}
    In the last line, we defined $K_\xi:\R_+\ni t\mapsto e^{-\nu t (1+|\xi|^2)^{\a}}\in\R$ and rewrote the second integral as a convolution.
    We now use a straightforward generalisation of the Young convolution inequality, see \cite[Theorem $3.9.4$]{Bogachev_2007_measure_theory_1_2}: for measurable functions $g:\R\to\R$ and $h:\R\to \R^2$,
    \[
        \|g\ast h\|^{}_{L^p(\R;\R^2)}\leq \|g\|^{}_{L^q(\R)}\|h\|^{}_{L^r(\R;\R^2)},\qquad \forall \, p,q,r\in[1,+\infty] \ \textit{ s.t. } \ 1+\frac1p=\frac1q+\frac1r.
    \]
    We employ this estimate with the choice of parameters $p=r=2$, $q=1$, or $p=\infty$, $q=r=2$. 
    We obtain for all $\xi\in\R^2$:
    \begin{equation}
    \label{EQ:fourier_Gamma_1}
    \begin{aligned}
        \|
        \hat u (\,\cdot\,,\xi)\|^{2}_{L^2(\R_+;\R^2)}
        &\leq 2\int_0^{+\infty}e^{-2\nu t(1+|\xi|^2)^{\a}}\big|\mathscr F[x](\xi)\big|^2\d t +2\big\|\1_{\R_+}K_\xi\big\|^2_{L^1(\R)}\big\|\1_{\R_+}\mathscr F[f(\,\cdot\,)](\xi)\big\|^2_{L^2(\R;\R^2)}\\
        &= \frac{1}{\nu}(1+|\xi|^2)^{-\a}\big|\mathscr F[x](\xi)\big|^2+2\|K_\xi\|^{2}_{L^1(\R_+)}\big\|
        \mathscr F[f(\,\cdot\,)](\xi)
        \big\|^{2}_{L^2(\R_+;\R^2)},
    \end{aligned}
    \end{equation}
    and similarly:
    \begin{align}
    \label{EQ:fourier_Gamma_2}
        \|\hat u (\,\cdot\,,\xi)\|^{2}_{L^\infty(\R_+;\R^2)}
        \leq 2\big|\mathscr F[x](\xi)\big|^2+2\|K_{\xi}\|^{2}_{L^2(\R_+)}\big\|
        \mathscr F[f(\,\cdot\,)](\xi)\big\|^{2}_{L^2(\R_+;\R^2)}.
    \end{align}
    In particular, we have:
    \begin{align}
        \|K_\xi\|^{2}_{L^1(\R_+)}
        &=\bigg(\int_0^{+\infty}e^{-\nu r(1+|\xi|^2)^{\a}}\d r\bigg)^2 =\frac{1}{\nu^2}{(1+|\xi|^2)^{-2\a}},\\
        \|K_\xi\|^{2}_{L^2(\R_+)}
        &=\int_0^{+\infty}e^{-2\nu r(1+|\xi|^2)^{\a}}\d r =\frac{1}{2\nu}{(1+|\xi|^2)^{-\a}}.
    \end{align}
    By the Fubini Theorem and the last calculations, we infer:
    \begin{align}
        &\quad \ \|\Gamma (x,f)\|^{2}_{L^2(\R_+;H^{3\a})}+\|\Gamma(x,f)\|^{2}_{L^\infty(\R_+;H^{2\a})}\\
        &=\int_0^{+\infty} \int_{\R^2}(1+|\xi|^2)^{3\a}|\hat u (t,\xi)|^2\d \xi\d t
        +\esssup_{t\geq 0}\int_{\R^2}(1+|\xi|^2)^{2\a}|\hat u (t,\xi)|^2\d \xi\\
        &= \int_{\R^2}(1+|\xi|^2)^{3\a}\|\hat u (\,\cdot\,,\xi)\|^2_{L^2(\R_+;\R^2)}\d \xi
        +\int_{\R^2}(1+|\xi|^2)^{2\a}\|\hat u (\,\cdot\,,\xi)\|^2_{L^\infty(\R_+;\R^2)}\d \xi\\
        &\leq \bigg(2+\frac{1}{\nu}\bigg)\int_{\R^2}(1+|\xi|^2)^{2\a}\big|\mathscr F[x](\xi)\big|^2\d \xi
        +\bigg(\frac1\nu+\frac{2}{\nu^2}\bigg)\int_{\R^2}(1+|\xi|^2)^\a\big\|
        \mathscr F[f(\,\cdot\,)](\xi)
        \big\|^2_{L^2(\R_+;\R^2)}\d \xi\\
        &= \bigg({2+\frac{1}{\nu}}\bigg)\|x\|^{2}_{H^{2\a}}+\frac1\nu\bigg(1+{\frac2\nu}\bigg)\|f\|^{2}_{L^2(\R_+;H^{\a})}.
    \end{align}
    This is the sought estimate~\eqref{EQ:estim_Gamma}, once proved that $\Gamma (x,f)$ is continuous in time.
    
    As for the continuity of $\Gamma (x,f)$, first observe that $\hat u (\,\cdot\,,\xi)$ is continuous in time for $a.e.\ \xi\in\R^2$, then take a sequence of times $\{t_n\}_{n\in\N}$ convergent to a fixed $t\geq 0$. The above estimate allows to apply the Dominated Convergence Theorem and conclude: 
    \[
        \lim_{n\to\infty}\|\Gamma (x,f)(t_n)-\Gamma (x,f)(t)\|^2_{H^{2\a}} =\lim_{n\to\infty}\int_{\R^2}(1+|\xi|^2)^{2\a}\big|\hat u (t_n,\xi)-\hat u (t,\xi)\big|^2\d \xi=0.
    \]

\textbf{\textit{Step $\mathbf{2}$.}}
    For fixed $T>0$, we introduce the Banach space $\big(\mathcal Y_T,\|\,\cdot\,\|_{\mathcal Y_T}^{}\big)$ given by:
    \begin{equation}
    \label{EQ:def_Y_T}
    \begin{aligned}
        \mathcal Y_T&:= L^2(0,T;H^{3\a}) \cap  C\big([0,T];H^{2\a}\big),\\ 
        \|v\|^2_{\mathcal Y_T}&:= \|v\|^2_{L^2(0,T;H^{3\a})} + \|v\|^2_{C([0,T];H^{2\a})}, \qquad \forall\, v\in \mathcal Y_T.
    \end{aligned}
    \end{equation}
    Fix $z\in C(\R_+;H^{2\a})$. 
    We prove that the function 
    \[
        \Lambda_T:\mathcal Y_T\ni v\mapsto \Lambda_T(v):=\Gamma^x_T \big(B(v+z)\big)\in \mathcal Y_T
    \]
    is a well-defined contraction on bounded sets of $\mathcal Y_T$. 

    First of all, we recall that  for sufficiently regular $x_1,x_2,y\in H$:
    \[
        B(x_1+y)-B(x_2+y)=B(x_1-x_2,x_1+y)+B(x_2+y,x_1-x_2).
    \]
    Therefore, thanks to the estimate~\eqref{EQ:estim_B_1} in Lemma~\ref{LEM:estimates_B} (with $\s=\a+1>2$), and to the Sobolev embeddings, we have for $x_1,x_2,y\in H^{2\a}$:
    \begin{align}
        \|B(x_1+y)-B(x_2+y)\|^{}_{H^\a}
        &\leq \|B(x_1-x_2,x_1+y)\|^{}_{H^{\a}}+\|B(x_2+y,x_1-x_2)\|^{}_{H^{\a}}\\
        &\leq c\|x_1-x_2\|^{}_{H^{\a}}\|x_1+y\|^{}_{H^{\a+1}}+c\|x_2+y\|^{}_{H^{\a}}\|x_1-x_2\|^{}_{H^{\a+1}}\\
        &\leq c\|x_1-x_2\|^{}_{H^{2\a}}\Big[\|x_1\|^{}_{H^{2\a}}+\|x_2\|^{}_{H^{2\a}}+2\|y\|^{}_{H^{2\a}}\Big].
    \end{align}
    If we use this estimate together with the estimate~\eqref{EQ:estim_Gamma} from the previous step, we obtain for all ${v_1,v_2\in \mathcal Y_T}$:
    \begin{align}
        \|\Lambda_T(v_1)-\Lambda_T(v_2)\|^2_{\mathcal Y_T}&\leq 
        C_\nu\|B(v_1+z)-B(v_2+z)\|^2_{L^2(0,T;H^{\a})}\\
        &\leq 
        2c^2C_\nu\int_0^T\|v_1(t)-v_2(t)\|^2_{H^{2\a}}\Big[\|v_1(t)\|^{}_{H^{2\a}}+\|v_2(t)\|^{}_{H^{2\a}}+2\|z(t)\|^{}_{H^{2\a}}\Big]^2\d t \\
        &\leq  2c^2C_\nu T \Big[ \|v_1\|^{}_{\mathcal Y_T}+\|v_2\|^{}_{\mathcal Y_T}+2\|z\|^{}_{C([0,T];H^{2\a})} \Big]^2 \|v_1-v_2\|^2_{\mathcal Y_T}.
    \end{align}
    Similarly, we have for all $v\in \mathcal Y_T$:
    \begin{align}
        \|\Lambda_T(v)\|^{2}_{\mathcal Y_T}
        &\leq          C_\nu\|x\|^{2}_{H^{2\a}} + C_\nu\|B(v+z)\|^2_{L^2(0,T;H^{\a})} \\
        &\leq          C_\nu\|x\|^{2}_{H^{2\a}} + 2c^2C_\nu         T\Big[\|v\|^{}_{\mathcal Y_T}+\|z\|^{}_{C([0,T];H^{2\a})}\Big]^4.
    \end{align}
    Therefore, given $\e\in(0,1)$, we can choose $R>0$ and $T>0$ such that 
    \[
    \begin{cases}
        \ds 2c^2C_\nu T\Big[2R+2\sup_{t\in[0,T]}\|z(t)\|^{}_{H^{2\a}}\Big]^2\leq \e^2\\
        \ds
        C_\nu \|x\|^{2}_{H^{2\a}} + 2c^2C_\nu T\Big[R+\|z\|^{}_{C([0,T];H^{2\a})}\Big]^4 \leq R^2
    \end{cases},
    \]
    which in turn means that $\Lambda_T$ is an $\e$-contraction on the closed ball $\{v\in \mathcal Y_T \ \colon \ \|v\|^{}_{\mathcal Y_T}\leq R\}$. 
    Therefore, for a sufficiently small $T>0$, there exists a unique fixed point $v\in \mathcal Y_T$ for $\Lambda_T$, which is the unique mild solution to equation~\eqref{EQ:v_nu_weak} on the interval $[0,T]$. 
     
\textbf{\textit{Step $\mathbf{3}$.}} We prove a local uniqueness result for the solution. 

    Let us take $x_1,x_2\in H^{2\a}$ and $z_1,z_2\in C\big(\R_+;H^{2\a})$. 
    In the previous step we proved the existence of times $T_i>0$, for $i=1,2$, and mild  solutions $v_i\in C([0,T_i];H^{2\a})\cap L^2(0,T_i;H^{3\a})$, to the equation~\eqref{EQ:v_nu_weak} with parameters $x_i,z_i$, respectively. We restrict the functions to the time interval $[0,T]$, where $T:=T_1\wedge T_2$.
    
    On the interval $[0,T]$, we denote $w=v_1-v_2$, $\zeta=z_1-z_2$ and $u_i=v_i+z_i$, for $i=1,2$. 
    Then $w\in C([0,T];H^{2\a})\cap L^2(0,T;H^{3\a})$, $u_i\in C([0,T];H^{2\a})$, for $i=1,2$, and $w$ satisfies in $H^{2\a}$ the equation
    \[
        w(t)+\nu\int_0^tA^{\a}w(s)\d s +\int_0^tB(u_1(s))-B(u_2(s))\d s =x_1-x_2, \qquad \forall\, t\in[0,T].
    \]
    If we compute the {scalar distributional derivative in time}, we obtain:
    \[
        w'+\nu A^{\a}w+B(u_1(s))-B(u_2(s))=0,
    \]
    thus we also have $w'\in L^2(0,T;H^{\a})$. 
    We now take the duality product of this last equation with $A^{2\a}w(s)\in H^{-\a}$, for \textit{a.e.} $s\in[0,T]$, and integrate over the time interval $[0,t]\subseteq [0,T]$.
    \begin{equation}
    \label{EQ:w_duality}
    \begin{gathered}
        \int_0^t{\big\lan}A^{2\a}w(s),w'(s){\big\ran}_{\ss{\!H^{-\a}\!\times\!H^\a}}\d s+\nu \int_0^t{\big\lan} A^{2\a}w(s), A^{\a}w(s){\big\ran}_{\ss{\!H^{-\a}\!\times\!H^\a}}\d s \\
        =\int_0^t{\big\lan} A^{2\a}w(s),B(u_2(s))-B(u_1(s)){\big\ran}_{\ss{\!H^{-\a}\!\times\!H^\a}}\d s.
    \end{gathered}
    \end{equation}
    We recall from \cite[Chapter $3$, Lemma $1.2$]{Temam_2001_Navier-Stokes_equations} that the first integral equals $\big(\|w(t)\|_{H^{2\a}}^2-\|w(0)\|_{H^{2\a}}^2\big)/2$, while the  integrand of the second term is simply $\|w(s)\|^2_{H^{3\a}}$ for almost any $s\in[0,t]$.
    Moreover, almost everywhere in time:
    \begin{align}
        {\big\lan} A^{2\a}w,B(u_2)-B(u_1){\big\ran}_{\ss{\!H^{-\a}\!\times\!H^\a}}
        &=
        {\big\lan} A^{2\a}w,B(u_2,u_2-u_1)+B(u_2-u_1,u_1){\big\ran}_{\ss{\!H^{-\a}\!\times\!H^\a}}
        \\
        &=-{\big\lan} A^{2\a}w,B(u_1,w+\zeta){\big\ran}_{\ss{\!H^{-\a}\!\times\!H^\a}}
        -{\big\lan} A^{2\a}w,B(w+\zeta,u_2){\big\ran}_{\ss{\!H^{-\a}\!\times\!H^\a}}.
    \end{align}
    These last two terms can be controlled with the estimate~\eqref{EQ:estim_B_1} in Lemma~\ref{LEM:estimates_B} (with $\s=\a+1>2$), the Sobolev embeddings and the Young inequality.
    \begin{align}
        \Big|{\big\lan} A^{2\a}w,B(u_1,w+\zeta){\big\ran}_{\ss{\!H^{-\a}\!\times\!H^\a}}\Big|
        &\leq c\|u_1\|^{}_{H^{\a}}\|w+\zeta\|_{H^{\a+1}}^{}\|A^{2\a}w\|^{}_{H^{-\a}}\\
        &\leq \frac{\nu}{4}\|w\|^{2}_{H^{3\a}}+\frac{c}    {\nu}\|u_1\|^{2}_{H^{2\a}}\|w+\zeta\|_{H^{2\a}}^{2}.
    \end{align}
    Analogously, we have:
    \[                  
        \Big|{\big\lan} A^{2\a}w,B(w+\zeta,u_2){\big\ran}_{\ss{\!H^{-\a}\!\times\!H^\a}}\Big| 
        \leq \frac\nu4\|w\|^2_{H^{3\a}}+\frac{c}{\nu}\|u_2\|^2_{H^{2\a}}\|w+\zeta\|_{H^{2\a}}^{2}.
    \]
    Therefore, rewriting equation~\eqref{EQ:w_duality} as discussed and implementing these estimates, we obtain for all $t\in[0,T]$:
    \begin{equation}
    \begin{aligned}
    \label{EQ:esti_w_zeta}
        &\frac12\|w(t)\|^2_{H^{2\a}}+\frac\nu2\int_0^t\|w(s)\|^2_{H^{3\a}}\d s  \\
        \leq &\;\|x_1-x_2\|^2_{H^{2\a}}
        +\frac{c}\nu T^2\Big[\|u_1\|^2_{C([0,T];H^{2\a})} + \|u_2\|^2_{C([0,T];H^{2\a})}\Big] \|\zeta\|^2_{C([0,T];H^{2\a})}\\
        &\hspace{2.2cm}+\frac{c}\nu T\Big[\|u_1\|^2_{C([0,T];H^{2\a})}+\|u_2\|^2_{C([0,T];H^{2\a})}\Big] \int_0^t\|w(s)\|^2_{H^{2\a}}\d s.
    \end{aligned}
    \end{equation}
    Gr\"onwall's lemma yields for some finite constant $C=C(T,\nu,u_1,u_2)>0$:
    \begin{equation}\label{EQ2:estim_V_infty}
        \|v_1(t)-v_2(t)\|^2_{H^{2\a}}\leq  C\big(\|x_1-x_2\|^2_{H^{2\a}}+\|z_1-z_2\|^2_{C([0,T];H^{2\a})}\big), \qquad \forall\, t\in[0,T].
    \end{equation}
    If we insert this estimate back into equation~\eqref{EQ:esti_w_zeta}, we also obtain:
    \begin{equation}\label{EQ2:estim_V_L2}
        \int_0^T\|v_1(t)-v_2(t)\|^2_{H^{3\a}}\d t  \leq  C\big(\|x_1-x_2\|^2_{H^{2\a}}+\|z_1-z_2\|^2_{C([0,T];H^{2\a})}\big).
    \end{equation}

    The sought local uniqueness result can be directly inferred from the estimates~\eqref{EQ2:estim_V_infty} and \eqref{EQ2:estim_V_L2}.
    Specifically, if $x_1=x_2$ and $z_1=z_2$, then $v_1=v_2$ in $Y_{T}$.
 
\textbf{\textit{Step $\mathbf{4}$.}} 
    We prove \textit{a priori} estimates. 
    
    Assume that $T>0$ and that $v\in \mathcal Y_T$ satisfies equation~\eqref{EQ:vmild} on $[0,T]$, then we can differentiate in the distributional sense and obtain the differential equation in \eqref{EQ:v_nu_weak}. 
    We observe that $A^{\a}v$ and $B(v+z)$ belong to $L^2(0,T;H^{\a})$, thus also $v'\in L^2(0,T;H^{\a})$. 
    In particular, the differential  equation in \eqref{EQ:v_nu_weak} is to be interpreted in $H^{\a}$ almost everywhere in time.  

    For almost every time $t\in[0,T]$, we take the product in $H$ between the differential equation in \eqref{EQ:v_nu_weak} and $Av(t)\in H^{3\a-2}$. 
    We suppress for simplicity the dependence on $t$. Since $b(x,x,Ax)=0$ for all $x\in H^2$, 
    \begin{equation}
    \begin{aligned}  
        \frac12\frac{\d}{\d t}\|v\|^2_{H^1}+\nu\|v\|^2_{H^{\a+1}}
        &\leq |b(v+z,v+z,Av)|\\
        &= |b(v,v,Av)+b(z,v,Av)+b(v+z,z,Av)|\\
        &= |b(z,v,Av)+b(v+z,z,Av)|\\
        &\leq c\|z\|^{}_{H^1}\big(2\|v\|_{H^1}{}+\|z\|_{H^1}{}\big)\|Av\|^{}_{H^{\a-1}}\\
        &\leq \frac\nu2\|v\|^2_{H^{\a+1}}+\frac{4c^2}{\nu}\|z\|^2_{H^1}\|v\|^2_{H^1}+\frac{c^2}{\nu}\|z\|^4_{H^1}.
    \end{aligned}
    \end{equation}
    We used the estimate~\eqref{EQ:estim_B_4} in Lemma~\ref{LEM:estimates_B} (with $\e=\a-1>0$) and the Young inequality.
    We multiply by $2$ the previous estimate and rename the constant $c>0$. 
    By integrating over the time interval $[0,t]\subseteq[0,T]$, we get for all $t\in[0,T]$:
    \begin{equation}  
    \label{EQ2:estimatev}
        \|v(t)\|^2_{H^1}+\nu \int_0^t\|v(s)\|^2_{H^{\a+1}}\d s 
        \leq \|x\|^2_{H^1}+\frac{c}{\nu}\|z\|^2_{C([0,T];H^1)}\int_0^t\|v(s)\|^2_{H^1}\d s +\frac{c}{\nu}T\|z\|^4_{C([0,T];H^1)}.
    \end{equation}
    By Gr\"onwall's lemma we obtain for all $t\in[0,T]$:
    \begin{equation}
    \label{EQ2:estim_v_C0}
        \|v(t)\|^2_{H^1}
        \leq \Big[\|x\|^2_{H^1}+\frac{c}{\nu}T\|z\|^4_{C([0,T];H^1)}\Big]\exp\Big[\frac{c}{\nu}T\|z\|^2_{C([0,T];H^1)}\Big]=:f(T),
    \end{equation}
    where $f:[0,+\infty)\to[0,+\infty)$ is a continuous function depending only on $x,z,\nu$. 
    We insert this estimate in the right-hand side of equation~\eqref{EQ2:estimatev} and divide by $\nu$.
    \begin{equation}
    \label{EQ2:estim_v_L2}
        \int_0^t\|v(s)\|^2_{H^{\a+1}}\d s
        \leq \frac1\nu\|x\|^2_{H^1}+\frac{c}{\nu^2}T\|z\|^2_{C([0,T]H^1)}\big[f(T)+\|z\|^2_{C([0,T];H^1)}\big].
    \end{equation}
    
    For almost every time $t\in[0,T]$, we take the duality product between the differential equation in \eqref{EQ:v_nu_weak} and $A^{2\a}v(t)\in H^{-\a}$. We suppress for simplicity the dependence on $t$.
    \begin{equation}
    \begin{aligned}
        \frac12\frac{\d}{\d t}\|v\|^2_{H^{2\a}}+\nu\|v\|^2_{H^{3\a}}
        &\leq \Big|{\big\lan}A^{2\a}v,B(v+z)\big\ran^{}_{\ss{\!H^{-\a}\!\times\!H^\a}}\Big|\\
        &= \Big|{\big\lan}A^{2\a}v,B(v)+B(z,v)+B(v,z)+B(z)\big\ran^{}_{\ss{\!H^{-\a}\!\times\!H^\a}}\Big|\\ &\leq c\big(\|v\|_{H^{\a+1}}^{2}+2\|z\|^{}_{H^{\a+1}}\|v\|_{H^{\a+1}}^{}+\|z\|^{2}_{H^{\a+1}}\big)\|A^{2\a}v\|^{}_{H^{-\a}}\\
        &\leq 2c\big(\|v\|_{H^{\a+1}}^{2}+\|z\|^{2}_{H^{\a+1}}\big)\|v\|^{}_{H^{3\a}}\\
        &\leq \frac\nu2\|v\|^2_{H^{3\a}}+\frac{4c^2}{\nu}\|z\|^4_{H^{2\a}}
        +\frac{4c^2}{\nu}\|v\|^2_{H^{2\a}}\|v\|^2_{H^{\a+1}}.
    \end{aligned}
    \end{equation}
    We used the estimate~\eqref{EQ:estim_B_1} in Lemma~\ref{LEM:estimates_B} (with $\s=\a+1>2$), the Sobolev embeddings and the Young inequality.
    We multiply by $2$ and rename the constant $c>0$. 
    Integrating over the time interval $[0,t]\subseteq[0,T]$ yields, for all $t\in[0,T]$:
    \begin{equation}  
    \label{EQ2:estimate_v_H2}
        \|v(t)\|^2_{H^{2\a}}+\nu \int_0^t\|v(s)\|^2_{H^{3\a}}\d s 
        \leq \|x\|^2_{H^{2\a}}+\frac{c}{\nu}T\|z\|^4_{C([0,T];H^{2\a})}+\frac{c}{\nu}\int_0^t\|v(s)\|^2_{H^{2\a}}\|v(s)\|^2_{H^{\a+1}}\d s.
    \end{equation}
    Thanks to Gr\"onwall's lemma, we obtain:
    \begin{align}
    \label{EQ2:estim_v_C0_H2}
        \|v(t)\|^2_{H^{2\a}}
        &\leq \Big[\|x\|^2_{H^{2\a}}+\frac{c}{\nu}T\|z\|^4_{C([0,T];H^{2\a})}\Big]\exp\bigg[\frac{c}{\nu}\int_0^T\|v(s)\|^2_{H^{\a+1}}\d s\bigg]
        =:g(T),
    \end{align}
    where $g:[0,+\infty)\to[0,+\infty)$ is a continuous well-defined function, thanks to estimate~\eqref{EQ2:estim_v_L2}, and it depends only on $x,z,\nu$. 
    If we insert this estimate in the right-hand side of equation~\eqref{EQ2:estimate_v_H2} and later divide by $\nu$, we have:
    \begin{equation}
    \label{EQ2:estim_v_L2_3}
        \int_0^t \|v(s)\|^2_{H^{3\a}} \d s
        \leq \frac1\nu\|x\|^2_{H^{2\a}}+\frac{c}{\nu^2}T\|z\|^2_{C([0,T]H^{2\a})}\big[\|z\|^2_{C([0,T];H^{2\a})}+g^2(T)\big]=:h(T),
    \end{equation}
    where also $h:[0,+\infty)\to[0,+\infty)$ is a continuous function depending on $x,z,\nu$.

    \textbf{\textit{Step $\mathbf{5}$.}} 
    We prove that the solution is global. 
    Let us define
    \begin{equation}
    \label{EQ2:defT*}
        T_\ast :=\sup\big\{T>0 \ | \ \exists\, v\in \mathcal Y_T \ : \ v=\Lambda_T (v)\big\}.
    \end{equation}
    By \textit{Step $1$}, $T_\ast$ is well-defined and $T_\ast>0$. Moreover, by the local uniqueness result proved in \textit{Step $3$}, we can concatenate the functions $v=\Lambda_T(v)\in \mathcal Y_T$, for different values of $T\in (0,T_\ast)$, to define a unique function  $v_\ast \in C\big([0,T_\ast );H^{2\a}\big)\cap L^2_{loc}\big([0,T_\ast) ;H^{3\a}\big)$ such that 
    $v_\ast (t)=\Lambda_T \Big(v_\ast \big|_{[0,T]}\Big) (t)$ for all $0\leq t\leq T<T_\ast $.

    We will prove by contradiction that $T_\ast =+\infty$.

    If $T_\ast <+\infty$, then take a sequence of times $T_n\geq0$, $n\in\N$, convergent to $T_\ast $ from below. 
    Applying the \textit{a priori} estimates~\eqref{EQ2:estim_v_C0_H2}, \eqref{EQ2:estim_v_L2_3} from the previous step to $v_\ast $ on $[0,T_n]$, yields:
    \[
        \int_0^{T_\ast }\|v_\ast (t)\|^2_{H^{3\a}}\d t +\sup_{t\in[0,T_\ast )}\|v_\ast (t)\|^2_{H^{2\a}}=\lim_{n\to\infty}\Big\|v_\ast\big|_{[0,T_n]} \Big\|^2_{Y_{T_n}}
        \!\!\!\leq \lim_{n\to\infty}\big[g(T_n)+h(T_n)\big]=g(T_\ast )+h(T_\ast ).
    \]
    This proves that $v_\ast \in L^\infty(0,T_\ast;H^{2\a})\cap L^2(0,T_\ast;H^{3\a})$. 
    Additionally, by \textit{Step $1$}, $\Gamma^x_{T_\ast}\big(B(v_\ast+z)\big)\in Y_{T_\ast}$, hence we can define $v_\ast(T_\ast):=\Gamma^x_{T_\ast}\big( B(v_\ast+z)\big)(T_\ast)$ and obtain $v_\ast=\Lambda_{T_\ast}(v_\ast)\in Y_{T_\ast}$. In particular, $T_\ast $, as defined in equation~\eqref{EQ2:defT*}, is a maximum.

    However, if $T_\ast$ is a maximum, we can restart the equation with with $v_\ast(T_\ast)$, provided we shift in time the forcing.
    Define
    \[
        (\tau_{T_\ast}z)(t):=z(T_\ast+t), \qquad t\geq 0.
    \]
    Since $z\in C(\R_+;H^{2\a})$, we have $\tau_{T_\ast}z\in C(\R_+;H^{2\a})$.
    Hence, by \textit{Step $2$} applied to the initial datum $v_\ast(T_\ast)\in H^{2\a}$
    and to the forcing $\tau_{T_\ast}z$, there exist $T'>0$ and $\bar v\in \mathcal Y_{T'}$ such that, for all $t\in[0,T']$,
    \begin{equation}
    \label{EQ:barv_shifted_mild}
        \bar v(t)
        =e^{-\nu tA^\a}v_\ast(T_\ast)
        -\int_0^t e^{-\nu (t-s)A^\a}
        B\big(\bar v(s)+z(T_\ast+s)\big)\d s.
    \end{equation}
    We now define
    \[
        v(t):=
        \begin{cases}
            v_\ast(t), & t\in[0,T_\ast],\\[1mm]
            \bar v(t-T_\ast), & t\in[T_\ast,T_\ast+T'].
        \end{cases}
    \]
    Since $v(0)=v_\ast(T_\ast)$ by \eqref{EQ:barv_shifted_mild}, the function $v$ is well defined and $v\in \mathcal Y_{T_\ast+T'}$.
    It remains to prove that $v$ is a mild solution to~\eqref{EQ:v_nu_weak} on $[0,T_\ast+T']$. 
    For $t\in[0,T_\ast]$ this is immediate, since $v(t)=v_\ast(t)$.
    Let now $t\in[T_\ast,T_\ast+T']$. 
    By \eqref{EQ:barv_shifted_mild}, the definition of $v$, and the change of variable $r=T_\ast+s$, we get
    \begin{align}
        v(t)
        &=e^{-\nu (t-T_\ast)A^\a}v_\ast(T_\ast)
        -\int_0^{t-T_\ast} e^{-\nu (t-T_\ast-s)A^\a}
        B\big(\bar v(s)+z(T_\ast+s)\big)\d s \notag\\
        &=e^{-\nu (t-T_\ast)A^\a}v_\ast(T_\ast)
        -\int_{T_\ast}^{t} e^{-\nu (t-r)A^\a}
        B\big(v(r)+z(r)\big)\d r.
        \label{EQ:tildev_after_Tstar}
    \end{align}
    On the other hand, since $v_\ast$ is a mild solution on $[0,T_\ast]$, we have
    \[
        v_\ast(T_\ast)
        =e^{-\nu T_\ast A^\a}x
        -\int_0^{T_\ast} e^{-\nu (T_\ast-r)A^\a}
        B\big(v_\ast(r)+z(r)\big)\d r.
    \]
    Applying the semigroup $e^{-\nu (t-T_\ast)A^\a}$ and using the semigroup property,
    we infer
    \begin{align}
        e^{-\nu (t-T_\ast)A^\a}v_\ast(T_\ast)
        &=e^{-\nu tA^\a}x
        -\int_0^{T_\ast} e^{-\nu (t-r)A^\a}
        B\big(v_\ast(r)+z(r)\big)\d r \notag\\
        &=e^{-\nu tA^\a}x
        -\int_0^{T_\ast} e^{-\nu (t-r)A^\a}
        B\big(v(r)+z(r)\big)\d r.
        \label{EQ:semigroup_restart}
    \end{align}
    Combining \eqref{EQ:tildev_after_Tstar} and \eqref{EQ:semigroup_restart}, we obtain
    \[
        v(t)
        =e^{-\nu tA^\a}x
        -\int_0^{t} e^{-\nu (t-r)A^\a}
        B\big(v(r)+z(r)\big)\d r,
        \qquad \forall\, t\in[T_\ast,T_\ast+T'].
    \]
    Therefore $v=\Lambda_{T_\ast + T'}(v)$, contradicting the definition of $T_\ast$ in \eqref{EQ2:defT*}. 
    We conclude by contradiction that $T_\ast =+\infty$, namely the unique local solution $v_\ast $ is global.

\textbf{\textit{Step $\mathbf{6}$.}} 
    In the previous step we proved the existence of a unique solution $v$ to the problem~\eqref{EQ:v_nu_weak} with regularity ${C(\R_+;H^{2\a})\cap L^2_{loc}(\R_+;H^{3\a})}$, thus ensuring the well-posedness of the function $\mathcal V$ as in the statement of the theorem.
    It only remains to prove the additional estimate~\eqref{EQ2:boundedness_continuity_v(x,z)}. 
    
    Assume that $T>0$ and $(x_i,z_i)\in H^{2\a}\times C\big([0,T];H^{2\a}\big)$, $i=1,2$, are such that $\|x_i\|^{}_{H^{2\a}}\leq R$ and $\sup_{t\in[0,T]}\|z_i(t)\|^{}_{H^{2\a}}\leq R$ for $i=1,2$. Let us denote $v_i:=\mathcal V(x_i,z_i)$, then estimates~\eqref{EQ2:estim_v_C0_H2}, \eqref{EQ2:estim_v_L2} and \eqref{EQ2:estim_v_C0} imply the existence of a constant $K>0$ depending on $R,T,\nu>0$ and $\a>1$ such that 
    \[
        \|v_i+z_i\|^{}_{C([0,T];H^{2\a})}\leq \|v_i\|^{}_{C([0,T];H^{2\a})}+\|z_i\|^{}_{C([0,T];H^{2\a})}\leq K.
    \]
    We use this estimate in inequality~\eqref{EQ:esti_w_zeta}.
    For all $t\in[0,T]$,
    \begin{gather}\label{EQ:esti_w_zeta_R}
            \frac12\|v_1(t)-v_2(t)\|^2_{H^{2\a}}+\frac\nu2\int_0^t\|v_1(s)-v_2(s)\|^2_{H^{3\a}}\d s  \\
            \leq \|x_1-x_2\|^2_{H^{2\a}}+\frac{2c}\nu T^2K^2 \|z_1-z_2\|^2_{C([0,T];H^{2\a})}
            +\frac{2c}\nu TK^2 \int_0^t\|v_1(s)-v_2(s)\|^2_{H^{2\a}}\d s.
    \end{gather}
    Gr\"onwall's Lemma gives an estimate for $\sup_{t\in[0,T]}\|v_1(t)-v_2(t)\|^2_{H^{2\a}}$. We insert this estimate back in \eqref{EQ:esti_w_zeta_R} and find the sought inequality~\eqref{EQ2:boundedness_continuity_v(x,z)} for a new constant $C>0$ depending on $R,T,\nu>0$ and $\a>1$.
 \end{proof}

 Let us now adapt the deterministic results of the last theorem to the stochastic hyperviscous Navier-Stokes Equation, as discussed in Remark~\ref{REM:additive_noise}.

\begin{theorem}
\label{TH:exi!_solutions_NS_X}
    Assume that $\a>1$.
    \begin{enumerate}[label=(\roman*)]
    \item 
    \label{IT:solutions_NS_uniqueness}
        For any $\nu>0$, the solutions to the SHNS$_{\nu,\a}$ Equation~\eqref{EQ:NSnu} are both pathwise unique and unique in law, see Definitions~\ref{DEF:pathwise_uniqueness_NS} and \ref{DEF:uniqueness_law_NS}.
    \item 
        Assume that $\nu>0$ and that $(\Omega,\F,\{\F_t\}_{t\geq 0},\P)$ is an augmented filtered probability space with an adapted $H^{2\a}$-valued Wiener process $W$.
        Then we have the following properties:
        \begin{enumerate}[label=(ii.\alph*)]
        \item
        \label{IT:solutions_NS_random}
            If $\xi:\Omega\to H^{2\a}$ is an $\F_0$-measurable random variable, and if $X^\xi$ is defined $\P-a.s.$ by
            \[
                X^\xi_t:=\mathcal V(\xi,Z)(t)+Z_t, \qquad \forall \,t\geq0,
            \]
            where $Z$, $\mathcal V$ were introduced in Theorems  \ref{TH:z} and \ref{TH:exisuniqv}, then $(\Omega,\F,\{\F_t\}_{t\geq 0},\P;W;X^\xi)$ is a solution to the SHNS$_{\nu,\a}$ Equation~\eqref{EQ:NSnu} such that $\P(X^\xi_0=\xi)=1$.
        \item
        \label{IT:solutions_NS_determin}
            For any $x\in H^{2\a}$, let $(\Omega,\F,\{\F_t\}_{t\geq 0},\P;W;X^x)$ be the pathwise unique solution to the SHNS$_{\nu,\a}$ Equation~\eqref{EQ:NSnu} such that $\P(X^x_0=x)=1$.
            Then, for any $T>0$, the function
            \[
                H^{2\a}\ni x\mapsto X^x\in C([0,T];H^{2\a})\cap L^2(0,T;H^{3\a}),
            \]
            is $\P-a.s.$ continuous.  
        \end{enumerate}
    \item
    \label{IT:solutions_NS_regularities}
        If $\nu>0$ and if $(\Omega,\F,\{\F_t\}_{t\geq 0},\P;W;X)$ is a solution to the SHNS$_{\nu,\a}$ Equation~\eqref{EQ:NSnu} such that  
        \[
            X_0\in \Bigg(\bigcap_{p\geq 1}L^p(\Omega;H^1)\Bigg)\cap L^4(\Omega;H^\a)\cap L^2(\Omega;H^{2\a}),
        \]
        then for all $T>0$
        \begin{equation}
        \begin{gathered}
        \label{EQ:solution_NS_regularities}
            X\in \Bigg(\bigcap_{p\geq 1}L^p\big(\Omega;C\big([0,T];H^{1}\big)\big)\Bigg)
            \cap L^4\big(\Omega;C\big([0,T];H^{\a}\big)\big)\\
            \cap \; L^2\big(\Omega;C\big([0,T];H^{2\a}\big)\big)
            \cap L^2\big(\Omega\times[0,T];H^{3\a}\big).    
        \end{gathered}
        \end{equation}
    \item 
    \label{IT:solutions_NS_estimates}
        
        For any $p\geq 2$, there exists $C_p>0$, such that, if $\nu>0$, and if $(\Omega,\F,\{\F_t\}_{t\geq 0},\P;W;X)$ is a solution to the SHNS$_{\nu,\a}$ Equation~\eqref{EQ:NSnu}, then the following estimates hold true for any $T>0$:
        \begin{align}
        \label{EQ:estim_E_sup_X_p_1}
            &\E\Bigg[\sup_{t\in[0,T]}\|X_t\|^p_{H^1}
            +\nu\int_0^T\|X_t\|^2_{H^{\a+1}}\|X_t\|^{p-2}_{H^1}\d t\Bigg]
            \leq C_p\big(\E\|X_0\|^p_{H^1}+(T \nu)^{p/2}\big),\\
        \label{EQ:estim_E_XT_p_1}
            &{\E\big\|X_T\big\|^p_{H^1}
            \leq e^{-(p-1)\nu\,T}\E\|X_0\|^p_{H^1}}+C_p.
        \end{align}
        Moreover, there exists a finite constant $C>0$ such that, if $\nu>0$, and if $(\Omega,\F,\{\F_t\}_{t\geq 0},\P;W;X)$ is a solution to the SHNS$_{\nu,\a}$ Equation~\eqref{EQ:NSnu}, then the following estimates hold true for any $T>0$:
        \begin{align}
        \label{EQ:estim_E_sup_X_2_a}
            &\E\Bigg[\sup_{t\in[0,T]}\|X_t\|^2_{H^\a}
            +\nu\int_0^T\|X_t\|^2_{H^{2\a}}\d t\Bigg]
            \leq 
            C\bigg(\E\|X_0\|^2_{H^\a}+\frac{1}{\nu^2}\E\|X_0\|^2_{H^1}+\frac {T}{\nu}+T \nu\bigg),\\
        \label{EQ:estim_E_sup_X_4_a}
            &\begin{aligned}
                &\E\Bigg[\sup_{t\in[0,T]}\|X_t\|^4_{H^\a}
                +\nu\int_0^T\|X_t\|^2_{H^{2\a}}\|X_t\|^{2}_{H^\a}\d t\Bigg]
                \leq 
                C\bigg(\E\|X_0\|^4_{H^\a}
                +\frac{T}{\nu^3}\E\|X_0\|^{8}_{H^1}\\
                &\hspace{9.8cm}
                +T^2\nu(T^3+\nu)\bigg)
            \end{aligned}\\
        \label{EQ:estim_E_sup_X_2_2a}
            &\begin{aligned}
                &\E\Bigg[
                \sup_{t\in[0,T]}\|X_t\|^2_{H^{2\a}} +\nu\int_0^T\|X_t\|^2_{H^{3\a}}\d t
                \Bigg]\leq 
                C\bigg(\E \|X_0\|^2_{H^{2\a}}
                +\frac{1}{\nu^2}\E\|X_0\|^4_{H^\a}+\frac {T}{\nu^5}\E\|X_0\|^{8}_{H^1}\\
                &\hspace{10.5cm}
                +\frac{T^5}{\nu}+T^2+T\nu
                \bigg)
            \end{aligned}\\
        \label{EQ:estim_E_XT_4_a}
            &\E\big\|X_T\big\|^4_{H^{\a}}
            \leq e^{-\nu T}\E\|X_0\|^4_{H^{\a}}+ 
            C\bigg(\frac{1}{\nu^3}\E\|X_0\|^8_{H^1}+T^4+1\bigg),\\
        \label{EQ:estim_E_XT_2_2a}
            &\E\big\|X_T\big\|^2_{H^{2\a}}
            \leq e^{-\nu T}\E\|X_0\|^2_{H^{2\a}}
            + C\bigg(\frac{1}{\nu^2}\E\|X_0\|^4_{H^\a}+\frac {T}{\nu^5}\E\|X_0\|^{8}_{H^1}+\frac{T^5}{\nu}+T^2+1\bigg).
        \end{align}
    \end{enumerate}
\end{theorem}

\begin{proof}[Proof of \textbf{\itref{IT:solutions_NS_uniqueness}}]
    Let $\nu>0$ and $(\Omega,\F,\{\F_t\}_{t\geq 0},\P;W;X^i)$, $i=1,2$ be two solutions to the SHNS$_{\nu,\a}$ Equation~\eqref{EQ:NSnu} such that the event $\Omega_2:=\{X^1_0=X^2_0\}$ has probability $\P(\Omega_2)=1$.
    Let us also choose $\Omega_1\in\F$ such that $\P(\Omega_1)=1$ and for every $\omega\in\Omega_1$ the corresponding trajectory $Z(\omega)$ of the Ornstein-Uhlenbeck process $Z$ from Theorem~\ref{TH:z} (with $\b:=2\a$) has the regularity $L^2_{loc}(\R_+;H^{3\a})\cap C(\R_+;H^{2\a})$  and satisfies the equality \eqref{EQ:strong_Z} in $H^{2\a}$.
    If $\omega\in \Omega_1\cap \Omega_2$, we denote $x:=X^1_0(\omega)=X^2_0(\omega)$ and $v_i:=X^i(\omega)-Z(\omega)$, for $i=1,2$, then, recalling Definition~\ref{DEF:solution_NSnu}, we have the following identity in $H^{\a}$:
    \[
        v_i(t)+\nu\int_0^tA^{\a}v_i(s)\d s +\int_0^tB\big(v_i(s)+Z_s(\omega)\big)\d s =x, \qquad \forall\, t\geq 0.
    \]
    The uniqueness result stated in Theorem~\ref{TH:exisuniqv} implies that $v_1=v_2=\mathcal V\big(x,Z(\omega)\big)$ in $C(\R_+;H^{2\a})$.
    In particular,
    \[
        1=\P(\Omega_1\cap\Omega_2)\leq \P\Bigg(\bigcap_{t\geq 0}\{v_1(t)=v_2(t)\}\Bigg)=\P\Bigg(\bigcap_{t\geq 0}\{X^1_t=X^2_t\}\Bigg).
    \]
    
    The uniqueness in law property descends from \cite[Corollary $7.7$]{Brzezniak+Motyl_2013_Existence_Martingale_Solution}, see also \cite[Theorem $2$]{Ondrejat_2004_uniqueness_stochastic_evolution_equations}.
\end{proof}
    
\begin{proof}[Proof of \textbf{\itref{IT:solutions_NS_random}}]
    Let us fix $\nu>0$ and an augmented filtered probability space $(\Omega,\F,\{\F_t\}_{t\geq 0},\P)$ with an adapted $H^{2\a}$-valued Wiener process and an $\F_0$-measurable, $H^{2\a}$-valued random variable $\xi$.
    Let $\Omega_1\in\F$ be as above, \textit{i.e.}, such that $\P(\Omega_1)=1$ and $Z(\omega)\in C(\R_+;H^{2\a})\cap L^2_{loc}(\R_+;H^{3\a})$ for all $\omega\in\Omega_1$.
    
    We define the process $X^\xi$ as follows:
    \[
        X^\xi:\R_+\times \Omega\ni(t,\omega)\mapsto X^\xi_t(\omega):=
        \1_{\Omega_1}(\omega)\big[\mathcal V\big(\xi(\omega),Z(\omega)\big)(t)+Z_t(\omega)\big]\in H^{2\a},
    \]  
    where $\mathcal V$ and $Z$ were defined in Theorems~\ref{TH:exisuniqv} and \ref{TH:z}, respectively. 
    By direct inspection, if $\omega\in\Omega_1$, then $X_0^\xi(\omega)=\xi(\omega)$, hence $\P(X^\xi_0=\xi)\geq \P(\Omega_1)=1$.
    We now show that $(\Omega,\F,\{\F_t\}_{t\geq 0},\P;W;X^\xi)$ is a solution to the SHNS$_{\nu,\a}$ Equation~\eqref{EQ:NSnu}, according to Definition~\ref{DEF:solution_NSnu}.

    First of all, the $\F_0$-measurability of $\xi$, the predictability of $Z$ from  Theorem~\ref{TH:z}, the continuity of the function $\mathcal V$ in Theorem~\ref{TH:exisuniqv} and the continuity of the evaluation map ${\iota_{t}^{}:C(\R_+;H^{2\a})\ni u \to u(t)\in H^{2\a}}$, for $t\geq 0$, imply the predictability for the $H^{2\a}$-valued process $X^\xi$. 

    If $\omega\in \Omega_1$, then $Z_t(\omega)$ and $\mathcal V\big(\xi(\omega),Z(\omega)\big)$ belong to $C(\R_+;H^{2\a})\cap L^2_{loc}(\R_+;H^{3\a})$ by definition of $\Omega_1$ and by Theorem~\ref{TH:exisuniqv}, respectively. 
    Hence 
    \[
        \P\big(X^\xi \in C(\R_+;H^{2\a})\cap L^2_{loc}(\R_+;H^{3\a})\big)\geq \P(\Omega_1)=1.
    \]
    
    Finally, if $\omega\in\Omega_1$, we denote $z:=Z(\omega)$ and $v:=\mathcal V\big(\xi(\omega),Z(\omega)\big)$. 
    Then $X^\xi(\omega)$ satisfies the following equation in $H^{\a}$ for all $t\geq0$:
    \begin{align}
        X^\xi_t(\omega)&=v(t)+z(t)\\
        &=-\nu\int_0^tA^{\a}v(s)\d s -\int_0^tB\big(v(s)+z(s)\big)\d s +\xi(\omega)
        -\nu\int_0^tA^{\a}z(s)\d s +\sqrt\nu\,W_t(\omega)\\
        &=-\nu\int_0^tA^{\a}X^\xi_s(\omega)\d s -\int_0^tB\big(X^\xi_s(\omega)\big)\d s +\xi(\omega) +\sqrt\nu\,W_t(\omega).
    \end{align}
    In particular, equation~\eqref{EQ:def_solution_NS} for $X^\xi$ is satisfied $\P-a.s.$ in $H^{\a}$.
\end{proof}
    
\begin{proof}[Proof of \textbf{\itref{IT:solutions_NS_determin}.}]
    We recall from \itref{IT:solutions_NS_random} the following fact. 
    If $\nu>0$, if $(\Omega, \F, \{\F_t\}_{t\geq 0},\P)$ is a filtered probability space satisfying the usual assumptions, with an adapted $H^{2\a}$-valued Wiener process $W$, and if $x\in H^{2\a}$, then the solution $(\Omega, \F, \{\F_t\}_{t\geq 0},\P;W;X^x)$ is given by:
    \[
        X^x:\R_+\times\Omega\ni(t,\omega)\mapsto X^x_t(\omega):=\1_{\Omega_1}(\omega)\big[\mathcal V\big(x,Z(\omega)\big)(t)+Z_t(\omega)\big]\in H^{2\a},
    \]
    where $\Omega_1\in\F$ is such that $\P(\Omega_1)=1$ and $Z(\omega)\in C(\R_+;H^{2\a})\cap L^2_{loc}(\R_+; H^{3\a})$ for any $\omega\in\Omega_1$.
    Moreover, the continuity of the map $\mathcal V$ in Theorem~\ref{TH:exisuniqv} gives the additional continuity of the map 
    \[
        H^{2\a}\ni x\mapsto X^x(\omega)\in C\big([0,T];H^{2\a}\big)\cap L^2(0,T;H^{3\a}),
    \]
    for any $T>0$ and any $\omega\in\Omega_1$.

    The solution $(\Omega, \F, \{\F_t\}_{t\geq 0},\P;W;X^x)$ is pathwise unique by \itref{IT:solutions_NS_uniqueness}. 
    Therefore,  $X^x$ is indistinguishable from any other process $X$ such that $(\Omega,\F,\{\F_t\}_{t\geq 0},\P;W;X)$ is a solution to the SHNS$_{\nu,\a}$ Equation~\eqref{EQ:NSnu} with $\P(X_0=x)=1$.
\end{proof}

\begin{proof}[Proof of \textbf{\itref{IT:solutions_NS_regularities}.}]
    The regularities stated in equation~\eqref{EQ:solution_NS_regularities} trivially derive from the estimates~\eqref{EQ:estim_E_sup_X_p_1}, \eqref{EQ:estim_E_sup_X_4_a}, \eqref{EQ:estim_E_sup_X_2_2a}, and from the path regularities for solutions to the SHNS$_{\nu,\a}$ Equation~\eqref{EQ:NSnu}, see Definition~\ref{DEF:solution_NSnu}. 
\end{proof}

\begin{proof}[Proof of \textbf{\itref{IT:solutions_NS_estimates}.}]
    Let us fix $\nu>0$ and $T>0$.
    Assume that $p\geq 2$.
    We apply seven times the It\^o formula in Lemma~\ref{LEM:Ito_formula_h(|x|^2_b)} to the process $X$, twice to the function $h:\R_+\ni r\mapsto r^{p/2}\in\R_+$ and with $\g=1$, once to the function $h=\id_{\R_+}:\R_+\ni r\mapsto r\in\R_+$ and with $\g=\a$, twice to the function $h:\R_+\ni r\mapsto r^{2}\in\R_+$ and with $\g=\a$, and twice to $h=\id_{\R_+}:\R_+\ni r\mapsto r\in\R_+$ and $\g=2\a$.

\textbf{\textit{Estimate} (\ref{EQ:estim_E_sup_X_p_1}).}
    For $\g=1$, the nonlinearity in equation~\eqref{EQ:ito_X_xi_h} vanishes because of the well-known property $\lan B(x),x\ran^{}_{H^1}=\lan B(x), Ax\ran=0$, for all $x\in H^2$.
    Moreover, by simple properties of bounded operators and traces:
    \begin{equation}
    \label{EQ:control_Tr_Ito}
        \big\|Q_1^{1/2}x\big\|^2_{H^1}\leq \big\|Q_1^{1/2}\big\|^{2}_{\mathcal L(H^1)}\|x\|^{2}_{H^1}=\|Q_1\|^{}_{\mathcal L(H^1)}\|x\|^{2}_{H^1}\leq \Tr[Q_1]\|x\|^{2}_{H^1}, \qquad \forall\, x\in H^1.
    \end{equation}
    In addition, for $h(r)=r^{p/2}$, $r\geq 0$, we have {$h'(r)=pr^{(p-2)/2}/2$ and $h''(r)=p(p-2)r^{(p-4)/2}/4$}.
    We compute the supremum in time over the interval $[0, T]$ on both sides of equation~\eqref{EQ:ito_X_xi_h}, then take the expectation and finally use the estimate~\eqref{EQ:control_Tr_Ito}. 
    We obtain:
    \begin{equation}
    \label{EQ:expected_Ito_X_xi_p}
    \begin{gathered}
        \E\sup_{t\in[0,T]}\|X_t\|^p_{H^1}
        +p\nu \E\int_0^T\|X_s\|^2_{H^{\a+1}}\|X_s\|^{p-2}_{H^1}\d s \\
        \leq \E\|X_0\|^p_{H^1}
        +p\sqrt \nu \,\E\sup_{t\in[0,T]}\int_0^t\|X_s\|^{p-2}_{H^1}\lan X_s, \d W_s\ran ^{}_{H^1}
        + \frac{p(p-1)\nu}{2} \operatorname{Tr}[Q_1]\, \E\int_0^T\|X_s\|^{p-2}_{H^1}\d s.
    \end{gathered}
    \end{equation}
    We control the second term on the right-hand side, thanks to the Burkholder-Davis-Gundy inequality, see \cite[Section $48$, Chapter IV]{Protter_2005_Stochastic_integration_differential_equations}, and by the Schwarz inequality. 
    For an absolute constant $c>0$:
    \begin{align}
        p\sqrt \nu\, \E\sup_{t\in[0,T]} \int_0^t\|X_s\|^{p-2}_{H^1}\lan X_s, \d W_s\ran ^{}_{H^1} 
        &\leq  cp\sqrt\nu\,\E\bigg(\int_0^T\|X_s\|^{2p-2}_{H^1}\d s\bigg)^{1/2} \\
        &\leq cp\sqrt\nu \,\E\Bigg[\sup_{t\in[0,T]}\|X_t\|^{p/2}_{H^1}\bigg(\int_0^T\|X_t\|^{p-2}_{H^1}\d t \bigg)^{1/2}\Bigg]\\
        &\leq \frac12 \E\sup_{t\in[0,T]}\|X_t\|^p_{H^1}
        +\frac{c^2p^2\nu}{2}\E\int_0^T\|X_t\|^{p-2}_{H^1}\d t.
    \end{align}
    By plugging this last estimate into equation~\eqref{EQ:expected_Ito_X_xi_p}, we obtain:
    \begin{equation}
    \label{EQ:expected_ito_X_p_second}
    \begin{split}
        \frac12\E\sup_{t\in[0,T]}\|X_t\|^p_{H^1}
        +p\nu \E\int_0^T\|X_t\|^2_{H^{\a+1}}\|X_t\|^{p-2}_{H^1}\d t \qquad \\
         \leq \E\|X_0\|^p_{H^1}
        +\frac{p\nu }{2}\big(c^2p+
        (p-1) \Tr[Q_1]\big)\E\int_0^T\|X_t\|^{p-2}_{H^1}\d t.
    \end{split}
    \end{equation}
    
    If $p>2$, the last term in the right-hand side of equation~\eqref{EQ:expected_ito_X_p_second} can be controlled by the Young inequality with conjugate exponents $p/2$ and $p/(p-2)$. 
    For a finite constant $C>0$, depending on $p,W$, that may vary from line to line:
    \begin{align}
        \frac{p\nu}{2}\big(c^2p+
        (p-1)\Tr[Q_1]\big)\E\int_0^T\|X_t\|^{p-2}_{H^1}\d t
        &\leq CT\nu\E\sup_{t\in[0,T]}\|X_t\|^{p-2}_{H^1}\\
        &\leq C(T\nu)^{p/2}+\frac12\E\sup_{t\in[0,T]}\|X_t\|^p_{H^1}.
    \end{align}
    This inequality remains trivially true for $p=2$.
    We then insert this last estimate into equation~\eqref{EQ:expected_ito_X_p_second}, and rearrange the terms. 
    We find a new constant $C_p>0$, depending on $p$ and $W$, such that
    \begin{equation}
        \E\sup_{t\in[0,T]}\|X_t\|^p_{H^1} +\nu\E\int_0^T\|X_t\|^2_{H^{\a+1}}\|X_t\|^{p-2}_{H^1}\d t
        \leq C_p\big(\E\|X_0\|^p_{H^1}+(T \nu)^{p/2}\big),
    \end{equation}
    which gives the sought estimate~\eqref{EQ:estim_E_sup_X_p_1}.

\textbf{\textit{Estimate} (\ref{EQ:estim_E_sup_X_2_a}).}
    The last but one term in equation~\eqref{EQ:ito_X_xi_h}, with $\g=\a$ and $h=\1_{\R_+}$, is equal to $0$, and the last term is deterministic. 
    We compute the supremum in time over the interval $[0, T]$ on both sides of equation~\eqref{EQ:ito_X_xi_h}, then take the expectation. 
    We obtain:
    \begin{equation}
    \label{EQ:expected_Ito_X_2_a}
    \begin{aligned}
        \E\sup_{t\in[0,T]}\|X_t\|^2_{H^\a}
        +2\nu \E\int_0^T\|X_s\|^2_{H^{2\a}}\d s
        \leq \E\|X_0\|^2_{H^\a}
        &-2\,\E\sup_{t\in[0,T]}\int_0^t\big\lan B(X_s),X_s\big\ran^{}_{\!H^\a}\d s \\
        &+2\sqrt \nu \,\E\sup_{t\in[0,T]}\int_0^t\lan X_s, \d W_s\ran ^{}_{\!H^\a}\\
        &+ \nu \operatorname{Tr}[Q_\a]T.
    \end{aligned}
    \end{equation}
    The nonlinear term no longer vanishes and has to be controlled by estimate~\eqref{EQ:estim_B_2} in Lemma~\ref{LEM:estimates_B} (with $\e=\a-1>0$), the Young inequality, and by the estimate~\eqref{EQ:estim_E_sup_X_p_1} proved above (with $p=2$).
    \begin{equation}
    \label{EQ:nonlinearity_Ito_2_a}
    \begin{aligned}
        \bigg|2\,\E\sup_{t\in[0,T]}\int_0^t\big\lan B(X_s),X_s\big\ran^{}_{\!H^\a}\d s \bigg|
        &\leq  2\, \E\int_0^T\big|\lan B(X_s),A^\a X_s\big\ran\big|\d s\\
        &\leq  2\, \E\int_0^T\|X_s\|^{}_{H^1}\|X_s\|^{}_{H^{\a}}\|A^\a X_s\|\d s\\
        &\leq  \frac {1}{\nu}\E\int_0^T\| X_s\|^{2}_{H^1}\|X_s\|^{2}_{H^{\a+1}}\d s+\nu\, \E\int_0^T\|X_s\|^{2}_{H^{2\a}}\d s\\
        &\leq  \frac {C_2}{\nu^2}\big(\E\|X_0\|^2_{H^1}+T\nu\big)+\nu\, \E\int_0^T\|X_s\|^{2}_{H^{2\a}}\d s.
    \end{aligned}
    \end{equation}
    
    We control the third term on the right-hand side of the inequality~\eqref{EQ:expected_Ito_X_2_a}, thanks to the Burkholder-Davis-Gundy inequality, see \cite[Section $48$, Chapter IV]{Protter_2005_Stochastic_integration_differential_equations}, and the Schwarz inequality. 
    For a constant $c>0$:
    \begin{equation}
    \begin{aligned}
    \label{EQ:ito_integral_ito_2_a}
        2\sqrt \nu\, \E\sup_{t\in[0,T]} \int_0^t\lan X_s, \d W_s\ran ^{}_{H^\a} 
        &\leq  2c\sqrt\nu\,\E\Bigg[\bigg(\int_0^T\|X_s\|^{2}_{H^\a}\d s\bigg)^{1/2}\Bigg] \\
        &\leq 2c\sqrt{\nu T} \,\E\sup_{t\in[0,T]}\|X_t\|^{}_{H^\a}\\
        &\leq {2c^2\nu}T+\frac12 \E\sup_{t\in[0,T]}\|X_t\|^2_{H^\a}.
    \end{aligned}
    \end{equation}
    
    We plug the estimates~\eqref{EQ:nonlinearity_Ito_2_a} and \eqref{EQ:ito_integral_ito_2_a} into equation~\eqref{EQ:expected_Ito_X_2_a}, and rearrange the terms.
    \begin{equation}
    \begin{gathered}
        \frac12\E\sup_{t\in[0,T]}\|X_t\|^2_{H^\a}
        +\nu \E\int_0^T\|X_t\|^2_{H^{2\a}}\d t  
         \leq \E\|X_0\|^2_{H^\a}
         +\frac {C}{\nu^2}\E\|X_0\|^2_{H^1}+\frac{T}{\nu}
        +2c^2\nu T,
    \end{gathered}
    \end{equation}
    which is the sought estimate~\eqref{EQ:estim_E_sup_X_2_a}, after renaming the constants.

\textbf{\textit{Estimate} (\ref{EQ:estim_E_sup_X_4_a}).}
    By simple properties of bounded operators and traces:
    \begin{equation}
        \big\|Q_\a^{1/2}x\big\|^2_{H^\a}\leq \big\|Q_\a^{1/2}\big\|^{2}_{\mathcal L(H^\a)}\|x\|^{2}_{H^\a}=\|Q_\a\|^{}_{\mathcal L(H^\a)}\|x\|^{2}_{H^\a}\leq \Tr[Q_\a]\|x\|^{2}_{H^\a}, \qquad \forall\, x\in H^\a.
    \end{equation}
    In addition, if $h(r)=r^2$, $r\geq 0$, we have {$h'(r)=2r$ and $h''(r)=2$}.
    We employ these considerations in equation~\eqref{EQ:ito_X_xi_h} with $\g=\a$.
    We compute the supremum in time over the interval $[0, T]$ on both sides of equation~\eqref{EQ:ito_X_xi_h}, then take the expectation and finally use the estimate~\eqref{EQ:control_Tr_Ito}. We obtain:
    \begin{equation}
    \label{EQ:expected_Ito_X_4_a}
    \begin{aligned}
        \E\sup_{t\in[0,T]}\|&X_t\|^4_{H^\a}
        +4\nu \E\int_0^T\|X_s\|^2_{H^{2\a}}\|X_s\|^{2}_{H^\a}\d s\\
        \leq \E\|X_0\|^4_{H^\a}
        &-4\,\E\sup_{t\in[0,T]}\int_0^t\big\lan B(X_s),X_s\big\ran^{}_{\!H^\a}\|X_s\|^{2}_{H^\a}\d s \\
        &+4\sqrt \nu \,\E\sup_{t\in[0,T]}\int_0^t\|X_s\|^{2}_{H^\a}\lan X_s, \d W_s\ran ^{}_{\!H^\a}\\
        &+ {6\nu}\operatorname{Tr}[Q_\a]\, \E\int_0^T\|X_s\|^{2}_{H^\a}\d s.
    \end{aligned}
    \end{equation}
    The nonlinear term no longer vanishes and has to be controlled by estimate~\eqref{EQ:estim_B_2} in Lemma~\ref{LEM:estimates_B} for an appropriate $\e>0$ to be determined, and the Young inequality.
    \begin{equation}
    \label{EQ:nonlinearity_Ito_4_a}
    \begin{aligned}
        &\bigg|4\,\E\sup_{t\in[0,T]}\int_0^t\big\lan B(X_s),X_s\big\ran^{}_{\!H^\a}\|X_s\|^{2}_{H^\a}\d s \bigg|\\
        \leq&\,  4\, \E\int_0^T\big|\lan B(X_s),A^\a X_s\big\ran\big|\|X_s\|^{2}_{H^\a}\d s\\
        \leq&\,  4\, \E\int_0^T\|X_s\|^{}_{H^1}\|X_s\|^{}_{H^{1+\e}}\|A^\a X_s\|\|X_s\|^{2}_{H^\a}\d s\\
        =&\,  4\, \E\int_0^T\| X_s\|^{}_{H^1}\|X_s\|^{}_{H^{1+\e}}\|X_s\|^{}_{H^\a}\|X_s\|^{}_{H^{2\a}}\|X_s\|^{}_{H^\a}\d s\\
        \leq&\,  \frac {2}{\nu}\E\int_0^T\| X_s\|^{2}_{H^1}\|X_s\|^{2}_{H^{1+\e}}\|X_s\|^{2}_{H^\a}\d s+ {2\nu }\, \E\int_0^T\|X_s\|^{2}_{H^{2\a}}\|X_s\|^{2}_{H^\a}\d s.
    \end{aligned}
    \end{equation}
    We now employ the interpolation inequality in Lemma~\ref{LEM:sobolev_interpolation_inequality} with the choices $r=1+\e$, $q=\a$, $\la=1/2$, and the Sobolev embeddings
    \begin{equation}
    \label{EQ:interpolation_a}
        \|X_s\|^{ 2}_{H^{1+\e}}
        \leq \|X_s\|^{}_{H^{2(1+\e)-\a}}\|X_s\|^{}_{H^{\a+1}}
        \leq \|X_s\|^{}_{H^{1}}\|X_s\|^{}_{H^{\a+1}},
    \end{equation}
    where the last inequality holds as soon as
    \begin{equation}
    \label{EQ:ass_p_a}
        2(1+\e)-\a\leq 1\quad \Longleftrightarrow\quad \e\leq \frac{\a-1}{2}.
    \end{equation}
    We insert the estimate~\eqref{EQ:interpolation_a} into the first integral in the last side of the chain of inequalities~\eqref{EQ:nonlinearity_Ito_4_a}, and conclude by the Young and the H\"older inequality. 
    We get, for a finite constant $C>0$, independent of $\nu$ or $T$:
    \begin{equation}
    \begin{aligned}
        &\frac {2}{\nu}\E\int_0^T\| X_s\|^{2}_{H^1}\|X_s\|^{2}_{H^{1+\e}}\|X_s\|^{2}_{H^\a}\d s\\
        \leq\, &\,  \frac {2}{\nu}\E\int_0^T\|X_s\|^{3}_{H^{1}}\|X_s\|^{}_{H^{\a+1}}\|X_s\|^{2}_{H^\a}\d s\\
        \leq\, &\,  \E\bigg[\sup_{t\in[0,T]}\|X_t\|^{2}_{H^\a}\frac{2}{\nu}\int_0^T\|X_s\|^3_{H^1}\|X_s\|^{}_{H^{\a+1}}\d s \bigg]\\
        \leq\, &\,  \frac14\E\sup_{t\in[0,T]}\|X_t\|^{4}_{H^\a} +\frac{C}{\nu^2}\E\bigg(\int_0^T\|X_s\|^3_{H^1}\|X_s\|^{}_{H^{\a+1}}\d s \bigg)^2\\
        \leq\, &\,  \frac14\E\sup_{t\in[0,T]}\|X_t\|^{4}_{H^\a} +\frac{C}{\nu^2}T\, \E\int_0^T\|X_s\|^6_{H^1}\|X_s\|^{2}_{H^{\a+1}}\d s\\
        \leq\, &\,  \frac14\E\sup_{t\in[0,T]}\|X_t\|^{4}_{H^\a} +\frac{C}{\nu^3}TC_8\big(\E\|X_0\|^{8}_{H^1}+T^4 \nu^{4}\big),
    \end{aligned}
    \end{equation}
    where in the last line we used the estimate~\eqref{EQ:estim_E_sup_X_p_1} proved above, with $p=8$.
    If we insert this last estimate into the last line of equation~\eqref{EQ:nonlinearity_Ito_4_a}, and rename the constant $C>0$, we obtain:
    \begin{equation}
    \label{EQ:nonlinearity_Ito_4_a_end}
    \begin{gathered}
        \bigg|4\,\E\sup_{t\in[0,T]}\int_0^t\big\lan B(X_s),X_s\big\ran^{}_{\!H^\a}\|X_s\|^{2}_{H^\a}\d s \bigg|\\
        \leq \frac14\E\sup_{t\in[0,T]}\|X_t\|^{4}_{H^\a} +C\bigg(\frac{T}{\nu^3}\E\|X_0\|^{8}_{H^1}+T^5 \nu\bigg)+{2\nu }\, \E\int_0^T\|X_s\|^{2}_{H^{2\a}}\|X_s\|^{2}_{H^\a}\d s
    \end{gathered}
    \end{equation}
    
    We control the third term on the right-hand side of the inequality~\eqref{EQ:expected_Ito_X_4_a}, thanks to the Burkholder-Davis-Gundy inequality, see \cite[Section $48$, Chapter IV]{Protter_2005_Stochastic_integration_differential_equations}, and the Schwarz inequality. 
    For a constant $c>0$:
    \begin{equation}
    \begin{aligned}
    \label{EQ:ito_integral_ito_p_a}
        4\sqrt \nu\, \E\sup_{t\in[0,T]} \int_0^t\|X_s\|^{2}_{H^\a}\lan X_s, \d W_s\ran ^{}_{H^\a} 
        &\leq  4c\sqrt\nu\,\E\Bigg[\bigg(\int_0^T\|X_s\|^{6}_{H^\a}\d s\bigg)^{1/2}\Bigg] \\
        &\leq 4c\sqrt\nu \,\E\Bigg[\sup_{t\in[0,T]}\|X_t\|^{2}_{H^\a}\bigg(\int_0^T\|X_t\|^{2}_{H^\a}\d t \bigg)^{1/2}\Bigg]\\
        &\leq \frac14 \E\sup_{t\in[0,T]}\|X_t\|^4_{H^\a}
        +{16c^2\nu}\E\int_0^T\|X_t\|^{2}_{H^\a}\d t.
    \end{aligned}
    \end{equation}
    
    By plugging the estimates~\eqref{EQ:nonlinearity_Ito_4_a_end} and \eqref{EQ:ito_integral_ito_p_a} into equation~\eqref{EQ:expected_Ito_X_4_a}, and rearranging the terms, we obtain:
    \begin{equation}
    \begin{aligned}
        &\frac34\E\sup_{t\in[0,T]}\|X_t\|^4_{H^\a}
        +{2\nu} \E\int_0^T\|X_t\|^2_{H^{2\a}}\|X_t\|^{2}_{H^\a}\d t \qquad \\
         \leq\,&\, \E\|X_0\|^4_{H^\a}
         +C\bigg(\frac{T}{\nu^3}\E\|X_0\|^{8}_{H^1}+T^5 \nu\bigg)
        +2\nu\big(3\Tr[Q_\a]+8c^2\big)\E\int_0^T\|X_t\|^{2}_{H^\a}\d t\\
        \leq\,&\, \E\|X_0\|^4_{H^\a}
         +C\bigg(\frac{T}{\nu^3}\E\|X_0\|^{8}_{H^1}+T^5 \nu\bigg)
        +2\nu\big(3\Tr[Q_\a]+8c^2\big)T\, \E\sup_{t\in[0,T]}\|X_t\|^{2}_{H^\a},\\
        \leq\,&\, \E\|X_0\|^4_{H^\a}
         +C\bigg(\frac{T}{\nu^3}\E\|X_0\|^{8}_{H^1}+T^5 \nu\bigg)
        +CT^2\nu^2+\frac14\E\sup_{t\in[0,T]}\|X_t\|^4_{H^\a},
    \end{aligned}
    \end{equation}
    where in the last line we applied the H\"older inequality followed by the Young inequality.
    
    After rearranging the terms and renaming the constants, we reach the sought estimate~\eqref{EQ:estim_E_sup_X_4_a}.

\textbf{\textit{Estimate} (\ref{EQ:estim_E_sup_X_2_2a}).}
    We apply again the It\^o formula in Lemma~\ref{LEM:Ito_formula_h(|x|^2_b)}, with $h=\id_{\R_+}$ and $\g=2\a$.
    The last but one term in equation~\eqref{EQ:ito_X_xi_h} is equal to $0$ and the last term is deterministic. 
    We take the supremum in time and the expectation to both sides.
    \begin{equation}
    \label{EQ:expected_Ito_2_2a}
    \begin{aligned}
            \E\sup_{t\in[0,T]}\|X_t\|^2_{H^{2\a}}
            +2\nu \E\int_0^T\|X_s\|^2_{H^{3\a}}\d s
            = \E\|X_0\|^2_{H^{2\a}}
            &-2\,\E\sup_{t\in[0,T]}\int_0^t
            {\big\lan} A^{2\a}X_s,B\big(X_s\big)\big\ran^{}_{\ss{\!H^{-\a}\!\times\!H^\a}} 
            \d s\\
            &+2\sqrt \nu \,\E\sup_{t\in[0,T]}\int_0^t\lan X_s, \d W_s\ran ^{}_{H^{2\a}}\\
            &+ \nu \operatorname{Tr}[Q_{2\a}]T.
    \end{aligned}
    \end{equation}
    
    The nonlinear term can be controlled by estimate~\eqref{EQ:estim_B_1} (with $\s=\a+1>2$), by the Young inequality, the Sobolev embeddings, and by the estimate~\eqref{EQ:estim_E_sup_X_4_a}.
    \begin{equation}
    \label{EQ:estimate_nonlinear_2_2a}
    \begin{aligned}    
        &\bigg|2\,\E\sup_{t\in[0,T]}\int_0^t
        {\big\lan} A^{2\a}X_s,B\big(X_s\big)\big\ran^{}_{\ss{\!H^{-\a}\!\times\!H^\a}}\d s\bigg|\\
        \leq &\, 2c\,\E\int_0^T\|X_s\|^{}_{H^\a}\|X_s\|^{}_{H^{\a+1}}\|A^{2\a}X_s\|^{}_{H^{-\a}}\d s\\
        \leq &\, \nu\E\int_0^T\|X_s\|^2_{H^{3\a}}\d s +
        \frac{c^2}{\nu}\E\int_0^T\|X_s\|^{2}_{H^{\a+1}}\|X_s\|^2_{H^{\a}}\d s\\
        \leq &\, \nu\E\int_0^T\|X_s\|^2_{H^{3\a}}\d s +
        \frac{c^2}{\nu}\E\int_0^T\|X_s\|^{2}_{H^{2\a}}\|X_s\|^2_{H^{\a}}\d s\\
        \leq &\, \nu\E\int_0^T\|X_s\|^2_{H^{3\a}}\d s +C\bigg(\frac{1}{\nu^2}\E\|X_0\|^4_{H^\a}+\frac {T}{\nu^5}\E\|X_0\|^{8}_{H^1}+\frac{T^5}{\nu}+T^2\bigg).
    \end{aligned}
    \end{equation}
    
    We control the third term on the right-hand side of equation~\eqref{EQ:expected_Ito_2_2a}, thanks to Burkholder-Davis-Gundy inequality, see \cite[Section $48$, Chapter IV]{Protter_2005_Stochastic_integration_differential_equations}, and the Young inequality.
    For a constant $c>0$:
    \begin{align}
    \label{EQ:estimate_Ito_integral_2_2a}
        2\sqrt \nu\, \E\sup_{t\in[0,T]} \int_0^t\lan X_s, \d W_s\ran ^{}_{H^{2\a}}
        &\leq  2c\sqrt\nu\E\bigg(\int_0^T\|X_s\|^{2}_{H^{2\a}}\d s\bigg)^{1/2} \\
        &\leq 2c\sqrt{\nu T} \,\E\sup_{t\in[0,T]}\|X_t\|^{}_{H^{2\a}}\\
        &\leq 2{c^2\nu T}+\frac12\E\sup_{t\in[0,T]}\|X_t\|^2_{H^{2\a}}.
    \end{align}
    
    We use estimates~\eqref{EQ:estimate_nonlinear_2_2a} and \eqref{EQ:estimate_Ito_integral_2_2a} to bound from above the right-hand side of equation~\eqref{EQ:expected_Ito_2_2a}. 
    After rearranging the terms, we find a finite constant $C>0$, such that
    \begin{align}
    \E\Bigg[
        \sup_{t\in[0,T]}\|X_t\|^2_{H^{2\a}} +\nu\int_0^T\|X_t\|^2_{3\a}\d t
        \Bigg]
        \leq &\, C\Big(\E \|X_0\|^2_{H^{2\a}}
        +\frac{1}{\nu^2}\E\|X_0\|^4_{H^\a}+\frac {T}{\nu^5}\E\|X_0\|^{8}_{H^1}+\frac{T^5}{\nu}+T^2+T\nu\bigg).
    \end{align}
    Hence, we proved the sought estimate~\eqref{EQ:estim_E_sup_X_2_2a}.

\textbf{\textit{Estimate} (\ref{EQ:estim_E_XT_p_1}).}
    As far as the fourth estimate is concerned, we fix $s\in[0,T]$ and apply the It\^o formula in Lemma~\ref{LEM:Ito_formula_h(|x|^2_b)} with $\g=1$, $h(r)=r^{q/2}$, for $r\geq 0$, and to the process $\{X_{t+s}\}_{t\geq 0}$, which is an It\^o process adapted to the filtration $\{\F_{t+s}\}_{t\geq 0}$ and with $H^{2\a}$-valued Wiener process $\{W_{t+s}\}_{t\geq 0}$. We reach, $\P-a.s.$ for all $t\geq 0$:
    \begin{equation}
    \begin{aligned}
        \|X_{t+s}\|^q_{H^1}
        = \|X_s\|^q_{H^1}
        &-q\nu\int_0^t\big\lan  A^{\a}X_{u+s}, X_{u+s}\big\ran^{}_{\! H^1}\|X_{u+s}\|^{q-2}_{H^1}\d u \\
        &-q\int_0^t\big\lan B(X_{u+s}), X_{u+s}\big\ran^{}_{\! H^1}\|X_{u+s}\|^{q-2}_{H^1}\d u \\
        &+ q\sqrt \nu \int_0^t\|X_{u+s}\|^{q-2}_{H^1}\lan X_{u+s}, \d W_{u+s}\ran^{}_{\! H^1} \\
        &+ \frac{q(q-2)\nu}{2} \int_0^t
        \big\|Q_1^{1/2}X_{u+s}\big\|^2_{H^1}\|X_{u+s}\|^{q-4}_{H^1}
        \d u\\
        &+ \frac{q\nu}{2} \Tr[Q_1]\int_0^t\|X_{u+s}\|^{q-2}_{H^1}\d u.
    \end{aligned}
    \end{equation}
    The nonlinear term vanishes because of the known property $\lan B(x),x\ran^{}_{\! H^1}=\lan B(x),Ax\ran=0$ for $x\in H^2$. 
    In addition, we evaluate the above expression at the time $t=T-s$ and perform the change of variable $u\mapsto u+s$ in the integrals. Next, we take the expectation to both sides, which cancels the It\^o integral. We are left with:
    \begin{equation}
    \begin{aligned}
        \E\|X_{T}\|^q_{H^1}
        = \E\|X_s\|^q_{H^1}
        &-q\nu\, \E\int_s^T\|X_u\|^2_{H^{\a+1}}\|X_{u}\|^{q-2}_{H^1}\d u \\
        &+ \frac{q(q-2)\nu}{2} \E\int_s^T\big\|Q_1^{1/2}X_u\big\|^2_{H^1}\|X_{u}\|^{q-4}_{H^1}\d u\\
        &+ \frac{q\nu}{2} \Tr[Q_1]\int_s^T\E\|X_{u}\|^{q-2}_{H^1}\d u.
    \end{aligned}
    \end{equation}
    We estimate the right-hand side from above thanks to  $\|x\|^{}_{H^{\a+1}}\geq \|x\|^{}_{H^1}$ and to
    $\big\|Q_1^{1/2}x\big\|^{}_{H^1}\leq \Tr[Q_1]\|x\|^{}_{H^1}$ for $x\in H^{1}$, see equation~\eqref{EQ:control_Tr_Ito}.
    We reach:
    \begin{equation}
        \E\|X_{T}\|^q_{H^1}
        \leq \E\|X_s\|^q_{H^1}
        -q\nu\, \E\int_s^T\|X_{u}\|^{q}_{H^1}\d u 
        + \frac{q(q-1)\nu}{2} \Tr[Q_1]\int_s^T\E\|X_{u}\|^{q-2}_{H^1}\d u.
    \end{equation}
    When $q>2$, the last integral can be controlled with the H\"older inequality with conjugate exponents $q/2$ and $q/(q-2)$ as follows:
    \[
        \frac{q(q-1)\nu}{2} \Tr[Q_1]\int_s^T\E\|X_{u}\|^{q-2}_{H^1}\d u\leq \nu\, C_q(T-s)+\nu\int_s^T\E\|X_u\|^q_{H^1}\d u,
    \]
    where we introduced a constant $C_q>0$ that depends only on $q$ and $\Tr[Q_1]$.
    The inequality remains true also for $p=2$. 
    Therefore,
    \begin{equation}
        \E\|X_{T}\|^q_{H^1}
        \leq \E\|X_s\|^q_{H^1}
        -(q-1)\nu\, \int_s^T\E\|X_{u}\|^{q}_{H^1}\d u 
        + \nu\, C_q(T-s).
    \end{equation}
    By resorting to the version of Gr\"onwall's Lemma~\ref{LEM:Gronwall_lemma_negative}, we obtain:
    \[
        \E\big\|X_T\big\|^q_{H^1}\leq \frac{C_q}{q-1}+e^{-(q-1)\nu T}\E\|X_0\|^q_{H^1},
    \]
    which gives the sought estimate~\eqref{EQ:estim_E_XT_p_1}, after appropriately renaming the constant $C_q>0$.

\textbf{\textit{Estimate} (\ref{EQ:estim_E_XT_4_a}).}
    As for the fifth estimate, we start from the It\^o formula in Lemma~\ref{LEM:Ito_formula_h(|x|^2_b)} with $h:\R_+\ni r\mapsto r^{2}\in\R_+$  and $\g=\a$. 
    We take the expectation to both members of equation~\eqref{EQ:ito_X_xi_h}, which cancels out the It\^o integral, and use Fubini's Theorem.
    We reach for all $t\in[0,T]$
    \begin{equation}
    \begin{aligned}
        \E\|X_{t}\|^4_{H^\a}
        = \E\|X_0\|^4_{H^\a}
        &-4\nu\int_0^t\E\Big[\|X_s\|^2_{H^{2\a}}\|X_{s}\|^{2}_{H^\a}\Big]\d s \\
        &-4\int_0^t\E\Big[\big\lan B(X_{s}), X_{s}\big\ran^{}_{\! H^\a}\|X_{s}\|^{2}_{H^\a}\Big]\d s \\
        &+ {4\nu} \int_0^t\E\Big[
        \big\|Q_\a^{1/2}X_{s}\big\|^2_{H^\a}
        \Big]\d s\\
        &+ {2\nu} \Tr[Q_\a]\int_0^t\E\|X_{s}\|^{2}_{H^\a}\d s.
    \end{aligned}
    \end{equation}
    We can compute the distributional derivative with respect to time and later estimate the right-hand side by
    \begin{equation}
        \big\|Q_\a^{1/2}x\big\|^2_{H^\a}\leq \big\|Q_\a^{1/2}\big\|^{2}_{\mathcal L(H^\a)}\|x\|^{2}_{H^\a}=\|Q_\a\|^{}_{\mathcal L(H^\a)}\|x\|^{2}_{H^\a}\leq \Tr[Q_\a]\|x\|^{2}_{H^\a}, \qquad \forall\, x\in H^\a.
    \end{equation}
    and by 
    \[
        |\lan B(x), x\ran^{}_{\! H^\a}|
        \leq \| B(x)\|\|A^\a x\|
        \leq \|x\|^{}_{H^1}\|x\|^{}_{H^{1+\e}}\|x\|^{}_{H^{2\a}}, \qquad \forall\, x\in H^{2\a}.
    \]
    We reach, after rearranging the terms, and by means of the Young inequality:
    \begin{equation}
    \begin{aligned}
        &\frac{\d}{\d t}\E\|X_t\|^4_{H^{\a}}
        +4\nu\E\Big[\|X_t\|^2_{H^{2\a}} \|X_t\|^{2}_{H^{\a}} \Big]\\
        \leq
        &\;4\,\E\Big[\|X_s\|^{}_{H^1}\|X_s\|^{}_{H^{1+\e}}\|X_s\|^{}_{H^\a}\|X_s\|^{}_{H^{2\a}}\|X_s\|^{}_{H^{\a}}\Big]
        +{6\nu}\Tr[Q_\a]\E\|X_t\|^{2}_{H^{\a}}\\
        \leq &\;\frac{2}{\nu}\E\Big[\|X_s\|^{2}_{H^1}\|X_s\|^{2}_{H^{1+\e}}\|X_s\|^{2}_{H^\a}\Big]+2\nu\E\Big[\|X_s\|^{2}_{H^{2\a}}\|X_s\|^{2}_{H^{\a}}\Big]
        +\frac{1}{2\nu}\big({6\nu}\Tr[Q_\a]\big)^2+\frac\nu2\E\|X_t\|^{4}_{H^{\a}}.
    \end{aligned}
    \end{equation} 
    We rearrange the terms once again, and use the estimate in equation~\eqref{EQ:interpolation_a}, followed by Young inequality
    \begin{equation}
    \begin{aligned}
        \frac{\d}{\d t}\E\|X_t\|^4_{H^{\a}}+2\nu\E\Big[\|X_s\|^{2}_{H^{2\a}}\|X_s\|^{2}_{H^{\a}}\Big]
        &\leq\frac{2}{\nu}\E\Big[\|X_s\|^{3}_{H^1}\|X_s\|^{}_{H^{\a+1}}\|X_s\|^{2}_{H^\a}\Big]
        +\frac{1}{2\nu}\big({6\nu}\Tr[Q_\a]\big)^2+\frac\nu2\E\|X_t\|^{4}_{H^{\a}}\\
        &\leq\frac{2}{\nu^3}\E\Big[\|X_s\|^{6}_{H^1}\|X_s\|^{2}_{H^{\a+1}}\Big]
        +C\nu+\nu\E\|X_t\|^{4}_{H^{\a}}.
    \end{aligned}
    \end{equation} 
    This finally leads, by the Sobolev embedding $\|x\|^4_{H^\a}\leq \|x\|^2_{H^\a}\|x\|^2_{H^{\a+1}}$, for $x\in H^{\a+1}$, to
    \[
        \frac{\d}{\d t}\E\|X_t\|^4_{H^{\a}}+\nu\E\|X_t\|^4_{H^\a}
        \leq \frac{2}
        {\nu^3}\E\Big[\|X_s\|^{6}_{H^1}\|X_s\|^{2}_{H^{\a+1}}\Big]
        +C\nu
    \]
    The last side is locally integrable in time thanks to \eqref{EQ:estim_E_sup_X_p_1}.
    We apply the differential form of Gr\"onwall's Lemma~\ref{LEM:Gronwall_lemma_diff} and obtain for all $t\in[0,T]$
    \begin{align}
        \E\|X_t\|^4_{H^{\a}}
        &\leq e^{-\nu t}\bigg(\E\|X_0\|^4_{H^{\a}}+ \int_0^te^{\nu s}\bigg(\frac{2}{\nu^3}\E\Big[\|X_s\|^6_{H^{1}}\|X_s\|^2_{H^{\a+1}}\Big]
        + C\nu\bigg)\d s \bigg)\\
        &\leq e^{-\nu t}\E\|X_0\|^4_{H^{\a}}+ \frac{2}{\nu^3}\E\int_0^t\|X_s\|^6_{H^{1}}\|X_s\|^2_{H^{\a+1}}\d s +  C\\
        &\leq e^{-\nu t}\E\|X_0\|^4_{H^{\a}}+ \frac{2}{\nu^4}C_8\Big(\E\|X_0\|^8_{H^1}+(t\nu)^4\Big)+  C.
    \end{align}
    We used the linearity of the integral for the second inequality, computing the second part and estimating $e^{\nu s}\leq e^{\nu t}$ for the first part. We used estimate~\eqref{EQ:estim_E_sup_X_p_1}, with $p=8$, for the last inequality.
    The wanted estimate~\eqref{EQ:estim_E_XT_4_a} follows after renaming the constants.
    
\textbf{\textit{Estimate} (\ref{EQ:estim_E_XT_2_2a}).}
    As for the last estimate, we start from the It\^o formula in Lemma~\ref{LEM:Ito_formula_h(|x|^2_b)} with $h=\id_{\R_+}$ and $\g=2\a$. 
    The last but one term in equation~\eqref{EQ:ito_X_xi_h} vanishes, while the last is deterministic. 
    Next, we take the expectation to both members, which cancels out the It\^o integral. 
    We reach for all $t\in[0,T]$
    \begin{equation}
        \E\|X_t\|^2_{H^{2\a}}
        = \E\|X_0\|^2_{H^{2\a}}
        -2\nu\int_0^t\E\|X_s\|^2_{H^{3\a}}\d s 
        +2\int_0^t\E\, 
        {\big\lan} A^{2\a}X_s,B\big(X_s\big)\big\ran^{}_{\ss{\!H^{-\a}\!\times\!H^\a}}
        \d s 
        + \nu  \Tr[Q_{2\a}]\,t.
    \end{equation}
    We can compute the distributional derivative with respect to time and later estimate the right-hand by \eqref{EQ:estim_B_1} (with $\s=\a+1>2$)
    \begin{equation}
    \begin{aligned}
        \frac{\d}{\d t}\E\|X_t\|^2_{H^{2\a}}
        &=
        -2\nu\E\|X_t\|^2_{H^{3\a}} 
        +2\E\, {\big\lan} A^{2\a}X_s,B\big(X_s\big)\big\ran^{}_{\ss{\!H^{-\a}\!\times\!H^\a}}
        + \nu  \Tr[Q_{2\a}]\\
        &\leq -2\nu\E\|X_t\|^2_{H^{3\a}}  
        +\nu\E\, \|A^{2\a} X_t\|^2_{H^{-\a}}
        +\frac1\nu\E\big[ \|X_t\|^2_{H^{\a}}\|X_t\|^2_{H^{\a+1}}\big]
        + \nu  \Tr[Q_{2\a}]\\
        &\leq -\nu\E\|X_t\|^2_{H^{2\a}}  
        +\frac1\nu\E\big[ \|X_t\|^2_{H^{\a}}\|X_t\|^2_{H^{\a+1}}\big]
        + \nu \Tr[Q_{\a}].
    \end{aligned}
    \end{equation}
    We used the Sobolev embedding $\|x\|^{}_{H^{3\a}}\geq \|x\|^{}_{H^{2\a}}$, for $x\in H^{3\a}$, in the last inequality. 
    The right-hand side is locally integrable in time thanks to \eqref{EQ:estim_E_sup_X_4_a}.
    We apply the differential form of Gr\"onwall's Lemma~\ref{LEM:Gronwall_lemma_diff} and obtain for all $t\in[0,T]$
    \begin{align}
        \E\|X_t\|^2_{H^{2\a}}
        &\leq e^{-\nu t}\bigg(\E\|X_0\|^2_{H^{2\a}}+ \int_0^te^{\nu s}\Big(\frac1\nu\E\big[\|X_s\|^2_{H^{\a}}\|X_s\|^2_{H^{\a+1}}\big]
        + \nu \Tr[Q_{2\a}]\Big)\d s \bigg)\\
        &\leq e^{-\nu t}\E\|X_0\|^2_{H^{2\a}}+ \frac1\nu\E\int_0^t\|X_s\|^2_{H^{\a}}\|X_s\|^2_{H^{2\a}}\d s +  \Tr[Q_{2\a}]\\
        &\leq e^{-\nu t}\E\|X_0\|^2_{H^{2\a}}
        + \frac{1}{\nu^2}C\bigg(\E\|X_0\|^4_{H^\a}+\frac {T}{\nu^3}\E\|X_0\|^{8}_{H^1}+T^2\nu(T^3+\nu)\bigg)
        +  \Tr[Q_{2\a}].
    \end{align}
    We used the linearity of the integral for the second inequality, computing the second part and estimating $e^{\nu s}\leq e^{\nu t}$ for the first part. We used estimate~\eqref{EQ:estim_E_sup_X_4_a}, for the last inequality.
    The wanted estimate~\eqref{EQ:estim_E_XT_2_2a} follows after renaming the constants.
\end{proof}

\section{Marginally stationary solutions for the stochastic hyperviscous Navier-Stokes Equation }
\label{SEC:markov_invariant_NS}

In this section, we first prove that the solutions to the stochastic hyperviscous Navier-Stokes Equation, see Definition~\ref{DEF:solution_NSnu} and Theorem~\ref{TH:exi!_solutions_NS_X}, enjoy the Markov property. 
Hence, we can define a Markov semigroup $P$ associated to the equation and prove the existence of an invariant measure for the semigroup $P$, together with some moment estimates.
Next, given a fixed invariant measure $\mu_\nu$, we are able to construct a solution to the SHNS$_{\nu,\a}$ Equation~\eqref{EQ:NSnu} distributed as $\mu_\nu$ for all times.
We call this process a marginally stationary solution to the SHNS$_{\nu,\a}$ Equation~\eqref{EQ:NSnu}.
The marginal stationarity allows us to obtain further moment estimates for the invariant measure, that are also uniform with respect to the kinematic viscosity $\nu$.

\subsection{Markov property}
\label{SEC:markov_property}

We start by recalling the definition of the usual semigroup associated with a solution to a stochastic equation.
\begin{definition}
\label{DEF:semigroup_NS}
    Assume that $\nu>0$ and $\a>1$.
    We define $P:\R_+\ni t\mapsto P_t\in\mathcal L\big(\mathcal B_b(H^{2\a})\big)$ as follows.
    If $x\in H^{2\a}$ and if $(\Omega,\F,\{\F_t\}_{t\geq 0},\P;W;X^x)$ is a solution to the SHNS$_{\nu,\a}$ Equation~\eqref{EQ:NSnu} such that $\P(X^x_0=x)=1$, see Theorem~\ref{TH:exi!_solutions_NS_X} \itref{IT:solutions_NS_determin}, then for all $\varphi\in\mathcal B_b(H^{2\a})$ and $t\geq 0$:
    \begin{equation}
    \label{EQ:def_markov_semigroup}
        P_t\varphi:H^{2\a}\ni x\mapsto \E\big[\varphi(X^x_t)\big]\in\R.
    \end{equation}
\end{definition}
\begin{remark}
    Assume that $t\geq 0$.
    For any $\varphi\in\mathcal B_b(H^{2\a})$, the definition of $P_t\varphi:H^{2\a}\to \R$ is well-posed by Theorem~\ref{TH:exi!_solutions_NS_X} \itref{IT:solutions_NS_determin}.
    Moreover, $P_t\varphi\in\mathcal B(H^{2\a})$ by \cite[Corollary $23$]{Ondrejat_2005_Brownian_representations_cylindrical_local_martingales_strong_Markov},
    and it is bounded by simple calculations.
    Moreover, by direct inspection, the function $P_t:\mathcal B_b(H^{2\a})\ni \varphi\mapsto P_t\varphi\in \mathcal B_b(H^{2\a})$ is linear and bounded.
    
    Finally, the definition depends only on the law of the solution, \textit{i.e.} if $(\tilde \Omega,\tilde \F,\{\tilde \F_t\}_{t\geq 0},\tilde \P;\tilde W;\tilde X^x)$ is another solution to the SHNS$_{\nu,\a}$ Equation~\eqref{EQ:NSnu} such that $\tilde\P(\tilde X^x_0=x)=1$, then by the uniqueness in law property, see Theorem~\ref{TH:exi!_solutions_NS_X} \itref{IT:solutions_NS_uniqueness}:
    \begin{align}
        \E\big[\varphi(X^x_t)\big]
        &=\int_\Omega\varphi(X^x_t(\omega))\d\P(\omega)\\
        &=\int_{H^{2\a}}\varphi(x)\d\big((X^x_t)_\ast\P\big)(x)\\
        &=\int_{H^{2\a}}\varphi(x)\d\big((\tilde X^x_t)_\ast\tilde\P\big)(x)\\
        &=\int_{\tilde\Omega}\varphi(\tilde X^x_t(\omega))\d\tilde \P(\omega)\\
        &=\tilde \E\big[\varphi(\tilde X^x_t)\big].
    \end{align}
\end{remark}

We will now prove that the semigroup $P$ is the Markov semigroup associated to the solution of the stochastic hyperviscous Navier-Stokes equation.
We need a preliminary lemma.

\begin{lemma}
\label{LEM:joint_measurable_solution_map}
    Assume that $\nu>0$, $\a>1$ and let $T>0$.
    Let $(\Omega,\F,\{\F_t\}_{t\geq 0},\P)$ be an augmented filtered probability space
    with an adapted $H^{2\a}$-valued Wiener process $W$.
    For every $x\in H^{2\a}$, let
    $(\Omega,\F,\{\F_t\}_{t\geq 0},\P;W;X^x)$ be the pathwise unique solution to the
    SHNS$_{\nu,\a}$ Equation~\eqref{EQ:NSnu} such that $\P(X^x_0=x)=1$.
    Then there exist a set $\Omega_T\in\F$ with $\P(\Omega_T)=1$ and a map
    \[
        \Psi_T:H^{2\a}\times\Omega\to C\big([0,T];H^{2\a}\big)
    \]
    such that:
    \begin{enumerate}[label=(\roman*)]
        \item for every $x\in H^{2\a}$ and $\omega\in\Omega_T$,  $\Psi_T(x,\omega)=X^x(\omega)\big|_{[0,T]}$,
        \item $\Psi_T$ is
        $\mathscr B_{H^{2\a}}\otimes\F/\mathscr B_{C([0,T];H^{2\a})}$-measurable,
        \item for every $\omega\in\Omega_T$, the map $H^{2\a}\ni x\mapsto \Psi_T(x,\omega)\in C\big([0,T];H^{2\a}\big)$ is continuous.
    \end{enumerate}
    Consequently, if $(\Omega',\mathcal G)$ is a measurable space and
    $\eta:\Omega'\to H^{2\a}$ is $\mathcal G$-measurable, then the map
    \[
        [0,T]\times\Omega'\times\Omega
        \ni (t,\omega',\omega)
        \mapsto \Psi_T(\eta(\omega'),\omega)(t)\in H^{2\a}
    \]
    is
    $\mathscr B_{[0,T]}\otimes\mathcal G\otimes\F/\mathscr B_{H^{2\a}}$-measurable.
    Moreover, if $(\Omega',\mathcal G,\{\mathcal G_t\}_{t\geq 0})$ is filtered and
    $\eta$ is $\mathcal G_0$-measurable, then the process
    \[
        [0,T]\times\Omega'\times\Omega\ni (t,\omega',\omega)\mapsto \Psi_T(\eta(\omega'),\omega)(t)\in H^{2\a},
    \]
    is $\{\mathcal G_t\otimes\F_t\}_{t\in[0,T]}$-adapted and has continuous trajectories;  in particular, it is predictable.
\end{lemma}

\begin{proof}
    By Theorem~\ref{TH:exi!_solutions_NS_X} \itref{IT:solutions_NS_determin},
    there exists a set $\Omega_T\in\F$ with $\P(\Omega_T)=1$ such that, for every
    $\omega\in\Omega_T$, the map $H^{2\a}\ni x\mapsto X^x(\omega)\big|_{[0,T]}\in C\big([0,T];H^{2\a}\big)$ is continuous.
    For every fixed $x\in H^{2\a}$, the random variable $\Omega\ni\omega\mapsto X^x(\omega)\big|_{[0,T]}\in C\big([0,T];H^{2\a}\big)$ is $\F_T$-measurable, because $X^x$ is predictable and has trajectories in $C\big([0,T];H^{2\a}\big)$ $\P$-a.s.
    We define
    \[
        \Psi_T(x,\omega):=\1_{\Omega_T}(\omega)X^x(\omega)\big|_{[0,T]}.
    \]
    Then, for every fixed $x$, the map $\omega\mapsto \Psi_T(x,\omega)$ is $\F_T$-measurable,
    while for every fixed $\omega\in\Omega_T$, the map
    $x\mapsto \Psi_T(x,\omega)$ is continuous.
    Since $H^{2\a}$ and $C\big([0,T];H^{2\a}\big)$ are metric spaces, a standard
    Carath\'eodory measurability theorem (see, e.g., \cite[Lemma~4.51]{Aliprantis+Border_2006_InfiniteDimensionalAnalysis_HitchhikerGuide}) implies that $\Psi_T$ is $\mathscr B_{H^{2\a}}\otimes\F_T/\mathscr B_{C([0,T];H^{2\a})}$-measurable.

    Let now $\eta:\Omega'\to H^{2\a}$ be $\mathcal G$-measurable.
    Then the map
    \[
        \Omega'\times\Omega\ni(\omega',\omega)\mapsto\Psi_T(\eta(\omega'),\omega) \in C([0,T];H^{2\a}),
    \]
    is $\mathcal G\otimes\F_T/\mathscr B_{C([0,T];H^{2\a})}$-measurable.
    Now consider the evaluation map
    \[
        \mathrm{Ev}:[0,T]\times C\big([0,T];H^{2\a}\big)\ni(t,u)\to H^{2\a},
        \qquad
        \mathrm{Ev}(t,u):=u(t).
    \]
    Since $C([0,T];H^{2\a})$ is endowed with the supremum norm and all paths are continuous,
    $\mathrm{Ev}$ is continuous.
    Therefore the map
    \[
        [0,T]\times\Omega'\times\Omega\ni (t,\omega',\omega)
        \mapsto \mathrm{Ev}\bigl(t,\Psi_T(\eta(\omega'),\omega)\bigr)
        =\Psi_T(\eta(\omega'),\omega)(t)
    \]
    is
    $\mathscr B_{[0,T]}\otimes\mathcal G\otimes\F_T/\mathscr B_{H^{2\a}}$-measurable.
    
    Assume in addition that $(\Omega',\mathcal G,\{\mathcal G_t\}_{t\ge0})$ is filtered and that $\eta$ is $\mathcal G_0$-measurable.
    Fix $t\in[0,T]$.
    Applying the previous argument on the interval $[0,t]$, we obtain a map $\Psi_t$ that is $\mathscr B_{H^{2\a}}\otimes\F_t/\mathscr B_{C([0,t];H^{2\a})}$-measurable.
    Then
    \[
        \Omega'\times\Omega\ni(\omega',\omega)\mapsto \Psi_t(\eta(\omega'),\omega)\in C([0,t];H^{2\a})
    \]
    is $\mathcal G_0\otimes\mathcal F_t/\mathscr B_{C([0,t];H^{2\a})}$-measurable.
    Evaluating at time $t$ and using continuity of the fixed-time evaluation map $\mathrm{ev}_t:C([0,t];H^{2\a})\ni u \mapsto u(t)\in H^{2\a}$, we infer that
    \[
        \Omega'\times\Omega\ni(\omega',\omega)\mapsto \Psi_t(\eta(\omega'),\omega)(t)\in H^{2\a}
    \]
    is $\mathcal G_0\otimes\mathcal F_t$-measurable, hence also
    $\mathcal G_t\otimes\mathcal F_t$-measurable.
    Since this holds for every $t\in[0,T]$, the corresponding process is adapted.
    Moreover, by construction its trajectories are continuous; therefore it is predictable.
\end{proof}

\begin{theorem}
\label{TH:markov_X^x}
    Assume that $\nu>0$ and $\a>1$.
    The solutions to the SHNS$_{\nu,\a}$ Equation~\eqref{EQ:NSnu} enjoy the Markov property with semigroup $P$.
    Namely, if $x\in H^{2\a}$ and if $(\Omega,\F,\{\F_t\}_{t\geq 0},\P;W;X^x)$ is a solution to the SHNS$_{\nu,\a}$ Equation~\eqref{EQ:NSnu} such that $\P(X^x_0=x)=1$,  then for all $t\geq 0$, $h\geq 0$ and $\varphi\in\mathcal B_b(H^{2\a})$
    \[
        \E\big[\varphi(X^x_{t+h})\big|\s\big(X^x_s \ : \ s\in[0,t]\big)\big]=P_h\varphi(X^x_t), \qquad \P-a.s.
    \]
\end{theorem}
\begin{proof}
    Assume the hypotheses and let  $t\geq 0$, $h\geq 0$, $x\in H^{2\a}$ and $\varphi\in\mathcal B_b(H^{2\a})$ be fixed. 
    For any $y\in H^{2\a}$, we consider a solution $(\Omega,\F,\{\F_t\}_{t\geq 0},\P;W;X^y)$ to the SHNS$_{\nu,\a}$ Equation~\eqref{EQ:NSnu} such that $\P(X^y_0=y)=1$.
    
    First of all, since we have $P_h\varphi\in \mathcal B(H^{2\a})$ by Definition~\ref{DEF:semigroup_NS}, then the function ${\Omega\ni \omega\mapsto P_h\varphi\big(X^x_t(\omega)\big)\in\R}$ is {$\sigma \big(X^x_t\big)$-measurable}, hence also $\s\big(X^x_s \ : \ s\in[0,t]\big)$-measurable.
    We now need to show that 
    \[
        \E\big[P_h\varphi(X^x_t)\1_E\big]=\E\big[\varphi(X^x_{t+h})\1_E\big], \qquad \forall\, E\in\sigma \big(X^x_s \ : \ s\in[0,t]\big).
    \]
    According to \cite[Theorem $9.2$]{williams_1991_probability}, since the family
    \[
        \Bigg\{\bigcap_{i=1}^n\big\{X^x_{s_i}\in\Gamma_i\big\} \ : \ n\in\N, \ \Gamma_1,\dots,\Gamma_n\in \mathscr B_{H^{2\a}}, \ s_1,\dots,s_n\in [0,t]\Bigg\}
    \]
    is a $\pi$-system that generates the $\s$-algebra $\s\big(X^x_s \ : \ s\in[0,t]\big)$, it is sufficient to prove that, for arbitrarily fixed $n\in\N$, $\Gamma_1,\dots,\Gamma_n\in \mathscr B_{H^{2\a}}$ and $s_1,\dots,s_n\in [0,t]$, we have 
    \begin{equation}
    \label{EQ:markov_property}
        \E\bigg[P_h\varphi(X^x_t)\prod_{i=1}^n\1_{\Gamma_i}(X^x_{s_i})
        \bigg]= \E\bigg[\varphi(X^x_{t+h})
        \prod_{i=1}^n\1_{\Gamma_i}(X^x_{s_i})\bigg].
    \end{equation}
    
    We start from the left-hand side and use the definition of semigroup in equation~\eqref{EQ:def_markov_semigroup}:
    \begin{equation}
    \label{EQ:markov_property_1}
    \begin{aligned}
        \E\bigg[
        P_h\varphi(X^x_t)\prod_{i=1}^n\1_{\Gamma_i}(X^x_{s_i})
        \bigg]=&
        \int_\Omega P_h\varphi(X^x_t(\omega))\prod_{i=1}^n\1_{\Gamma_i}(X^x_{s_i}(\omega))\d \P(\omega)\\
        =&
        \int_\Omega \int_\Omega \varphi\Big(X_h^{X^x_t(\omega)}(\omega')\Big)\prod_{i=1}^n\1_{\Gamma_i}(X^x_{s_i}(\omega))\d \P(\omega')\d \P(\omega).
    \end{aligned}
    \end{equation}
    Let $\Psi_h$ be the map furnished by
    Lemma~\ref{LEM:joint_measurable_solution_map} on the interval $[0,h]$.
    We define, for $s\in[0,t+h]$,
    \begin{align}
        \tilde X_s: \Omega\times \Omega \ni(\omega,\omega')\mapsto \tilde X_s(\omega,\omega')
        &:=
        \begin{cases}
            X_s^x(\omega), \qquad &\textit{if }s\in[0,t],\\[1mm]
            \Psi_h\big(X^x_t(\omega),\omega'\big)(s-t), \qquad &\textit{if } s\in(t,t+h],
        \end{cases}\\
        \tilde W_s: \Omega\times \Omega \ni(\omega,\omega')\mapsto \tilde W_s(\omega,\omega')
        &:=
        \begin{cases}
            W_s(\omega), \qquad &\textit{if } s\in[0,t],\\
            W_{s-t}(\omega')+W_t(\omega), &\textit{if } s\in(t,t+h].
        \end{cases}
    \end{align}
    and consider the augmentation of the filtered probability space $(\Omega\times\Omega, \F\otimes\F, \{\F_s\otimes\F_s\}_{s\geq0},\P\otimes\P)$, that we denote by $(\Omega\times\Omega, \tilde\F, \{\tilde\F_s\}_{s\geq0},\tilde\P)$, where $\tilde \Omega:=\Omega\times \Omega$.
    In particular, by the Lévy Martingale Characterization Theorem, see \cite[Theorem $4.6$]{DaPrato+Zabczyk_2014_stochastic_equations_infinite}, $\tilde W$ is an $\{\tilde\F_s\}_{s\geq 0}$-adapted $H^{2\a}$-valued Wiener process. See \cite[Lemma $4.8$]{Debussche+Hofmanova+Vovelle_2016_Degenerate_parabolic_stochastic_partial_dfferential_equations_Quasilinear_case}.
    
    We first justify that $\tilde X$ is a predictable $H^{2\a}$-valued process on
    $[0,t+h]$.
    Since $X^x_t$ is $\F_t$-measurable,  Lemma~\ref{LEM:joint_measurable_solution_map} yields that
    $\Psi_h\big(X^x_t(\omega),\omega'\big)(s-t)$ is
    $\F_t\otimes\F_{s-t}$-measurable, hence also
    $\F_s\otimes\F_s$-measurable; therefore $\tilde X$ is adapted to
    $\{\tilde\F_s\}_{s\in[0,t+h]}$.
    Since both branches have continuous trajectories, $\tilde X$ is predictable.

    We will now show that
    $(\tilde\Omega,\tilde\F,\{\tilde\F_s\}_{s\in[0,t+h]},\tilde\P;\tilde W;\tilde X)$
    is a solution to the SHNS$_{\nu,\a}$ Equation~\eqref{EQ:NSnu} on $[0,t+h]$
    such that $\tilde\P(\tilde X_0=x)=1$.
    The path regularities and the initial condition are immediate from the definition.
    We only need to show the validity of the equality \eqref{EQ:def_solution_NS} for $\tilde W$ and $\tilde X$, $\tilde\P-a.s.$ in $H^{\a}$.
    If $s\in[0,t]$, then $\tilde X_s(\omega,\omega')= X^x_s(\omega)$ and $\tilde W_s(\omega,\omega')=W_s(\omega)$ for all $(\omega,\omega')\in\tilde\Omega$, hence there is nothing to prove.
    If $s>t$, for $\tilde\P-a.e.$ $(\omega,\omega')\in\tilde\Omega$, by a change of variable in time and by the definition of $X^x$ in equation~\eqref{EQ:def_solution_NS}
    \begin{align}
        \tilde X_s(\omega,\omega')
        &=\Psi_h\big(X^x_t(\omega),\omega'\big)(s-t)\\
        &=-\nu\int_0^{s-t}\!\!\!\!A^{\a}\Psi_h\big(X^x_t(\omega),\omega'\big)(r)\d r-\int_0^{s-t}\!\!\!\!B\Big(\Psi_h\big(X^x_t(\omega),\omega'\big)(r)\Big)\d r+\sqrt \nu \,W_{s-t}(\omega')+X^x_t(\omega)\\
        &=-\nu\int_t^{s}A^{\a}\Psi_h\big(X^x_t(\omega),\omega'\big)(r-t)\d r-\int_t^{s}B\Big(\Psi_h\big(X^x_t(\omega),\omega'\big)(r-t)\Big)\d r+\sqrt \nu \,W_{s-t}(\omega')\\
        &\quad -\nu\int_0^{t}A^{\a}X_{r}^x(\omega)\d r-\int_0^{t}B(X_{r}^x(\omega))\d r+\sqrt \nu \,W_{t}(\omega)+x\\
        &=-\nu\int_0^{s}A^{\a}\tilde X_{r}(\omega,\omega')\d r-\int_0^{s}B(\tilde X_{r}(\omega,\omega'))\d r+\sqrt \nu \,\tilde W_{s}(\omega,\omega')+x.
    \end{align}
    
    The uniqueness in law from Theorem~\ref{TH:exi!_solutions_NS_X} \itref{IT:solutions_NS_uniqueness} yields $(\tilde X)_\ast \tilde\P=(X^x)_\ast\P$ on $\mathscr P\big(C([0,T];H^{2\a})\big)$ for any $T>0$. 
    Therefore, we can conclude the chain of equalities \eqref{EQ:markov_property_1} by means of Fubini's Theorem and thanks to the fact that $\tilde X$ and $X^x$ share the same finite-dimensional laws on $\mathscr B_{H^{2\a}}$
    \begin{align}
        \E\bigg[
        P_h\varphi(X^x_t)\prod_{i=1}^n\1_{\Gamma_i}(X^x_{s_i})
        \bigg]=&
        \int_{\Omega\times\Omega}\varphi\Big(\Psi_h\big(X^x_t(\omega),\omega'\big)(t)\Big)\prod_{i=1}^n\1_{\Gamma_i}(X^x_{s_i}(\omega))\d (\P\otimes\P)(\omega,\omega')\\
        =&
        \int_{\tilde\Omega}\varphi\big(\tilde X_{t+h}(\omega,\omega')\big)\prod_{i=1}^n\1_{\Gamma_i}\big(\tilde X_{s_i}(\omega,\omega')\big)\d \tilde\P(\omega,\omega')\\
        =&
        \int_{\Omega}\varphi\big(X^x_{t+h}(\omega)\big)\prod_{i=1}^n\1_{\Gamma_i}\big(X^x_{s_i}(\omega)\big)\d \P(\omega)\\
        =&\, \E\bigg[\varphi(X^x_{t+h})
        \prod_{i=1}^n\1_{\Gamma_i}(X^x_{s_i})
        \bigg].
    \end{align}
    We obtained the sought equality \eqref{EQ:markov_property} and the thesis follows.
\end{proof}

\subsection{Invariant measure}

We now introduce the adjoint $P^\ast$ to the Markov semigroup $P$, which will be later used to define the invariant measure for the stochastic hyperviscous  Navier-Stokes equations.

\begin{definition}
\label{DEF:adjoint_semigroup}
    Assume that $\nu>0$ and $\a>1$.
    We define the adjoint semigroup $P^\ast:=\{P_t^\ast\}_{t\geq 0}$ of $P$, see Definition~\ref{DEF:semigroup_NS}, as follows:
    \[
        P_t^\ast:\mathscr P(H^{2\a})\ni\mu\mapsto P_t^\ast\mu\in\mathscr P(H^{2\a}), \qquad \forall\, t\geq 0,
    \]
    where, for all $t\geq 0$:
    \[
        \big(P_t^\ast\mu\big)(U):=\int_{H^{2\a}}P_t\1_U\d\mu, \qquad \forall\, U\in\mathscr B_{H^{2\a}}.
    \]
\end{definition}

\begin{remark}
    The definition of $P^\ast$ is well-posed, for any choice of parameters $\nu>0$ and $\a>1$.
    Namely, for any $\mu\in\mathscr P(H^{2\a})$ and any $t\geq 0$, one can easily verify that $P_t^\ast\mu$ is again a probability measure on $(H^{2\a},\mathscr B_{H^{2\a}})$, thanks to the definition of $P_t$, see equation~\eqref{EQ:def_markov_semigroup}, to Theorem~\ref{TH:exi!_solutions_NS_X} \itref{IT:solutions_NS_determin}, and to the continuity of $P_t$.
\end{remark}
\begin{remark}
    One can prove, by a standard machine from measure theory, that for any $t\geq0$ and any probability measure $\mu\in\mathscr P(H^{2\a})$
    \begin{equation}
    \label{EQ:int_Pt*}
        \int_{H^{2\a}}\varphi\d\big(P_t^\ast\mu\big)=\int_{H^{2\a}}P_t\varphi\d\mu, \qquad \forall\, \varphi\in \mathcal B_b(H^{2\a}).
    \end{equation}
    
    Indeed, if $\varphi=\1_{U}$ for some $U\in\mathscr B_{H^{2\a}}$, then the property \eqref{EQ:int_Pt*} is trivial.     
    If $\varphi$ is simple, \textit{i.e.} there exist $N\in\N$ and $\{U_i\}_{i=1}^N\subset \mathscr B_{H^{2\a}}$ such that $\varphi=\sum_{i=1}^N\1_{U_i}$, then the property \eqref{EQ:int_Pt*} follows by linearity of the integral and the previous step.     
    If $\varphi\in\mathcal B_b(H^{2\a})$, then there exists a sequence of simple functions $\{\varphi_n\}_{n\in\N}\subset \mathcal B_b(H^{2\a})$ that converges $P_t^\ast\mu-a.s.$ to $\varphi$.
    In this general case, the property \eqref{EQ:int_Pt*} follows by the Dominated Convergence Theorem and the previous step.
\end{remark}

\begin{definition}
\label{DEF:invariant_measure}
    Assume that $\nu>0$ and $\a>1$.
    A probability measure $\mu\in\mathscr P(H^{2\a})$ is an invariant measure for the semigroup $P$, see Definition~\ref{DEF:semigroup_NS}, if
    \[
        P_t^\ast\mu=\mu, \qquad \forall\, t\geq 0.
    \]
\end{definition}

A standard tool to study the existence of invariant measure for stochastic partial differential equations is the Krylov-Bogoliubov Theorem, see, for instance, \cite[Section $3.1$]{DaPrato+Zabczyk_1996_ergodicity_infinite}.
This theorem requires the Feller property for the semigroup $P$ and the tightness for the sequence of probability measures $\frac{1}{n}\int_0^nP_t^\ast\delta_x\d t$, $n\in\N$, where $x$ is any point in $H^{2\a}$.

In our scenario, the Feller property is satisfied thanks to Theorem~\ref{TH:exi!_solutions_NS_X} \itref{IT:solutions_NS_determin}. 
Were the equation set in a bounded domain $D\subset \R^2$, the tightness property could be easily proved by showing some uniform estimates in time for the $p-$moment of the solution, where $p>1$, with respect to the norm of some space $H^\s(D;\R^2)$, where $\s>2\a$.
Indeed, the Rellich-Kondrachov Theorem yields $H^\s(D;\R^2)\doublehookrightarrow H^{2\a}(D;\R^2)$ if $D$ is bounded and $\s>2\a$.
Hence closed balls in $H^\s(D;\R^2)$ are compact in $H^{2\a}(D;\R^2)$, and can be thus utilized to verify the tightness of the sequence.
However, in our study case, the equation is set in the unbounded domain $\R^2$. 
In particular, the Sobolev embeddings are still continuous, yet not compact. 

In order to address the problem, we resort to a different version of the Krylov-Bogoliubov Theorem, see \cite[Section $3$]{Maslowski+Seidler_1999_On_sequentially_weakly_Feller}, which weakens the requirement of tightness and strengthens the Feller property. 
Specifically, the compact set required by the tightness property can be replaced by a bounded set, provided the ambient space is endowed with the weak topology. Consequently, the Feller property has to be replaced by a sequentially weak version.

\begin{theorem}[\text{\cite[Proposition $3.1$]{Maslowski+Seidler_1999_On_sequentially_weakly_Feller}}]
\label{TH:krylov_bogoliubov_sequentially_weak_Feller}
    Assume that $\nu>0$ and $\a>1$.
    There exists an invariant measure for the semigroup $P$, see Definition~\ref{DEF:semigroup_NS}, if the following two hypotheses are satisfied:
    \begin{enumerate}[label=$(\roman*)$]
    \item (Sequentially weak Feller property)
        If $\varphi\in C_b(H^{2\a}_w)$, $x\in H^{2\a}$ and $\{x_n\}_{n\in\N}\subset H^{2\a}$ is a sequence weakly convergent to $x$ in $H^{2\a}$, then 
        \[
            \lim_{n\to\infty} P_t\varphi(x_n)=P_t\varphi(x), \qquad \forall\, t\geq0.
        \]
    \item 
        For any $\e>0$ there exists $R>0$ such that
        \[
            \limsup_{T\to\infty} \frac1T\int_0^T(P_t^\ast\delta_x)\Big(H^{2\a}\setminus \bar B^{H^{2\a}}_R\Big)\d t \leq \e,
        \]
        where $\bar B^{H^{2\a}}_R:=\big\{y\in H^{2\a} \ : \ \|y\|^{}_{H^{2\a}}\leq R\big\}$.
    \end{enumerate}
\end{theorem}

An auxiliary result is required to prove the sequentially weak Feller property for the semigroup.
 
\begin{lemma}
\label{LEM:deterministic_bw_Feller}
    Assume that $\nu>0$, $\a>1$ and $z\in C(\R_+;H^{2\a})$. Then the function
    \[
        H^{2\a}_w\ni x\mapsto \mathcal V(x,z)\in C\big([0,T];H^{2\a}_w\big),
    \]
    see Theorem~\ref{TH:exisuniqv}, is sequentially continuous. 
    Namely, the following property is satisfied. 
    For any $x\in H^{2\a}$, any sequence $\{x_n\}_{n\in\N}$ in $H^{2\a}$ weakly convergent to $x\in H^{2\a}$, and any $T>0$:
    \[
        \lim_{n\to\infty}\sup_{t\in[0,T]}\Big|{\big\lan} y,\mathcal V(x_n,z)(t)-\mathcal V(x,z)(t){\big\ran}^{}_{\ss{\!H^{-2\a}\!\times\!H^{2\a}}}\Big|=0, \qquad \forall\, y\in H^{-2\a}.  
    \]
\end{lemma}
\begin{proof}
\textit{\textbf{Step }$\mathbf{1}$.} 
    Let us fix $T>0$, $x\in H^{2\a}$ and a sequence $\{x_n\}_{n\in\N}$ weakly convergent to $x$ in $H^{2\a}$. 
    To ease the notation we denote for all $n\in\N$
    \begin{align}
        &v_n:=\mathcal V(x_n,z)\big|^{}_{[0,T]}:[0,T]\to H^{2\a},\\
        &\tilde v_n:=\mathcal V(x_n,z)\1_{[0,T]}:\R\to H^{2\a},\\
        &\tilde f_n:=-\nu A^{\a}\tilde v_n-B(\tilde v_n+z)\1_{[0,T]}:\R\to H^{2\a}.
    \end{align}
    We first prove that there exists $\gamma>0$ such that the sequence $\{\tilde v_n\}_{n\in\N}$ is bounded in $H^\g(\R;H^{2\a})$.
    
    The sequence 
    \begin{equation}
    \label{EQ:bounded_xn}
        \{x_n\}_{n\in\N} \quad \text{ is bounded in }\quad  H^{2\a},
    \end{equation}
    because it is weakly convergent in $H^{2\a}$.
    Hence the estimate~\eqref{EQ2:boundedness_continuity_v(x,z)} in Theorem~\ref{TH:exisuniqv} implies that $\{v_n\}_{n\in\N}$ is bounded in $C\big([0,T];H^{2\a}\big)\cap L^2(0,T;H^{3\a})$.
    In particular, the sequence 
    \begin{equation}
    \label{EQ:bounded_vn}
        \{\tilde v_n\}_{n\in\N} \quad \text{ is bounded in }\quad  L^\infty(\R;H^{2\a})\cap L^2(\R;H^{3\a}).
    \end{equation}
    Moreover, for all $n\in\N$, $\tilde f_n$ has compact support $[0,T]$, and 
    \begin{align}
        \bigg(\int_\R\|\tilde f_n(t)\|_{H^\a}\d t\bigg)^2
        &\leq T\int_0^T\|\tilde f_n(t)\|^2_{H^\a}\d t
        \label{EQ:bound_f_n_L1_1}\\
        &\leq 2T\nu^2\int_0^T\|A^{\a}v_n(t)\|^2_{H^\a}\d t
        +2T\int_0^T\|B\big(v_n(t)+z(t)\big)\|^2_{H^\a}\d t
        \notag\\
        &\leq 2T\nu^2\int_0^T\|v_n(t)\|^2_{H^{3\a}}\d t
        +2c^2T\int_0^T
        \|v_n(t)+z(t)\|^2_{H^\a}\,
        \|v_n(t)+z(t)\|^2_{H^{\a+1}}
        \d t
        \notag\\
        &\leq 2T\nu^2\int_0^T\|v_n(t)\|^2_{H^{3\a}}\d t
        +2c^2T\,
        \!\!\!\sup_{t\in[0,T]}\|v_n(t)+z(t)\|^2_{H^\a}
        \int_0^T\|v_n(t)+z(t)\|^2_{H^{\a+1}}\d t
        \notag\\
        &\leq C_{T,z}
        \Big(
            1+\sup_{t\in[0,T]}\|v_n(t)\|^2_{H^{2\a}}
        \Big)
        \Big(
            1+\int_0^T\|v_n(t)\|^2_{H^{3\a}}\d t
        \Big),
        \notag
    \end{align}
    We used Jensen's inequality, the estimate~\eqref{EQ:estim_B_1} from Lemma~\ref{LEM:estimates_B} with $\s=\a+1>2$, and the continuous embeddings $H^{2\a}\hookrightarrow H^{\a}$, $H^{2\a}\hookrightarrow H^{\a+1}$, $H^{3\a}\hookrightarrow H^{\a+1}$, which hold because $\a>1$.
    The right-hand side is uniformly bounded in $n$ thanks to~\eqref{EQ2:boundedness_continuity_v(x,z)} and to the boundedness of $\{x_n\}_{n\in\N}$ in $H^{2\a}$.
    Therefore $\{\tilde f_n\}_{n\in\N}$ is bounded in $L^1(\R;H^\a)$.
    In particular \begin{equation}
    \label{EQ:bounded_Ffn}
        \big\{\mathscr F[\tilde f_n]\big\}_{n\in\N}\in  \big(C(\R;H^\a)\big)^\N \quad \text{ is bounded in }\quad  L^\infty(\R;H^\a).
    \end{equation}
    
    Let us now fix $n\in\N$. By direct inspection, the function $\tilde v_n$ satisfies in $H^\a$
    \begin{equation}
    \begin{aligned}
    \label{EQ2:tilde_v_n}
        \tilde v_n(t)&=
        \begin{cases}
            \ds \int_0^t\tilde f_n(s)\d s +x_n, \qquad &\forall\, t\in[0,T],\\
            0,&\forall\, t\in\R\setminus[0,T],
        \end{cases}\\
        &=\1_{[0,T]}(t)\bigg(\int_0^t\tilde f_n(s)\d s +x_n\bigg).
    \end{aligned}
    \end{equation}
    Let $\{e_k\}_{k\in\N}$ be an orthonormal complete system for $H^{3\a}$, in particular $\{A^{2\a}e_k\}_{k\in\N}$ becomes an orthonormal complete system for $H^{-\a}$. 
    Since $\tilde v_n\in L^2(\R;H^{3\a})$, there exists $\{\psi_k\}_{k\in\N}\subset L^2(\R)$ such that 
    \begin{align}
        &\tilde v_n(t)=\sum_{k=1}^\infty \psi_k(t)e_k,\qquad \textit{a.e.}\ t\in\R,\\
        &\mathscr F[\tilde v_n](t)=\sum_{k=1}^\infty \mathscr F[\psi_k](t)e_k,\qquad\forall\, t\in\R,
    \end{align}
    where the series converge in $H^{3\a}$. 
    We take the duality product in $H^{-\a}\slash H^\a$ of $A^{2\a}e_k$ with equation~\eqref{EQ2:tilde_v_n}, for any $k\in\N$, and later derive in the distributional sense with respect to time:
    \[
        \Big({\big\lan}A^{2\a}e_k,\tilde v_n\big\ran^{}_{\ss{\!H^{-\a}\!\times\!H^\a}}\Big)'=
        {\big\lan}A^{2\a}e_k,\tilde f_n\big\ran^{}_{\ss{\!H^{-\a}\!\times\!H^\a}}+
        {\big\lan}A^{2\a}e_k,x_n\big\ran^{}_{\ss{\!H^{-\a}\!\times\!H^\a}}\delta_0
        -{\big\lan}A^{2\a}e_k,\tilde v_n(T)\big\ran^{}_{\ss{\!H^{-\a}\!\times\!H^\a}}\delta_T.
    \]
    We apply the distributional Fourier transform in time and resort to the property \eqref{EQ:properties_Fourier_duality}
    \[
        it{\big\lan}A^{2\a}e_k,\mathscr F[\tilde v_n](t)\big\ran^{}_{\ss{\!H^{-\a}\!\times\!H^\a}}=
        {\big\lan}A^{2\a}e_k,\mathscr F[\tilde f_n](t)
        \big\ran^{}_{\ss{\!H^{-\a}\!\times\!H^\a}}       
        +(2\pi)^{-1/2}
        {\big\lan}A^{2\a}e_k,x_n -\tilde v_n(T)e^{-itT}\big\ran^{}_{\ss{\!H^{-\a}\!\times\!H^\a}}.
    \]
    We multiply both sides of the last equality by $\mathscr F[\psi_k](t)$ and sum over $k$. We obtain for all $t\in\R$
    \begin{align}
        it\big\|\mathscr F[\tilde v_n](t)\big\|^2_{H^{2\a}}=& 
        {\big\lan}A^{2\a}\mathscr F[\tilde v_n](t),\mathscr F[\tilde f_n](t)
        \big\ran^{}_{\ss{\!H^{-\a}\!\times\!H^\a}}
        +(2\pi)^{-1/2}
        {\big\lan}A^{2\a}\mathscr F[\tilde v_n](t),
        x_n
        -\tilde v_n(T) e^{-itT}\big\ran^{}_{\ss{\!H^{-\a}\!\times\!H^\a}}.
    \end{align}
    We take the complex absolute value of both members and estimate the right-hand side
    \begin{align}
    \label{EQ2:first_estimate_Fourier}
        |t|\big\|\mathscr F[\tilde v_n](t)\big\|^2_{H^{2\a}}&\leq
        \big\|\mathscr F[\tilde v_n](t)\big\|_{H^{3\a}}\bigg[
        \sup_{t\in\R}\big\|\mathscr F[\tilde f_n](t)\big\|^2_{H^{\a}}
        +(2\pi)^{-1/2}\big(\|x_n\|^{}_{H^\a}+\|\tilde v_n(T)\|^{}_{H^\a}\big)\bigg]\\
        &\leq c_1\big\|\mathscr F[\tilde v_n](t)\big\|_{H^{3\a}},
    \end{align}
    where we introduced a finite constant $c_1>0$ independent of $n$, thanks to \eqref{EQ:bounded_Ffn}, \eqref{EQ:bounded_xn} and to \eqref{EQ:bounded_vn}.
    A simple real-analysis exercise shows that
    \begin{equation}\label{EQ2:real_analysis}
        |t|^{2\g}\leq 2\frac{1+|t|}{1+|t|^{1-2\g}}, \qquad \forall\, t\in\R, \ \forall\, \gamma\in(0,1/2).
    \end{equation}
    Therefore we have 
    \begin{equation}
    \label{EQ:norm_Hgamma_fourier}
    \begin{aligned}
        \int_\R|t|^{2\g}\big\|\mathscr F[\tilde v_n](t)\big\|^2_{H^{2\a}}\d t &\leq 2\int_\R\frac{\big\|\mathscr F[\tilde v_n](t)\big\|^2_{H^{2\a}}}{1+|t|^{1-2\g}}+\frac{|t|\big\|\mathscr F[\tilde v_n](t)\big\|^2_{H^{2\a}}}{1+|t|^{1-2\g}}\d t \\
        &\leq 2\int_\R{\big\|\mathscr F[\tilde v_n](t)\big\|^2_{H^{3\a}}}+c_1\int_\R\frac{\big\|\mathscr F[\tilde v_n](t)\big\|^{}_{H^{3\a}}}{1+|t|^{1-2\g}}\d t,
    \end{aligned}
    \end{equation}
    where we employed equation~\eqref{EQ2:real_analysis} for the first inequality, and  estimate~\eqref{EQ2:first_estimate_Fourier} for the second, together with the embedding $H^{3\a}\hookrightarrow H^{2\a}$ and with $1+|t|^{1-2\g}\geq 1$.
    We control the first integral by the Plancherel Theorem and the uniform bound in \eqref{EQ:bounded_vn}:
    \begin{equation}\label{EQ2:stima_1_fattore}
    \int_\R\big\|\mathscr F[\tilde v_n](t)\big\|^2_{H^{3\a}}\d t =\|\tilde v_n\|^2_{L^2(\R;H^{3\a})}\leq C,
    \end{equation}
    where $C>0$ is a finite constant independent of $n$.
    Similarly, thanks to the H\"older inequality 
    \begin{equation}
    \label{EQ:second_integral}
        \int_\R\frac{\big\|\mathscr F[\tilde v_n](t)\big\|^{}_{H^{3\a}}}{1+|t|^{1-2\g}}\d t \leq \bigg(\int_\R\frac{1}{(1+|t|^{1-2\g})^2}\d t \bigg)^{1/2}\bigg(\int_\R\big\|\mathscr F[\tilde v_n](t)\big\|^2_{H^{3\a}}\d t \bigg)^{1/2}\leq C',
    \end{equation}
    where $C'>0$ is a finite constant independent of $n$ as soon as $\g<1/4$. 
    Indeed, the second factor is uniformly bounded by equation~\eqref{EQ2:stima_1_fattore}, while the first is finite for $\g<1/4$.
    By plugging estimates~\eqref{EQ2:stima_1_fattore} and \eqref{EQ:second_integral} into \eqref{EQ:norm_Hgamma_fourier} we conclude that 
    \begin{equation}
    \label{EQ:bounded_vn_Hgamma}
        \{\tilde v_n\}_{n\in\N} \quad \text{ is bounded in }\quad  H^\g(\R;H^{2\a}).
    \end{equation}
    
\textit{\textbf{Step }$\mathbf{2}$.}  
    We proved in \textit{Step $1$} that, for any $\gamma\in(0,1/4)$, the sequence $\{\tilde v_n\}_{n\in\N}$ is bounded in $H^\g(\R;H^{2\a})\cap L^\infty(\R;H^{2\a})\cap L^2(\R;H^{3\a})$.
    Moreover $H^{2\a}$ and $H^{3\a}$ are continuously embedded in $H_{loc}^{2\a}$, $H_{loc}^{3\a}$, respectively.
    In addition, \cite[Theorem $2.2$, Chapter $3$]{Temam_2001_Navier-Stokes_equations} states that $H^\g(\R;H^{2\a}_{loc})\cap L^2(\R;H^{3\a}_{loc})$ is compactly embedded into $L^2(\R;H^{2\a}_{loc})$. 
    Therefore
    \[
        \{\tilde v_n\}_{n\in\N}\quad  \text{ is bounded in }\quad H^\g(\R;H^{2\a})\cap L^\infty(\R;H^{2\a})\cap L^2(\R;H^{3\a})\doublehookrightarrow L^2(\R;H^{2\a}_{loc}).
    \]
    Consequently, we find $\tilde v\in H^\g(\R;H^{2\a})\cap L^\infty(\R;H^{2\a})\cap L^2(\R;H^{3\a})$ and a subsequence $\{\tilde v_{n_k}\}_{k\in\N}$ such that 
    \begin{gather}
        \tilde v_{n_k}\longrightarrow \tilde v\qquad \begin{cases}
            \textit{weakly$^\ast$  in }L^\infty(\R;H^{2\a})\\
            \textit{weakly in }L^2(\R;H^{3\a})\\
            \textit{strongly in }L^2(\R;H^{2\a}_{loc})
        \end{cases}, \textit{ as }k\to\infty.
    \end{gather}
    If we restrict to the time interval $[0,T]$ we find
    \begin{gather}\label{EQ:vnk_to_tilde_v_deterministic}
        v_{n_k}=\tilde v_{n_k}\big|^{}_{[0,T]}\longrightarrow \tilde v\big|^{}_{[0,T]}\qquad \begin{cases}
            \textit{weakly$^\ast$  in }L^\infty(0,T;H^{2\a})\\
            \textit{weakly in }L^2(0,T;H^{3\a})\\
            \textit{strongly in }L^2(0,T;H^{2\a}_{loc})\\
            \textit{strongly in }C\big([0,T];H^{2\a}_w\big),
        \end{cases}, \textit{ as }k\to\infty,
    \end{gather}
    where the last convergence is justified in the following. 

    By Theorem~\ref{TH:exisuniqv} and Lemma~\ref{LEM:v_mild_to_strong}, for every $k\in\N$ the function $v_{n_k}$ belongs to $H^1(0,T;H^\a)\cap C^1([0,T];H)$ and satisfies
    \[
        v_{n_k}'(t)=-\nu A^\a v_{n_k}(t)-B\big(v_{n_k}(t)+z(t)\big),
    \]
    for a.e.\ $t\in[0,T]$ in $H^\a$.
    We show that the right-hand side is bounded in $L^2(0,T;H^\a)$.
    Indeed, by Lemma~\ref{LEM:stokes_operator_semigroup}, $\{A^\a v_{n_k}\}_{k\in\N}$ is bounded in $L^2(0,T;H^\a)$ by \eqref{EQ:bounded_vn}. 
    On the other hand, by Lemma~\ref{LEM:estimates_B} with $\s=\a+1$ and by the embedding $H^{2\a}\hookrightarrow H^{\a+1}$, for every $t\in[0,T]$,
    \[
        \|B(v_{n_k}(t)+z(t))\|_{H^\a}
        \le
        c\|v_{n_k}(t)+z(t)\|_{H^\a}\|v_{n_k}(t)+z(t)\|_{H^{\a+1}}
        \le
        c\|v_{n_k}(t)+z(t)\|_{H^{2\a}}^2.
    \]
    Since $\{v_{n_k}\}_{k\in\N}$ is bounded in $C([0,T];H^{2\a})$ and $z\in C([0,T];H^{2\a})$, we deduce that $\{B(v_{n_k}+z)\}_{k\in\N}$ is bounded in $L^\infty(0,T;H^\a)$, hence in $L^2(0,T;H^\a)$. 
    Therefore
    \begin{equation}
    \label{EQ:bounded_vnk_prime}
        \{v_{n_k}'\}_{k\in\N}
        \qquad\text{is bounded in}\qquad
        L^2(0,T;H^\a).
    \end{equation}

    Fix now $y\in H^{-\a}$ and define
    \[
        g_k^y(t):={\big\lan}y,v_{n_k}(t){\big\ran}_{\ss{\!H^{-\a}\!\times\!H^\a}},
        \qquad t\in[0,T].
    \]
    Then $g_k^y\in H^1(0,T)$ and, for a.e.\ $t\in[0,T]$,
    \[
        \frac{\d}{\d t}g_k^y(t)
        =
        {\big\lan}y,v_{n_k}'(t){\big\ran}_{\ss{\!H^{-\a}\!\times\!H^\a}}.
    \]
    Hence, by \eqref{EQ:bounded_vn} and \eqref{EQ:bounded_vnk_prime},
    \[
        \|g_k^y\|_{H^1(0,T)}
        \le
        C\|y\|_{H^{-\a}},
    \]
    for a constant $C>0$ independent of $k$. 
    Since $H^1(0,T)\hookrightarrow C([0,T])$ compactly, the sequence $\{g_k^y\}_{k\in\N}$ is relatively compact in $C([0,T])$.
    On the other hand, the weak$^\ast$ convergence in $L^\infty(0,T;H^{2\a})$ from \eqref{EQ:vnk_to_tilde_v_deterministic} implies that, for every $\varphi\in L^1(0,T)$,
    \[
        \int_0^T g_k^y(t)\varphi(t)\d t
        \longrightarrow
        \int_0^T
        {\big\lan}y,\tilde v(t){\big\ran}_{\ss{\!H^{-\a}\!\times\!H^\a}}
        \varphi(t)\d t.
    \]
    Therefore every subsequence of $\{g_k^y\}_{k\in\N}$ which converges in $C([0,T])$ has the same limit, namely the unique continuous representative of $t\mapsto {\big\lan}y,\tilde v(t){\big\ran}_{\ss{\!H^{-\a}\!\times\!H^\a}}$.
    By the usual contradiction argument, we infer that the whole sequence $\{g_k^y\}_{k\in\N}$ converges in $C([0,T])$, i.e.
    \begin{equation}
    \label{EQ:conv_Cw_H-a}
        \sup_{t\in[0,T]}
        \Big|
        {\big\lan}y,v_{n_k}(t)-\tilde v(t){\big\ran}_{\ss{\!H^{-\a}\!\times\!H^\a}}
        \Big|
        \longrightarrow0,
        \qquad \forall\, y\in H^{-\a}.
    \end{equation}

    We now extend \eqref{EQ:conv_Cw_H-a} to every $y\in H^{-2\a}$ with a standard density argument.
    Fix $y\in H^{-2\a}$ and choose a sequence $\{y_m\}_{m\in\N}$ in $H^{-\a}$ such that $y_m\longrightarrow y$ in $H^{-2\a}$.
    Then, for every $m,k\in\N$,
    \[
    \begin{aligned}
        \sup_{t\in[0,T]}
        \Big|
        {\big\lan}y,v_{n_k}(t)-\tilde v(t){\big\ran}_{\ss{\!H^{-2\a}\!\times\!H^{2\a}}}
        \Big|
        &\le
        \sup_{t\in[0,T]}
        \Big|
        {\big\lan}y_m,v_{n_k}(t)-\tilde v(t){\big\ran}_{\ss{\!H^{-\a}\!\times\!H^\a}}
        \Big| \\
        &\quad+
        \sup_{t\in[0,T]}
        \Big|
        {\big\lan}y-y_m,v_{n_k}(t)-\tilde v(t){\big\ran}_{\ss{\!H^{-2\a}\!\times\!H^{2\a}}}
        \Big|.
    \end{aligned}
    \]
    By duality,
    \[
    \begin{aligned}
        \sup_{t\in[0,T]}
        \Big|
        {\big\lan}y-y_m,v_{n_k}(t)-\tilde v(t){\big\ran}_{\ss{\!H^{-2\a}\!\times\!H^{2\a}}}
        \Big|
        &\le
        \|y-y_m\|_{H^{-2\a}}
        \sup_{t\in[0,T]}\|v_{n_k}(t)-\tilde v(t)\|_{H^{2\a}} \\
        &\le
        C\|y-y_m\|_{H^{-2\a}},
    \end{aligned}
    \]
    where $C>0$ is independent of $k$, thanks to \eqref{EQ:bounded_vn} and \eqref{EQ:vnk_to_tilde_v_deterministic}. 
    Hence, for every fixed $m\in\N$,
    \[
        \limsup_{k\to\infty}
        \sup_{t\in[0,T]}
        \Big|
        {\big\lan}y,v_{n_k}(t)-\tilde v(t){\big\ran}_{\ss{\!H^{-2\a}\!\times\!H^{2\a}}}
        \Big|
        \le
        C\|y-y_m\|_{H^{-2\a}},
    \]
    where we used \eqref{EQ:conv_Cw_H-a} for the first term. 
    Letting $m\to\infty$, we obtain the last convergence in \eqref{EQ:vnk_to_tilde_v_deterministic}.

\textit{\textbf{Step }$\mathbf{3}$.} 
    We now prove that $\tilde v\big|^{}_{[0,T]}=\mathcal V(x,z)\big|^{}_{[0,T]}$ and that this is the limit of the whole sequence $\{v_n\}_{n\in\N}$, in the sense specified in \eqref{EQ:vnk_to_tilde_v_deterministic}. 
    
    With a slight abuse of notation we henceforth denote by $\tilde v$ the restriction on the interval $[0,T]$ of the function $\tilde v$ constructed in the previous step. 
    We also denote by $v$ the restriction to $[0,T]$ of $\mathcal V(x,z)$.
    We know from Theorem~\ref{TH:exisuniqv} that, for every $k\in\N$, the function $v_{n_k}$ satisfies the following identity in $H^\a$
    \[
        v_{n_k}(t)+\nu \int_0^tA^{\a}v_{n_k}(s)\d s +\int_0^tB\big(v_{n_k}(s)+z(s)\big)\d s =x_{n_k}, \qquad \forall\, t\in[0,T].
    \]
    
    We fix $t\in[0,T]$ and pass to the limit, for $k\to\infty$, each term in the previous equality.
    Thanks to $v_{n_k}\longrightarrow \tilde v$ in $C\big([0,T];H_w^{2\a}\big)$ from \eqref{EQ:vnk_to_tilde_v_deterministic}, we have $v_{n_k}(t)\longrightarrow \tilde v(t)$ weakly on $H^{2\a}$, and, recalling that $v_{n_k}(0)=x_{n_k}\longrightarrow x$ in $H^{2\a}$, we also deduce that $\tilde v(0)=x$.
    Analogously, $A^{\a}v_{n_k}\longrightarrow A^{\a}\tilde v$ weakly in $L^2(0,T;H^\a)$, thanks to the weak-convergence in $L^2(0,T;H^{3\a})$ of $v_{n_k}$. Hence, for all $y\in H^{-\a}$
    \[
        \nu\int_0^t{\big\lan}y,A^{\a}v_{n_k}(s)\big\ran^{}_{\ss{\!H^{-\a}\!\times\!H^\a}}\d s \longrightarrow \nu\int_0^t{\big\lan}y,A^{\a}\tilde v(s)\big\ran^{}_{\ss{\!H^{-\a}\!\times\!H^\a}}\d s, \qquad \textit{in }C\big([0,T]\big).
    \]
    
    For the integral of the nonlinearity we apply Lemma~\ref{LEM:nonlinear_loc_convergence}~\itref{ITEM:nonlinear_loc_convergence_strong} (with $\b=\a$) to the sequence $\{v_{n_k}+z\}_{n\in\N}$, which satisfies the hypotheses of the lemma because the strong convergence in $L^2(0,T;H^{2\a}_{loc})$ implies strong convergence in $L^2(0,T;H^{\a+1}_{loc})$ and the weak*-convergence in $L^\infty(0,T;H^{2\a})$ implies boundedness in $L^\infty(0,T;H^{2\a})$, which yields boundedness in $L^2(0,T;H^{\a+1})$. 
    
    We thus find, for all $t\in [0,T]$
    \[
        \tilde v(t)+\nu \int_0^tA^{\a}\tilde v(s)\d s +\int_0^tB\big(\tilde v(s)+z(s)\big)\d s =x,
    \]
    the equality holding in $H^\a$.
    
    We invoke the uniqueness result in Theorem~\ref{TH:exisuniqv} to infer that $\tilde v=v\in C\big([0,T];H^{2\a}\big)\cap L^2(0,T;H^{3\a})$.
    
    We conclude this step by observing that the whole sequence $\{v_n\}_{n\in\N}$ converges to $ v$ in the sense specified in \eqref{EQ:vnk_to_tilde_v_deterministic}. 
    Indeed, let us choose arbitrarily a strictly increasing sequence of indices $\{n_i\}_{i\in\N}$. 
    Then, by \textit{Step 3}, we can extract a subsequence $\{n_{i_k}\}_{k\in\N}$ such that $\big\{v_{n_{i_k}}\big\}_{k\in\N}$ converges to some limit in the sense specified in \eqref{EQ:vnk_to_tilde_v_deterministic}.
    We proved in the present step that the limit is the function $v$ from Theorem~\ref{TH:exisuniqv}, in particular, it is independent of the sequence of indices $\{n_i\}_{i\in\N}\subset \N$. 
    Since the weak*-convergence in $L^\infty(0,T;H^{2\a})$, the weak-convergence in $L^2(0,T;H^{3\a})$ and the strong convergence in $L^2(0,T;H^{2\a}_{loc})$ are all metrizable on bounded sets, we obtain by the usual contradiction argument that the whole sequence $\{v_n\}_{n\in\N}$ converges to $v$ in the specified topologies.
    
\textit{\textbf{Step }$\mathbf{4}$.} 
    Let us fix $y\in H^{-2\a}$, then from the previous steps ${\lan}y,v_n-v\ran^{}_{\ss{\!H^{-2\a}\!\times\!H^{2\a}}}\longrightarrow 0$ in $L^\infty(0,T)$ as $n\to\infty$ and also ${\lan}y,v_n-v\ran^{}_{\ss{\!H^{-2\a}\!\times\!H^{2\a}}}\in C\big([0,T]\big)$ for all $n\in\N$. 
    Hence, we reach
    \[
        {\big\lan}y,v_n-v\big\ran^{}_{\ss{\!H^{-2\a}\!\times\!H^{2\a}}}\longrightarrow 0 \qquad \textit{ in }C\big([0,T]\big), \textit{ as }n\to\infty,
    \]
    which is the sought thesis.
\end{proof}

\begin{theorem}
\label{TH:invariant_measure}
    Assume that $\nu>0$ and $\a>1$.
    There exists an invariant measure for the semigroup $P$, see Definition~\ref{DEF:semigroup_NS}.
    Moreover, if $\mu_\nu\in\mathscr P(H^{2\a})$ is an invariant measure for the semigroup $P$, then
    \begin{align}
    \label{EQ:moment_estimate_p_1}
        &\int_{H^{2\a}}\|x\|^p_{H^{1}}\d \mu_\nu(x)<+\infty, \qquad \forall\, p\in(2,+\infty),\\
    \label{EQ:moment_estimate_4_a}
        &\int_{H^{2\a}}\|x\|^4_{H^{\a}}\d \mu_\nu(x)<+\infty,\\
    \label{EQ:moment_estimate_2_2a}
        &\int_{H^{2\a}}\|x\|^2_{H^{2\a}}\d \mu_\nu(x)<+\infty.
    \end{align}
\end{theorem}
\begin{proof}
 \textit{\textbf{Step }$\mathbf{1}$.} 
    The existence result follows from Theorem~\ref{TH:krylov_bogoliubov_sequentially_weak_Feller}, whose hypotheses we will now verify.
    
    As for the sequentially weak Feller property, let us take $x\in H^{2\a}$ and a sequence $\{x_n\}_{n\in\N}\subset H^{2\a}$ weakly convergent to $x$ in $H^{2\a}$.
    Let $(\Omega,\F, \{\F_t\}_{t\geq 0},\P;W;X^{x_n})$, $n\in\N$, and $(\Omega,\F, \{\F_t\}_{t\geq 0},\P;W;X^{x})$ be solutions to the SHNS$_{\nu,\a}$ Equation~\eqref{EQ:NSnu} such that
    $\P(X^{x_n}_0=x_n)=\P(X^x_0=x)=1$.
    It follows from Theorem~\ref{TH:exi!_solutions_NS_X} \itref{IT:solutions_NS_random} and Lemma~\ref{LEM:deterministic_bw_Feller}  that $\P-a.s.$
    \[
        \lim_{n\to\infty} \varphi(X^{x_n}_t)
        =\lim_{n\to\infty} \varphi\big(\mathcal V(x_n,Z)(t)+Z_t\big)
        =\lim_{n\to\infty} \varphi\big(\mathcal V(x,Z)(t)+Z_t\big)
        =\varphi(X^x_t), \qquad\forall\, t\geq0.
    \]
    Thanks to the boundedness of $\varphi$, the Dominated Convergence Theorem gives
    \[
        \lim_{n\to\infty} (P_t\varphi)(x_n)=\lim_{n\to\infty} \E\big[\varphi(X^{x_n}_t)\big]=\E\big[\varphi(X^{x}_t)\big]=(P_t\varphi)(x), \qquad\forall\, t\geq0.
    \]
    
    As far as the condition \textit{$(ii)$} in Theorem~\ref{TH:krylov_bogoliubov_sequentially_weak_Feller} is concerned, we fix $\e>0$, then for all $T,R>0$, by the Chebyshev inequality and estimate~\eqref{EQ:estim_E_sup_X_2_a}
    \begin{align}
        \frac1T\int_0^T(P_t^\ast\delta_x) \Big(H^{2\a}\setminus \bar B^{H^{2\a}}_R\Big)\d t
        &=\frac1T\int_0^T\P\big(\|X^x_t\|^{}_{H^{2\a}}> R\big)\d t \\
        &\leq \frac{1}{TR^2}\int_0^T\E\|X^x_t\|^{2}_{H^{2\a}}\d t \\
        &= \frac{1}{TR^2}\E\int_0^T\|X^x_t\|^{2}_{H^{2\a}}\d t \\
        &\leq \frac{C}{\nu T R^2}\bigg(\|x\|^2_{H^{\a}}+\frac{1}{\nu^2}\|x\|^2_{H^1}+\frac T\nu+T\nu\bigg),
    \end{align}
    where the constant $C>0$ is independent of $T$ and $\nu$. 
    The hypothesis is verified if we choose $R>0$ such that 
    \[
        \frac{C}{\nu R^2}\bigg(\frac1\nu+\nu\bigg)\leq \e.
    \]

\textit{\textbf{Step }$\mathbf{2}$.} 
    Let us now prove the moment estimates for an invariant measure $\mu_\nu\in\mathscr P(H^{2\a})$ for the semigroup $P$.
    
    Assume that $p>2$ and let, for any $n\in\N$, $\varphi_n:H^{2\a}\ni x\mapsto \|x\|^p_{H^1}\wedge n$. 
    Further assume that $(\Omega,\F, \{\F_t\}_{t\geq 0},\P;W;X^{x})$ is a solution to the SHNS$_{\nu,\a}$ Equation~\eqref{EQ:NSnu} such that $\P(X^x_0=x)=1$.
    For all $n\in\N$, since $\varphi_n\in\mathcal B_b(H^{2\a})$, we have by the invariance property of $\mu_\nu$ and the property \eqref{EQ:int_Pt*} for $P_1^\ast$
    \begin{align}
        \int_{H^{2\a}}\|x\|^p_{H^1}\wedge n\d \mu_\nu(x)
        &=\int_{H^{2\a}}\varphi_n(x)\d (P_1^\ast\mu_\nu)(x) \\
        &=\int_{H^{2\a}}(P_1\varphi_n)(x) \d\mu_\nu(x) \\
        &=\int_{H^{2\a}}\E\big[\|X^x_1\|^p_{H^1}\wedge n\big] \d\mu_\nu(x).
    \end{align}
    In the last equality we used the definition of $P_1$.
    Both the first and last members of this chain of equalities converge by the Monotone Convergence Theorem as $n\to\infty$. By uniqueness of the limit we have
    \begin{align}
        \int_{H^{2\a}}\|x\|^p_{H^1}\d \mu_\nu(x)
        &=\int_{H^{2\a}}\E\big[\|X^x_1\|^p_{H^1}\big] \d\mu_\nu(x)\\
        &\leq \int_{H^{2\a}}C_p+e^{-(p-1)\nu}\|x\|^p_{H^1}\d\mu_\nu(x)\\
        &=C_p+e^{-(p-1)\nu}\int_{H^{2\a}}\|x\|^p_{H^1}\d\mu_\nu(x),
    \end{align}
    where we used the estimate~\eqref{EQ:estim_E_XT_p_1} in Theorem~\ref{TH:exi!_solutions_NS_X} \itref{IT:solutions_NS_estimates}. After rearranging the terms, we reach
    \begin{equation}
    \label{EQ:moment_estimate_p_1_proof}
        \int_{H^{2\a}}\|x\|^p_{H^1}\d \mu_\nu(x)\leq \frac{C_p}{1-e^{-(p-1)\nu}}=:C_{p,\nu}<+\infty.
    \end{equation}

    We argue \textit{mutatis mutandis} to get the other estimates. 
    With a similar argument as before we achieve
    \begin{align}
        \int_{H^{2\a}}\|x\|^4_{H^{\a}}\d \mu_\nu(x)
        &=\int_{H^{2\a}}\E\big[\|X^x_1\|^4_{H^{\a}}\big] \d\mu_\nu(x)\\
        &\leq \int_{H^{2\a}}\bigg(e^{-\nu}\|x\|^4_{H^{\a}}+\frac{C}{\nu^3}\|x\|^8_{H^1}+2C\bigg)\d\mu_\nu(x)\\
        &\leq  e^{-\nu}\int_{H^{2\a}}\|x\|^4_{H^{\a}}\d\mu_\nu(x)+\frac{C}{\nu^3}C_{8,\nu}+2C,
    \end{align}
    where we used the estimate~\eqref{EQ:estim_E_XT_4_a} from Theorem~\ref{TH:exi!_solutions_NS_X} \itref{IT:solutions_NS_estimates}, and the finite constant $C_{p,\nu}$ from equation~\eqref{EQ:moment_estimate_p_1_proof}.
    The thesis follows by rearranging the terms.
    Analogously, making use of the estimate~\eqref{EQ:estim_E_XT_2_2a} from Theorem~\ref{TH:exi!_solutions_NS_X}, we obtain 
    \begin{align}
        \int_{H^{2\a}}\|x\|^2_{H^{2\a}}\d \mu_\nu(x)
        &=\int_{H^{2\a}}\E\big[\|X^x_1\|^2_{H^{2\a}}\big] \d\mu_\nu(x)\\
        &\begin{aligned}
            \leq  e^{-\nu}\int_{H^{2\a}}\|x\|^2_{H^{2\a}}\d\mu_\nu(x)
            &+\frac{C}{\nu^2}\int_{H^{2\a}}\|x\|^4_{H^\a}\d\mu_\nu(x)\\
            &+\frac{C}{\nu^5}\int_{H^{2\a}}\|x\|^8_{H^1}\d\mu_\nu(x)+\frac C\nu+2C,
        \end{aligned}
    \end{align}
    which leads to the thesis by the estimates~\eqref{EQ:moment_estimate_p_1}, and \eqref{EQ:moment_estimate_4_a} just proved and by rearranging the terms.
\end{proof}

\subsection{Construction of marginally stationary solutions}
\label{SEC:stationary_solutions}

Once proved the existence of invariant measures for the stochastic hyperviscous Navier-Stokes Equation, we need to construct a solution whose law is a given invariant measure. 
This result is indeed essential to later study the inviscid limit to the Eulerian case, which will be performed by passing to the limit the marginally stationary solutions in appropriate trajectory spaces.

\begin{proposition}
\label{PROP:construct_stationary_solution}
    Assume that $\a>1$.
    There exists a filtered probability space $(\Omega,\F,\{\F_t\}_{t\geq 0},\P)$ enjoying the usual conditions, with an adapted $H^{2\a}$-valued Wiener process $W$, that satisfies the following property.
    If $\nu>0$ and if $\mu_\nu\in\mathscr P(H^{2\a})$ is an invariant measure for the SHNS$_{\nu,\a}$ Equation~\eqref{EQ:NSnu}, see Definition~\ref{DEF:invariant_measure}, then there exists a process $X^\nu:\R_+\times\Omega\to H^{2\a}$ such that:
    \begin{itemize}
        \item $(\Omega,\F,\{\F_t\}_{t\geq 0},\P;W;X^\nu)$ is a solution to the SHNS$_{\nu,\a}$ Equation~\eqref{EQ:NSnu}, see Definition~\ref{DEF:solution_NSnu},
        \item for all $T>0$
        \begin{equation}
        \label{EQ:regularities_stationary}
            \!\!\!\!\!\!\!\!\!\!\!\!X^\nu\in \ds
            \Bigg(\bigcap_{p\geq 1}L^p\big(\Omega;C\big([0,T];H^1\big)\big)\bigg)
            \cap L^4\big(\Omega;C\big([0,T];H^{\a}\big)\big)
            \cap L^2\big(\Omega;C\big([0,T];H^{2\a}\big)\big)
                \cap L^2\big(\Omega\times[0,T];H^{3\a}\big),
         \end{equation}
        \item $(X^\nu_0)_\ast\P=\mu_\nu\in\mathscr P(H^{2\a})$.
    \end{itemize}
    In particular, $X^\nu$ is marginally stationary, \textit{i.e.}:
    \[
        (X^\nu_t)_\ast\P=\mu_\nu\in\mathscr P(H^{2\a}), \qquad \forall\, t\geq  0.
    \]
\end{proposition}
\begin{proof}
\textit{\textbf{Step }$\mathbf{1}$.}
    For fixed $\a>1$, we will construct a filtered probability space $(\Omega,\F,\{\F_t\}_{t\geq 0},\P)$ that satisfies the usual conditions and enjoys the following properties.
    \begin{itemize}
    \item 
        There exists an $H^{2\a}$-valued Wiener process $W$ defined and adapted on $(\Omega,\F,\{\F_t\}_{t\geq 0},\P)$.
    \item 
        For any $\nu>0$, if $\mu_\nu\in\mathscr P(H^{2\a})$ is an invariant measure for the SHNS$_{\nu,\a}$ Equation~\eqref{EQ:NSnu} (which exists thanks to Theorem~\ref{TH:invariant_measure}), then there exists a random variable $\xi_\nu:\Omega\to H^{2\a}$ distributed as $\mu_\nu$. 
    \end{itemize}
    
    First, we define $\Omega^1:=(H^{2\a})^{(0,+\infty)}$, we fix, for any $\nu>0$, an invariant measure $\mu_\nu\in\mathscr P(H^{2\a})$ for the SHNS$_{\nu,\a}$ Equation~\eqref{EQ:NSnu}, and we let $\bar \xi_\nu:\Omega^1\ni \omega\mapsto \omega(\nu)\in H^{2\a}$, for all $\nu>0$.
    Then, \cite[Exercise $2$, Section $10.6$]{Cohn_2013_Measure_Theory_2nd_edition} states the existence of a $\sigma$-algebra $\F^1$ on $\Omega^1$ and a probability measure $\P^1$ on $(\Omega^1,\F^1)$ such that $\bar \xi_\nu$ is a random variable on $(\Omega^1,\F^1,\P^1)$ distributed as $\mu_\nu$, for any $\nu>0$.
    Next, we endow $(\Omega^1,\F^1,\P^1)$ with the constant filtration $\F^1_t:=\F^1$ for all $t\geq 0$.
    In addition, we consider a filtered probability space $\big(\Omega^2,\F^2,\{\F^2_t\}_{t\geq 0},\P^2\big)$ with an adapted $H^{2\a}$-valued Wiener process $\bar W$. 

    Finally, we define $(\Omega,\F,\P)$ as the completion of $(\Omega^1\times \Omega^2,\F^1\otimes\F^2,\P^1\times \P^2)$ and, for all $t\geq 0$, we let $\F_t$ be the completion with respect to $\P$ of $\bigcap_{s>t}\F^1_s\otimes \F^2_s$.
    The so constructed filtered probability space $(\Omega,\F,\{\F_t\}_{t\geq 0},\P)$ satisfies the usual conditions.

    Eventually, recalling that $\Omega:=\Omega^1\times\Omega^2$, we let
    \begin{align}
        W_t:\Omega\ni (\omega_1,\omega_2)\mapsto \bar W_t(\omega_2)\in H^{2\a}, \qquad \forall\, t\geq 0,\\
        \xi_\nu:\Omega\ni (\omega_1,\omega_2)\mapsto \bar \xi_\nu(\omega_1)\in H^{2\a}, \qquad \forall\, \nu> 0.
    \end{align}
    Then $W$ is an adapted $H^{2\a}$-valued Wiener process on $(\Omega,\F,\{\F_t\}_{t\geq 0},\P)$, and $\xi_\nu$ is $\F_0$-measurable for every $\nu>0$.
    Moreover,
    \[
        (\xi_\nu)_\ast\P=(\bar\xi_\nu)_\ast\P^1=\mu_\nu,
        \qquad \forall\, \nu>0.
    \]
    
\textit{\textbf{Step }$\mathbf{2}$.}
    Let us now fix $\nu>0$. 
    Since $\xi_\nu:\Omega\to H^{2\a}$ is $\F_0$-measurable, Theorem~\ref{TH:exi!_solutions_NS_X} \itref{IT:solutions_NS_random} states the existence of a process ${X^{\xi_\nu}:\R_+\times\Omega\to H^{2\a}}$, that we hereby denote simply by $X^\nu$, such that $(\Omega,\F,\{\F_t\}_{t\geq 0},\P;W;X^\nu)$ is a solution to the SHNS$_{\nu,\a}$ Equation~\eqref{EQ:NSnu} with $\P\big(X^\nu_0=\xi_\nu\big)=1$.
    In addition, since  $(\xi_\nu)_\ast\P=(X^\nu_0)_\ast\P=\mu_\nu$, the moment estimates~\eqref{EQ:moment_estimate_p_1}, \eqref{EQ:moment_estimate_4_a}, \eqref{EQ:moment_estimate_2_2a} imply $X^\nu_0\in L^2(\Omega;H^{2\a})\cap L^4(\Omega;H^{\a})\cap L^p(\Omega;H^1)$ for all $p\geq 1$.
    Hence, by Theorem~\ref{TH:exi!_solutions_NS_X} \itref{IT:solutions_NS_regularities}, the process $X^\nu$ has the sought regularities.    
    
    It only remains to show that $X^\nu$ is marginally stationary.
    Let us fix $t\geq 0$, then $(X^\nu_t)_\ast\P=\mu_\nu\in \mathscr P(H^{2\a})$ if and only if
    \begin{equation}
    \label{EQ:construct_stationary_solution}
        \int_{H^{2\a}}\varphi \d \mu_\nu=\int_{H^{2\a}}\varphi \d \big((X^\nu_t)_\ast\P\big), \qquad \forall\, \varphi\in \mathcal B_b(H^{2\a}).
    \end{equation}
    Let us fix $\varphi\in \mathcal B_b(H^{2\a})$.
    The following equalities come from the invariance property of $\mu_\nu$, see Definition~\ref{DEF:invariant_measure}, the property \eqref{EQ:int_Pt*} for $P_t$, the Change of Variable Theorem and the definition of $P_t$ in equation~\eqref{EQ:def_markov_semigroup}:
    \begin{align}
        \int_{H^{2\a}}\varphi \d \mu_\nu
        &=\int_{H^{2\a}}\varphi \d \big(P_t^\ast\mu_\nu\big)\\
        &=\int_{H^{2\a}}P_t\varphi \d \mu_\nu\\
        &=\int_{H^{2\a}}P_t\varphi \d \big((\xi_\nu)_\ast\P\big)\\
        &=\int_{\Omega}(P_t\varphi)\big((\xi_\nu)(\omega)\big) \d \P(\omega)\\
        &=\int_{\Omega}\int_{\Omega} \varphi\Big(X^{\xi_\nu(\omega)}_t(\omega')\Big) \d\P(\omega')\d\P(\omega).
    \end{align}
    Let us define 
    \begin{align}
        \tilde X_s&: \Omega\times \Omega \ni(\omega,\omega')\mapsto \tilde X_s(\omega,\omega'):=X_s^{\xi_\nu(\omega)}(\omega')\in H^{2\a}, \qquad \forall\, s\geq  0,\\
        \tilde W_s&: \Omega\times \Omega \ni(\omega,\omega')\mapsto \tilde W_s(\omega,\omega'):=W_s(\omega')\in H^{2\a}, \qquad \forall\, s\geq 0,
    \end{align}
    and consider the augmented filtered probability space $(\Omega\times\Omega, \tilde\F, \{\tilde\F_s\}_{s\geq0},\tilde\P)$ constructed, as already outlined in \textit{Step $1$}, from $(\Omega\times\Omega, \F\otimes\F, \{\F_s\otimes\F_s\}_{s\geq0},\P\otimes\P)$.
    We denote $\tilde \Omega:=\Omega\times \Omega$.
    In particular, $\tilde W$ is an adapted $H^{2\a}$-valued Wiener process.
    Moreover, by direct inspection, $(\tilde\Omega, \tilde\F, \{\tilde\F_s\}_{s\geq0},\tilde\P;\tilde W;\tilde X)$ is a solution to the SHNS$_{\nu,\a}$ Equation~\eqref{EQ:NSnu} such that $\tilde\P(\tilde X_0=\xi_\nu)=1$.
    In particular, $(\tilde X_0)_\ast\tilde\P=(\xi_\nu)_\ast\P$, thus the uniqueness in law for the solution, see Theorem~\ref{TH:exi!_solutions_NS_X} \itref{IT:solutions_NS_uniqueness}, yields 
    \[
        (\tilde X_s)_\ast \tilde\P=(X^\nu_s)_\ast\P\in\mathscr P(H^{2\a}),\qquad \forall\, s\geq 0.
    \]
    Therefore, we can conclude the chain of equalities above with
    \begin{align}
        \int_{H^{2\a}}\varphi \d \mu_\nu
        &=\int_{\tilde\Omega}\varphi\big(\tilde X_t(\omega,\omega')\big)\d \tilde\P(\omega,\omega')\\
        &=\int_{\tilde\Omega}\varphi\d \big((\tilde X_t)_\ast \tilde\P\big)\\
        &=\int_{\Omega}\varphi\d \big((X^\nu_t)_\ast\P\big).
    \end{align}
    We obtained the sought equality \eqref{EQ:construct_stationary_solution} and the claim follows.
\end{proof}

Eventually, the marginally stationary solutions provided in the previous theorem, also allow us to obtain further estimates for the invariant measure, as shown in the following result.

\begin{theorem}
\label{TH:moment_estimates_mu_nu}
    Assume that $\a>1$. 
    There exists a finite constant $C_\a>0$ such that, assuming $\nu>0$, if $\mu_\nu\in\mathscr P(H^{2\a})$ is an invariant measure for the SHNS$_{\nu,\a}$ Equation~\eqref{EQ:NSnu}, see Definition~\ref{DEF:invariant_measure}, then
    \begin{align}
    \label{EQ:moment_mu_nu_2_a+1}
        &\int_{H^{2\a}}\|x\|^2_{H^{\a+1}}\d \mu_\nu(x)=C_\a,\\
    \label{EQ:moment_mu_nu_2n_H1}
        &\int_{H^{2\a}}\|x\|^{2n}_{H^{1}}\d \mu_\nu(x)\leq (2n-1)!!\, C_\a^n, \qquad \forall\, n\in\N.    
    \end{align}
    Moreover, if $\b>0$ is such that $2\b C_\a<1$,  then
    \begin{equation}
    \label{EQ:moment_mu_nu_exp_Ha+1}
        \int_{H^{2\a}} e^{\b\|x\|^2_{H^{\a+1}}}\d \mu_\nu(x)\leq 2e^{\frac{2\b C_\a}{1-2\b C_\a}}.
    \end{equation}
\end{theorem}
\begin{proof}
    Fix $\nu>0$ and $\a>1$.
    If $\mu_\nu\in\mathscr P(H^{2\a})$ is an invariant measure for the SHNS$_{\nu,\a}$ Equation~\eqref{EQ:NSnu}, let $(\Omega,\F,\{\F_t\}_{t\geq 0},\P;W;X^\nu)$ be the solution to the SHNS$_{\nu,\a}$ Equation~\eqref{EQ:NSnu} from Proposition~\ref{PROP:construct_stationary_solution} that satisfies 
    \[
        (X^\nu_t)_\ast\P=(X^\nu_0)_\ast\P=\mu_\nu\in\mathscr P(H^{2\a}), \qquad \forall\, t\geq 0.
    \]
    
\textbf{\textit{Equation} (\ref{EQ:moment_mu_nu_2_a+1}).}
    We apply the It\^o formula in Lemma~\ref{LEM:Ito_formula_h(|x|^2_b)} to the marginally stationary process $X^\nu$ and with the choice $\g=1$ and $h=\id_{\R_+}$.
    The last but one term in equation~\eqref{EQ:ito_X_xi_h} vanishes, while the last is deterministic. 
    Also, the nonlinearity vanishes because of the property $\lan B(x),x\ran^{}_{\! H^1}=\lan B(x), Ax\ran=0$, for all $x\in H^2$.
    Next, we  take the  expectation on both members of the equality and obtain
    \begin{equation}
        \E\|X^\nu_t\|^2_{H^1}
        +2\nu \E\int_0^t\|X^\nu_s\|^2_{H^{\a+1}}\d s=\E\|X^\nu_0\|^2_{H^1}
        + {\nu}\operatorname{Tr}[Q_1]t, \qquad \forall\, t\geq 0,
    \end{equation}
    with $Q_1:=\iota Q\iota^\ast\in\mathcal L_1(H^1)$, where $Q\in\mathcal L_1(H^{2\a})$ is the covariance operator of $W_1$, and $\iota:H^{2\a}\to H^1$ is the Sobolev embedding.
    Since $X^\nu$ is marginally stationary, the first terms on both sides are equal and cancel out. 
    We apply Fubini's Theorem and rewrite
    \begin{equation}
        \int_0^t\nu\Big(2\E\|X^\nu_s\|^2_{H^{\a+1}}-\operatorname{Tr}[Q_1]\Big)\d s=0, \qquad \forall\,  t\geq 0.
    \end{equation}
    Since $\nu>0$ this last equality implies for \textit{a.e.} $s\geq 0$
    \[
        \int_{H^{2\a}}\|x\|^2_{H^{\a+1}}\d \mu_\nu(x)=\E\|X^\nu_s\|^2_{H^{\a+1}}=\frac{\operatorname{Tr}[Q_1]}{2}.
    \]
    The claim is proved with $C_\a:=\Tr[Q_1]/2$.
    
\textbf{\textit{Equation} (\ref{EQ:moment_mu_nu_2n_H1}).} 
    We prove the estimate~\eqref{EQ:moment_mu_nu_2n_H1} by induction on $n\in\N$.
    
    The base case follows from the previous part and the Sobolev embedding $\|x\|^{}_{H^1}\leq \|x\|^{}_{H^{\a+1}}$ for $x\in H^{\a+1}$.
    Let us prove the inductive step. 
    We apply the It\^o formula in Lemma~\ref{LEM:Ito_formula_h(|x|^2_b)} to the marginally stationary process $X^\nu$ and with the choices $\g=1$ and $h:\R_+\ni r\mapsto r^{n+1}\in\R_+$.
    The nonlinearity vanishes because of the well-known property $\lan B(x),x\ran^{}_{\!H^1}=\lan B(x), Ax\ran=0$, for all $x\in H^2$.
    Next, we  take the  expectation on both members of the equation~\eqref{EQ:ito_X_xi_h} and obtain for all $t\geq 0$
    \begin{align}
        &\E\|X^\nu_t\|^{2n+2}_{H^1}
        +2(n+1)\nu \, \E\int_0^t\|X^\nu_s\|^2_{H^{\a+1}}\|X^\nu_s\|^{2n}_{H^{1}}\d s\\
        =&\;\E\|X^\nu_0\|^{2n+2}_{H^1} +2{n(n+1)}\nu\, \E\int_0^t\big\|Q_1^{1/2}X^\nu_s\big\|^2_{H^1}\|X^\nu_s\|^{2(n-1)}_{H^1}\d s +(n+1){\nu}\Tr[Q_1]\E\int_0^t\|X^\nu_s\|^{2n}_{H^1}\d s\\
        \leq &\;\E\|X^\nu_0\|^{2n+2}_{H^1} +{(n+1)(2n+1)}\nu\Tr[Q_1]\E\int_0^t\|X^\nu_s\|^{2n}_{H^1}\d s,
    \end{align}
    where we used $\|Q_1^{1/2}x\|\leq \Tr[Q_1]\|x\|^{}_{H^1}$ for all $x\in H^1$, see equation~\eqref{EQ:control_Tr_Ito}.
    Since $X^\nu$ is marginally stationary, the first terms on both sides of the inequality are equal and cancel out. 
    We apply Fubini's Theorem, divide by $2(n+1)\nu$ and rewrite
    \begin{equation}
        \int_0^t\bigg(\E\big[\|X^\nu_s\|^2_{H^{\a+1}}\|X^\nu_s\|^{2n}_{H^1}\big]-(2n+1)\frac{\operatorname{Tr}[Q_1]}{2}\E\|X^\nu_s\|^{2n}_{H^1}\bigg)\d s\leq 0, \qquad \forall\,  t\geq 0.
    \end{equation}
    This last inequality implies, recalling that $(X^\nu_s)_\ast\P=\mu_\nu\in\mathscr P(H^{2\a})$ for all $s\geq 0$ and the Sobolev embedding $\|x\|^{2}_{H^1}\leq \|x\|^{2}_{H^{\a+1}}$ for $x\in H^{\a+1}$
    \begin{align}
        \int_{H^{2+\de}}\|x\|^{2n+2}_{H^{1}}\d \mu_\nu(x)
        &\leq \int_{H^{2\a}}\|x\|^{2}_{H^{\a+1}}\|x\|^{2n}_{H^{1}}\d \mu_\nu(x)\\
        &=\E\Big[\|X^\nu_s\|^2_{H^{\a+1}}\|X^\nu_s\|^{2n}_{H^{1}}\Big]\\
        &\leq (2n+1)\frac{\operatorname{Tr}[Q_1]}{2}\E\|X^\nu_s\|^{2n}_{H^1}\\
        &= (2n+1)\frac{\operatorname{Tr}[Q_1]}{2}\int_{H^{2+\de}}\|x\|^{2n}_{H^{1}}\d \mu_\nu(x)\\
        &\leq (2n+1)!!C_\a^{n+1},
    \end{align}
    where we used the inductive hypothesis in the last step and recalled, from the first part of the proof, that $C_\a:=\Tr[Q_1]/2$.

\textbf{\textit{Equation} (\ref{EQ:moment_mu_nu_exp_Ha+1}).}   
    Fix now $\b>0$, and apply the It\^o formula in Lemma~\ref{LEM:Ito_formula_h(|x|^2_b)} with the choices $\g=1$, $h:\R_+\ni r\mapsto h(r)=e^{\b r}$ and to the marginally stationary process $X^\nu$. 
    After taking the expectation to both sides of equation~\eqref{EQ:ito_X_xi_h}, we obtain for all $t\geq 0$
    \begin{align}
        \E\,  e^{\b\|X^\nu_t\|^2_{H^1}}=
        \E\,  e^{\b\|X^\nu_0\|^2_{H^1}}
        &-2\b\nu \,\E\int_0^te^{\b\|X^\nu_s\|^2_{H^1}}\|X^\nu_s\|^2_{H^{\a+1}}\d s\\
        &-2\b\, \E\int_0^te^{\b\|X^\nu_s\|^2_{H^1}}\lan B(X^\nu_s), X^\nu_s\ran^{}_{\! H^1}\d s\\
        &+2\b\sqrt \nu\, \E\int_0^t e^{\b\|X^\nu_s\|^2_{H^1}}\lan X^\nu_s, \d W_s\ran^{}_{\! H^1}\\
        &+2\b^2\nu\E\int_0^t\big\|Q_1^{1/2}X^\nu_s\big\|^2_{H^1}e^{\b\|X^\nu_s\|^2_{H^1}}\d s\\
        &+\b\nu \Tr[Q_1]\E\int_0^te^{\b\|X^\nu_s\|^2_{H^1}}\d s.
    \end{align}
    Due to the marginal stationarity of the process, the left-hand side and the first term on the right-hand side are equal. 
    The integral of the nonlinearity vanishes, because of the property ${\lan B(x), x\ran^{}_{\! H^1}=\lan B(x), Ax\ran =0}$ for all $x\in H^2$. 
    The expectation of the It\^o integral is $0$ by the properties of It\^o integration.
    We employ Fubini's Theorem and the marginal stationarity of the process to compute the time integral in all the other terms.
    We have for all $t\geq 0$
    \begin{align}
        0=
        -2\b\nu \,t\E\big[e^{\b\|X^\nu_0\|^2_{H^1}}\|X^\nu_0\|^2_{H^{\a+1}}\big]
        +2\b^2\nu\, t\E\Big[\big\|Q_1^{1/2}X^\nu_0\big\|^2_{H^1}e^{\b\|X^\nu_0\|^2_{H^1}}\Big]
        +\b\nu \Tr[Q_1]\, t\E\big[e^{\b\|X^\nu_0\|^2_{H^1}}\big].
    \end{align}
    Let us divide by $\b\nu t$, for $t>0$, and rewrite as follows
    \begin{equation}
    \label{EQ:stima_exp}
    \begin{aligned}
        0&=\E\Big[e^{\b\|X^\nu_0\|^2_{H^1}}\big(\Tr[Q_1]+2\b\big\|Q_1^{1/2}X^\nu_0\big\|^2-2\|X^\nu_0\|^2_{H^{\a+1}}\big)\Big]\\
        &\leq  \E\Big[e^{\b\|X^\nu_0\|^2_{H^1}}\big(\Tr[Q_1]+2\b\Tr[Q_1]\|X^\nu_0\|^2_{H^{1}}-2\|X^\nu_0\|^2_{H^{\a+1}}\big)\Big]\\
        &\leq \E\Big[e^{\b\|X^\nu_0\|^2_{H^{\a+1}}} \Big(\Tr[Q_1]+2\big(\b\Tr[Q_1]-1\big)\|X^\nu_0\|^2_{H^{\a+1}}\Big)\Big]\\
        &= 2\, \E\Big[e^{\b\|X^\nu_0\|^2_{H^{\a+1}}}\Big(\Tr[Q_1]-\big(1-\b\Tr[Q_1]\big)\|X^\nu_0\|^2_{H^{\a+1}}\Big)\Big]-\Tr[Q_1]\E e^{\b\|X^\nu_0\|^2_{H^{\a+1}}},
    \end{aligned}
    \end{equation}
    where for the first inequality we used $\|Q_1^{1/2}x\|^2_{H^1}\leq \|x\|^2_{H^1}\Tr[\Pi_xQ_1]$ for all $x\in H^{1}$, see equation~\eqref{EQ:control_Tr_Ito}, while for the second we used the Sobolev embedding $H^1\hookrightarrow H^{\a+1}$.
    Let us baptise the real random variable in the big round brackets as 
    \[
        J:=\Tr[Q_1]-\big(1-\b\Tr[Q_1]\big)\|X^\nu_0\|^2_{H^{\a+1}},
    \]
    and assume that $1-\b \Tr[Q_1]>0$ (\textit{i.e.} $2\b C_\a<1$), so that $\P(J<\Tr[Q_1])=1$.
    We reach from \eqref{EQ:stima_exp}
    \begin{align}
        \Tr[Q_1]\E e^{\b\|X^\nu_0\|^2_{H^{\a+1}}}
        &\leq 2\E\Big[e^{\b\|X^\nu_0\|^2_{H^{\a+1}}}J\Big]\\
        &\leq 2\E\Big[e^{\b\|X^\nu_0\|^2_{H^{\a+1}}}J\1_{\{J\geq 0\}}\Big]\\&\leq 2\Tr[Q_1]\E\Big[e^{\b\|X^\nu_0\|^2_{H^{\a+1}}}\1_{\{J\geq 0\}}\Big],
    \end{align}
    where we used $J=J\1_{\{J\geq 0\}}+J\1_{\{J<0\}}\leq J\1_{\{J\geq 0\}}\leq \Tr[Q_1]\1_{\{J\geq 0\}}$, $\P-a.s.$
    Observe that
    \[
        J\geq 0\quad\Longleftrightarrow\quad \|X^\nu_0\|^2_{H^{\a+1}}\leq\frac{\Tr[Q_1]}{1-\b\Tr[Q_1]},
    \]
    hence 
    \[
        \Tr[Q_1]\E e^{\b\|X^\nu_0\|^2_{H^{\a+1}}}\leq 2\Tr[Q_1]\E\Big[e^{\b\|X^\nu_0\|^2_{H^{\a+1}}}\1_{\{J\geq 0\}}\Big]\leq 2\Tr[Q_1]e^{\b\frac{\Tr[Q_1]}{1-\b\Tr[Q_1]}}\P(J\geq 0).
    \]
    We divide by $\Tr[Q_1]$ and use a change of variable, recalling that $X^\nu_0$ is distributed as $\mu_\nu$
    \[
        \int_{H^{2\a}}e^{\b\|x\|^2_{H^{\a+1}}}\d\mu_\nu(x)=\E e^{\b\|X^\nu_0\|^2_{H^{\a+1}}}\leq 2e^{\b\frac{\Tr[Q_1]}{1-\b\Tr[Q_1]}}.
    \]
    The sought estimate~\eqref{EQ:moment_mu_nu_exp_Ha+1} follows.
\end{proof}

\section{Invariant measure for the deterministic Euler Equation}
\label{SEC:stationary_solution_EE}

In this section, we study the inviscid limit of the stochastic hyperviscous Navier-Stokes equations. 
The passage to the limit, as $\nu\to 0^+$, for the marginally stationary solutions to SHNS$_{\nu,\a}$ Equation~\eqref{EQ:NSnu}, see Proposition~\ref{PROP:construct_stationary_solution}, will be conducted in the spirit of the Skorokhod theorem.
Hence, it is necessary to demonstrate some tightness property. 
However, the usual trajectory spaces for the solution have topologies that are too large to prove tightness. 
A suitable trajectory space, which allows for the proof of a tightness result and the application of Jakubowski's generalization of the Skorokhod theorem, is introduced in the next definition.

\begin{definition}
\label{DEF:Z_T_H^a}
    Assume that $\a>1$ and let $W$ be an $H^{2\a}$-valued Wiener process. 
    Assuming $T>0$, we define:
        \[
            \mathcal Z_T:=L^2_{w}(0,T;H^{\a+1})\cap L^2(0,T;H^{\a}_{loc})\cap C\big([0,T];H^{\a}_w\big)\cap C\big([0,T];U'\big),
        \]
        where 
    \begin{itemize}
    \item 
        $U$ is the reproducing kernel of $W$. 
        In particular, $U$ is compactly embedded into $H^{2\a}$ and simultaneously $H^{-2\a}$ is compactly embedded into $U'$.
        \[
        U\doublehookrightarrow 
        H^{2\a}\hookrightarrow 
        H=H'\hookrightarrow 
        H^{-2\a}\doublehookrightarrow U'.
        \]
    \item
        $L^2_w(0,T;H^{\a+1})$ is the linear space $L^2(0,T;H^{\a+1})$ endowed with the weak topology, namely the smallest  topology that makes all the linear maps $L^2_w(0,T;H^{\a+1})\to\R$ continuous.
        In particular, a sequence of functions $\{v_n\}_{n\in\N}\subset L^2(0,T;H^{\a+1})$ converges in $L^2_w(0,T;H^{\a+1})$ if there exists $v\in L^2(0,T;H^{\a+1})$ such that 
        \[
            \lim_{n\to\infty}\int_0^T\!\!\! {\big{\lan}} \phi(t),v_n(t)-v(t) \big{\ran}^{}_{\ss{\!H^{-\a-1}\!\times\!H^{\a+1}}} \d t =0, \qquad \forall\, \phi\in L^2(0,T;H^{-\a-1}).
        \]
    \item
        $C\big([0,T];H^{\a}_w\big)$ is the linear space of continuous functions $v:[0,T]\to H^{\a}_w$, where $H^{\a}_w$ is the space $H^{\a}$ endowed with weak topology, see Notation \ref{NOT:weak_topologies}. 
        Since the domain $[0,T]$ is a sequential space, the notion of continuity is equivalent to that of sequential continuity, see Notation \ref{NOT:sequential_spaces}. 
        Namely, a function $v:[0,T]\to H^\a$ belongs to $C\big([0,T];H^{\a}_w\big)$ if for any sequence of times $\{t_n\}_{n\in\N}\subset [0,T]$ convergent to some $t\in[0,T]$,
        \[
            \lim_{n\to\infty}{\big{\lan}} y,v(t_n)-v(t)\big\ran^{}_{\ss{\!H^{-\a}\!\times\!H^{\a}}}=0, \qquad \forall\, y\in H^{-{\a}}.
        \]
        This space is endowed with the compact-open topology, that is the smallest  topology that contains all the sets $\big\{f\in C\big([0,T];H^{\a}_w\big) \ : \ f(K)\subset V\big\}$, where $K$ varies among compact subsets of $[0,T]$, and $V$ among open subsets of $H^{\a}_w$.
    \end{itemize}
    Fix a sequence $\{\chi_n\}_{n\in\N}$ in $C_c^\infty(\R^2)$ such that
    \[
        0\leq \chi_n\leq 1,
        \qquad
        \chi_n=1 \ \text{on }\bar B_n,
        \qquad
        \supp \chi_n\subset B_{n+1},
        \qquad \forall\, n\in\N,
    \]
    where $\bar B_n:=\{\xi\in\R^2:\ |\xi|\le n\}$.
    \begin{itemize}
        \item 
            $H^\a_{loc}$ is the vector space $H^\a$ endowed with the Fr\'echet topology generated by the family of seminorms
            \begin{equation}
            \label{EQ:def_semin_n}
                H^\a\ni x\mapsto [\![x]\!]_{H^\a_n}
                :=\|\chi_n x\|_{H^\a(\R^2;\R^2)},
                \qquad \forall\, n\in\N.
            \end{equation}
            Since multiplication by $\chi_n$ is a bounded linear map on $H^\a(\R^2;\R^2)$, these seminorms are well-defined.
        \item
            $L^2(0,T;H^\a_{loc})$ is the linear space of all measurable functions $v:[0,T]\to H^\a$ such that, for every $n\in\N$, the function $t\mapsto [\![v(t)]\!]_{H^\a_n}$ belongs to $L^2(0,T)$.
            It is endowed with the Fr\'echet topology generated by the seminorms
            \begin{equation}
            \label{EQ:def_seminorm_pTn}
                L^2(0,T;H^\a_{loc})\ni v\mapsto
                [\![v]\!]_{L^2(0,T;H^\a_n)}
                :=\bigg(\int_0^T \|\chi_n v(t)\|^2_{H^\a(\R^2;\R^2)}\d t\bigg)^{1/2},
                \qquad \forall\, n\in\N.
            \end{equation}
        \end{itemize}
    We endow $\mathcal Z_T$ with the smallest topology that makes the natural embeddings from $\mathcal Z_T$ into each of its four components, continuous (\textit{i.e.} the smallest topology which contains the topologies of all the four components).
\end{definition}

\begin{lemma}
\label{LEM:Borel_sets_in_ZT} 
    Assume that $\a>1$ and let $W$ be an $H^{2\a}$-valued Wiener process.
    If $\b\geq -2\a$ and $T>0$, then $C\big([0,T];H^\b\big)\cap \mathcal Z_T\in \mathscr B_{\mathcal Z_T}$.
\end{lemma}
\begin{proof}
    From the embeddings $H^{\b}\hookrightarrow H^{-2\a}\hookrightarrow U'$, we infer that that $C\big([0,T];H^\b\big)$ is continuously embedded into $C\big([0,T];U'\big)$. 
    Therefore, by Kuratowski's theorem \cite[Theorem $1$, Section V, Chapter $39$]{Kuratowski_1966_topology_1}, see also \cite[Theorem $15.1$]{Kechris_1995_classical_descriptive_set_theory}, the space $C\big([0,T];H^\b\big)$ is a Borel subset of $C\big([0,T];U'\big)$. 
    Eventually, the intersection $C\big([0,T];H^\b\big)\cap \mathcal Z_T$ is a Borel subset of $C\big([0,T];U'\big)\cap \mathcal Z_T=\mathcal Z_T$.
\end{proof}

The following tightness criterion and the subsequent lemma are taken from \cite[Section $3.2$]{Brzezniak+Motyl_2013_Existence_Martingale_Solution}, where the results are stated in a general abstract setting.
The abstract Hilbert spaces $U$, $V$ and $H$ in the reference are replaced in our setting by  $U$, $H^{\a+1}$ and $H^{\a}$, respectively.
\begin{theorem}[Tightness criterion]
\label{TH:tight_crit}
    Assume that $\a>1$ and let $(\Omega,\F,\{\F_t\}_{t\geq 0},\P)$ be an augmented filtered probability space with an adapted $H^{2\a}$-valued Wiener process $W$.
    Let $U$ denote the reproducing kernel of $W$.
    Let also $T>0$ and $X^n:[0,T]\times\Omega\to U'$, $n\in\N$, be adapted and pathwise continuous stochastic processes. 
    Assume  that 
    \begin{align}
    \label{EQ:HP_bound_tight_crit}
        \sup_{n\in\N} \,\E\Bigg[\sup_{t\in[0,T]}\|X^n_t\|^{2}_{H^{\a}} +\int_0^T\|X^n_t\|_{H^{\a+1}}^2\d t\Bigg] <+\infty.
    \end{align}
    Assume also that the sequence $\{X^n\}_{n\in\N}$ satisfies the Aldous condition in $U'$, \textit{i.e.}, for every $\e>0$ and $\eta>0$, there exists $\theta\in(0,T]$ such that, for every sequence $\{\tau_n\}_{n\in\N}$ of $[0,T]$-valued stopping times, we have:
    \begin{equation}
        \sup_{n\in\N}\sup_{t\in[0,\theta]}\P\big(\|X^n_{t+\tau_n}-X^n_{\tau_n}\|^{}_{U'}\geq \eta\big)\leq \e.
    \end{equation}
    Then the laws on $(\mathcal Z_T,\mathscr B_{\mathcal Z_T})$ of the stochastic processes form a tight sequence, \textit{i.e.}, for any $\e>0$ there exists a compact set $K_\e\subset \mathcal Z_T$ such that 
    \[
        \sup_{n\in\N}\P (X^n\in K_\e)\geq 1-\e.
    \]
\end{theorem}

The space $\mathcal{Z}_T$ is constructed \textit{ad hoc} to guarantee the thesis of Theorem~\ref{TH:tight_crit}.
This space must be equipped with a topology that is sufficiently strong to facilitate the convergence of sequences utilized in the subsequent analysis.
Simultaneously, the topology must be weak enough to ensure the tightness result, meaning that there exists a compact set within this topology that satisfies the criteria for tightness.

The Aldous property mentioned in Theorem~\ref{TH:tight_crit} is quite intricate to be verified. 
Hence we present, in the following lemma, a sufficient condition to guarantee the Aldous property.
We refer to \cite[Section $3.2$]{Brzezniak+Motyl_2013_Existence_Martingale_Solution} for the proof of this result.

\begin{lemma}
\label{LEM:aldous}
    Assume that $(\Omega,\F,\{\F_t\}_{t\geq 0},\P)$ is an augmented filtered probability space. 
    Let $U'$ be an Hilbert space and $T>0$.
    Let $X^n:[0,T]\times\Omega\to U'$, $n\in\N$, be pathwise continuous stochastic processes.
    Assume that there exist $C>0$ and $\g\in(0,1]$ such that for every $t\in(0,T]$ and any sequence $\{\tau_n\}_{n\in\N}$ of $[0,T]$-valued stopping times:
    \begin{equation}
    \label{EQ:aldous_holder}
        \sup_{n\in\N}\E\|X^n_{\tau_n+t}-X^n_{\tau_n}\|^{}_{U'}\leq Ct^\g,
    \end{equation}
    Then the sequence $\{X^n\}_{n\in\N}$ satisfies the Aldous condition in $U'$, see Theorem~\ref{TH:tight_crit}.
\end{lemma}
We recall the Jakubowski's version \cite[Theorem $2$]{Jakubowski_1997_almost_sure_skorokhod} of the Skorokhod theorem, see also \cite{Brzezniak+Ondrejat_2013_stochastic_geometric_wave}, and deduce a simple corollary, that will later come in handy.
\begin{theorem}[Jakubowski Theorem]
\label{TH:jakubowski}
    Let $\mathcal X$ be a topological space such that there exists a continuous and injective function $F:\mathcal X\to \R^\N$. 
    Let $\{X_n\}_{n\in\N}$ be a sequence of $\mathcal X$-valued Borel random variables defined on some probability space $(\Omega,\F,\P)$. 
    Suppose that the sequence of laws $\{{(X_n)}_\ast\P\}_{n\in\N}$ is tight in $(\mathcal X,\mathscr B_{\mathcal X})$.
    Then there exists a subsequence $\{X_{n_k}\}_{k\in\N}$, a probability space $(\tilde\Omega,\tilde{\mathcal F}, \tilde \P)$, a Borel random variable $Y:(\tilde\Omega,\tilde\F)\to \mathcal X$ and a sequence $\{Y_k\}_{k\in\N}$ of $\mathcal X$-valued Borel random variables on $(\tilde\Omega,\tilde{\mathcal F})$ such that:
    \begin{align}
    \label{EQ:jakubowski_thesis_1}
        &({X_{n_k}})_\ast\P=({Y_k})_\ast \tilde{\P}\in\mathscr P(\mathcal X),\qquad \forall\, k\in\N,\\
    \label{EQ:jakubowski_thesis_2}
        &Y_k\longrightarrow Y, \qquad \tilde{\P}-a.s.\ \textit{in }\mathcal X, \ \textit{as }k\to\infty.
    \end{align}    
\end{theorem}
\begin{corollary}
\label{COR:Jakubowski}
    Let $\mathcal X$ be a topological space for which there exists a continuous and injective function $F:\mathcal X\to \R^\N$. 
    Let $\{\mu_n\}_{n\in\N}$ be a tight sequence of probability measures on $(\mathcal X,\mathscr B_\mathcal X)$. 
    Then there exists a subsequence $\{\mu_{n_k}\}_{k\in\N}$ and a probability measure $\mu$ on $(\mathcal X,\mathscr B_{\mathcal X})$ such that 
    \begin{gather}
        \mu_{n_k}\longrightarrow \mu, \qquad  \textit{in }\mathscr P(\mathcal X), \ \textit{as }k\to\infty.
    \end{gather} 
\end{corollary}
\begin{proof}
    Assume the hypotheses.
    Proposition $10.6.1$ in \cite{Cohn_2013_Measure_Theory_2nd_edition} states the existence of a probability space $(\Omega,\F,\P)$ (which is the countable product of the probability spaces $(\mathcal X,\mathscr B_{\mathcal X},\mu_n)$, $n\in\N$) and of $\mathcal X$-valued random variables $X_n$, $n\in\N$, defined on $(\Omega,\F,\P)$, such that $(X_n)_\ast\P=\mu_n$ on $(\mathcal X,\mathscr B_{\mathcal X})$ for every $n\in\N$.
    
    The sequence $\{X_n\}_{n\in\N}$ satisfies the assumptions of Jakubowski's Theorem~\ref{TH:jakubowski}. 
    Consequently there exists a subsequence $\{X_{n_k}\}_{k\in\N}$, a probability space $(\tilde\Omega,\tilde{\mathcal F}, \tilde \P)$, a Borel random variable $Y:(\tilde\Omega,\tilde\F)\to \mathcal X$ and a sequence $\{Y_k\}_{k\in\N}$ of $\mathcal X$-valued Borel random variables defined on $(\tilde\Omega,\tilde{\mathcal F})$ such that the properties in \eqref{EQ:jakubowski_thesis_1} and \eqref{EQ:jakubowski_thesis_2} hold true. 
    In particular, $\tilde\P-a.s.$ convergence implies convergence in law, \textit{i.e.}
    \[
        \mu_{n_k}=(X_{n_k})_\ast\P=(Y_k)_\ast\tilde\P\longrightarrow Y_\ast\tilde\P=\mu, \qquad \textit{as }k\to\infty,
    \]
    where the convergence is meant in the weak sense in $\mathscr P(\mathcal X)$, see Notation \ref{NOT:Borel_probabilities}.
\end{proof}
\begin{lemma}
\label{LEM:Z_T}
    Assume that $\a>1$ and let $W$ be an $H^{2\a}$-valued Wiener process.
    If $\b\in\R$, both the spaces $H^{\b}_{bw}$, and  {$\mathcal Z_T\times C\big([0,T];H^\b\big)$} with the product topology, satisfy the assumption of Theorem~\ref{TH:jakubowski} and Corollary~\ref{COR:Jakubowski}.
\end{lemma}
\begin{proof} 
    Let us momentarily say that a topological space $\mathcal X$ has the $S$ property if there exists an injective and continuous function $F:\mathcal X\to \R^\N$.
    
    As for the topological space $H^{\b}_{bw}$, we consider a complete orthonormal system $\{e_n\}_{n\in\N}$ for the separable Hilbert space $H^\b$ and define the function
    \[
        F:H^\b\ni x\mapsto F(x):=\big\{\lan x, e_n\ran^{}_{\! H^\b}\big\}_{n\in\N}\in \R^\N.
    \]
    Then $F$ is linear, thus $F\in C\big(H^\b_w;\R^\N\big)$ by definition of weak topology, see Notation \ref{NOT:weak_topologies}.
    It follows that ${F\in C\big(H^\b_{bw};\R^\N\big)}$ from Lemma~\ref{LEM:topologies_inclusions}. 
    Also, $F$ is injective. 
    Let indeed $x,x'\in H^\b$ be such that $F(x)=F(x')$, then $F(x-x')=0$ by linearity, hence $x=x'$ thanks to
    \[
        \|x-x'\|^2_{H^\b}=\sum_{n=1}^\infty \lan x-x',e_n\ran^{2}_{\!H^\b}=0.
    \]
    
    The spaces $C\big([0,T];H^\b\big)$ and $C\big([0,T];U'\big)$ have the $S$ property because they are Polish, see \cite[Exposè $8$, page $124$, Remark $(b)$]{Badrikian_1970_Seminaire_fonctions}. 
    
    As for the space $\mathcal Z_T$, we reason as follows. 
    Let $\iota:\mathcal Z_T\to C\big([0,T];U'\big)$ be the natural embedding, in particular, $\iota$ is injective and continuous. 
    Let $F:C\big([0,T];U'\big)\to\R^\N$ be an injective and continuous map, which exists because $C\big([0,T];U'\big)$ has the $S$ property. 
    Then $F\circ \iota:\mathcal Z_T\to\R^\N$ is continuous and injective.
    Hence $\mathcal Z_T$ has the $S$ property.
    
    It only remains to prove that, if two topological spaces $\mathcal X$, $\mathcal Y$ have the $S$ property, so does the space $\mathcal X\times \mathcal Y$ with the product topology. 
    Let $F:\mathcal X\to\R^\N$ and $G:\mathcal Y\to\R^\N$ be two injective and continuous functions, then, for every $n\in\N$, we denote by $f_n$ and $g_n$ the $n$-th component of $F$ and $G$, respectively. 
    We define $H:\mathcal X\times\mathcal Y\ni (x,y)\mapsto H(x,y):=\{h_n(x,y)\}_{n\in\N}\in \R^\N$, where, for all $n\in\N$ and $(x,y)\in \mathcal X\times\mathcal Y$
    \begin{align}
        h_{2n-1}(x,y):=f_n(x),\qquad\qquad
        h_{2n}(x,y):=g_n(y).
    \end{align}
    By direct inspection, $h_n$,  $n\in\N$, are continuous with respect to the product topology on $\mathcal X\times\mathcal Y$, thus so is $H$. 
    Finally, $H$ is injective. 
    Let indeed $(x,y),(x',y')\in \mathcal X\times\mathcal Y$ be such that $H(x,y)=H(x',y')$. 
    Then $f_n(x)=h_{2n-1}(x,y)=h_{2n-1}(x',y')=f_n(x')$ for all $n\in\N$, in particular, $F(x)=F(x')$, thus $x=x'$ by injectivity of $F$. 
    Analogously, $y=y'$ since $G$ is injective and $g_n(y)=h_{2n}(x,y)=h_{2n}(x',y')=g_n(y')$ for all $n\in\N$. Therefore $(x,y)=(x',y')$. 
\end{proof}

\subsection{The inviscid limit}

\begin{lemma}
\label{LEM:nu_to_0}
    Assume that $\a>1$, $T>0$, and let $\iota, \upsilon$ denote the natural embeddings
    \[
    H^{2\a}\xhookrightarrow{\iota}H^{\a+1}
    \xhookrightarrow{\upsilon}U'.
    \]
    Let $\{\nu_n\}_{n\in\N}$ be an infinitesimal strictly positive sequence.
    For every $n\in\N$, let $\mu_n\in\mathscr P(H^{2\a})$ be an invariant measure for the
    SHNS$_{\nu_n,\a}$ Equation~\eqref{EQ:NSnu}, and let $(\Omega,\F,\{\F_t\}_{t\ge0},\P;W;X^n)$ be a solution such that
    \[
        (X_t^n)_\ast\P=\mu_n,
        \qquad \forall\, t\ge0.
    \]
    Assume moreover that there exists
    $\mu\in\mathscr P(H^{\a+1})$ such that
    \begin{equation}
    \label{EQ:ass_conv_mu_n_to_mu_bw}
        \iota_\ast\mu_n\longrightarrow \mu,
        \qquad \textit{in }\mathscr P(H^{\a+1}_{bw}),
        \ \textit{as }n\to\infty.
    \end{equation}
    Then there exist:
    \begin{itemize}
        \item a strictly increasing map $\N\ni j\mapsto n_j\in\N$,
        \item a probability space $(\Omega^{\ss T},\F^{\ss T},\P^{\ss T})$,
        \item random variables $ X^{\ss T}, X^{{\ss T},j}:\Omega^{\ss T}\to\mathcal Z_T$, for $j\in\N$,
    \end{itemize}
    such that:
    \begin{enumerate}[label=$(\roman*)$]
        \item
        \label{IT:law_Xj_equals_law_Xnj_on_ZT}
            $(X^{{\ss T},j})_\ast\P^{\ss T}=(X^{n_j})_\ast\P \qquad\text{in }\mathscr P(\mathcal Z_T)$, for all $j\in\N$; 
        \item
        \label{IT:Xj_to_X_in_ZT}
            $X^{{\ss T},j}\longrightarrow X^{\ss T}
            \qquad\textit{in }\mathcal Z_T$ $\P^{\ss T}\text{-a.s., as }j\to\infty$;
        \item
        \label{IT:Xj_C_H2a}
            $X^{{\ss T},j}\in C([0,T];H^{2\a}),
            \qquad \P^{\ss T}\text{-a.s.,}\ \forall\, j\in\N$;
        \item
        \label{IT:X_reg_limit}
            $\ds X^{\ss T}\in
            \bigcap_{p\ge2}L^p\big(\Omega^{\ss T};L^\infty(0,T;H^1)\big)
            \cap
            L^2\big(\Omega^{\ss T}\times(0,T);H^{\a+1}\big)$;
        \item
        \label{IT:law_limit_time_Uprime}
            $(X^{\ss T}_t)_\ast\P^{\ss T}=\upsilon_\ast\mu
            \qquad\text{in }\mathscr P(U')$, for all $t\in[0,T]$;

        \item
        \label{IT:law_limit_time_Ha+1}
        for every $t\in[0,T]$ there exists an $H^{\a+1}$-valued random variable $\widehat X^{\ss T}_t$ such that $\widehat X^{\ss T}_t=X^{\ss T}_t$ $\P^{\ss T}\text{-a.s. in }U'$, and $(\widehat X^{\ss T}_t)_\ast\P^{\ss T}=\mu$ in $\mathscr P(H^{\a+1})$.
    \end{enumerate}
\end{lemma}
\begin{proof}
    First, recall from Lemma~\ref{LEM:weak_borel} that the $\s$-algebras generated by the strong and bounded weak topologies coincide, hence $\mathscr P(H_{bw}^{\a+1})=\mathscr P(H^{\a+1})$.
    However, since the bounded weak topology is finer by Lemma~\ref{LEM:topologies_inclusions}, the convergence in $\mathscr P(H^{\a+1})$ implies the convergence in $\mathscr P(H^{\a+1}_{bw})$, the converse being generally false.
    
\textbf{\textit{Step $\mathbf{1}$.}}
    Since $\a>1$, each process $X^n$ has continuous trajectories in $H^{2\a}$,
    hence also in $U'$. Therefore Theorem~\ref{TH:tight_crit} applies once its two
    hypotheses are verified.

    We first prove the uniform moment bound required by
    \eqref{EQ:HP_bound_tight_crit}.
    By Theorem~\ref{TH:exi!_solutions_NS_X}
    \itref{IT:solutions_NS_estimates}, with $p=2$, and by stationarity,
    for every $n\in\N$,
    \begin{align*}
        \E\sup_{t\in[0,T]}\|X_t^n\|_{H^1}^2
        &\le C\big(\E\|X_0^n\|_{H^1}^2+T\nu_n\big)= C\bigg(\int_{H^{2\a}}\|x\|_{H^1}^2\d\mu_n(x)+T\nu_n\bigg),\\
        \E\int_0^T\|X_t^n\|_{H^{\a+1}}^2\d t
        &=T\,\E\|X_0^n\|_{H^{\a+1}}^2
        =T\int_{H^{2\a}}\|x\|_{H^{\a+1}}^2\d\mu_n(x).
    \end{align*}
    By Theorem~\ref{TH:moment_estimates_mu_nu}, the sequence
    $\{\mu_n\}_{n\in\N}$ has uniformly bounded second moment in $H^{\a+1}$ and $H^1$; since
    also $\{\nu_n\}_{n\in\N}$ is bounded, these quantities are uniformly bounded in $n$.

    We now verify Aldous' condition through Lemma~\ref{LEM:aldous}.
    For every $n\in\N$, $\P$-a.s., for all $t\in[0,T]$, in $H^\a$,
    \[
        X_t^n
        =
        X_0^n
        -\nu_n\int_0^tA^\a X_s^n\d s
        -\int_0^tB(X_s^n)\d s
        +\sqrt{\nu_n}\,W_t.
    \]
    Let $\{\tau_n\}_{n\in\N}$ be stopping times valued in $[0,T]$, and let
    $t\in[0,T]$.
    The constant term $X_0^n$ trivially satisfies \eqref{EQ:aldous_holder}.
    For the linear term, using the embedding
    $H^{1-\a}\hookrightarrow H^{-2\a}\hookrightarrow U'$,
    Hölder's inequality, and again Theorem~\ref{TH:exi!_solutions_NS_X}
    \itref{IT:solutions_NS_estimates}, we get
    \begin{align*}
        \E\bigg\|
            \nu_n\int_{\tau_n}^{\tau_n+t}A^\a X_s^n\d s
        \bigg\|_{U'}
        &\le
        \nu_n\E\int_{\tau_n}^{\tau_n+t}\|A^\a X_s^n\|_{H^{1-\a}}\d s\\
        &\le
        \sqrt{\nu_nt}\,
        \bigg(
            \nu_n\E\int_0^{2T}\|X_s^n\|_{H^{\a+1}}^2\d s
        \bigg)^{1/2}\\
        &\le
        C\sqrt{\nu_nt}.
    \end{align*}
    For the nonlinear term, by Lemma~\ref{LEM:estimates_B}
    (estimate \eqref{EQ:estim_B_3} with $\e=\a-1>0$),
    \begin{align*}
        \E\bigg\|
            \int_{\tau_n}^{\tau_n+t}B(X_s^n)\d s
        \bigg\|_{U'}
        \le
        \E\int_{\tau_n}^{\tau_n+t}\|B(X_s^n)\|_{H^{1-\a}}\d s
        \le
        Ct\,\E\sup_{s\in[0,2T]}\|X_s^n\|_{H^1}^2
        \le Ct.
    \end{align*}
    Finally, for the noise term,
    \[
        \E\big\|\sqrt{\nu_n}(W_{\tau_n+t}-W_{\tau_n})\big\|_{U'}
        \le
        \sqrt{\nu_nt}\,\big(\Tr Q\big)^{1/2},
    \]
    where $Q\in\mathcal L_1(H^{2\a})$ is the covariance of $W_1$.
    Hence Aldous' condition holds uniformly in $n$.

    By Lemma~\ref{LEM:aldous} and Theorem~\ref{TH:tight_crit}, the laws
    $\{(X^n)_*\P\}_{n\in\N}$ form a tight family in
    $\mathscr P(\mathcal Z_T)$.
    
    By Lemma~\ref{LEM:Z_T}, the topological space $\mathcal Z_T$ satisfies the
    assumptions of Corollary~\ref{COR:Jakubowski}.
    Since the laws of $X^n$ are tight on $\mathcal Z_T$, Corollary~\ref{COR:Jakubowski}
    yields a subsequence $\{n_j\}_{j\in\N}$, a probability space
    $(\Omega^{\ss T},\F^{\ss T},\P^{\ss T})$, and $\mathcal Z_T$-valued random variables
    $ X^{{\ss T},j}, X^{\ss T}$ such that
    \begin{align}
    \label{EQ:law_tildeXj_equals_Xnj}
        ( X^{{\ss T},j})_*\P^{\ss T}
        &=
        (X^{n_j})_*\P
        \qquad\text{in }\mathscr P(\mathcal Z_T),
        \qquad \forall\, j\in\N,\\
    \label{EQ:tildeXj_to_tildeX_ZT_as}
         X^{{\ss T},j}
        &\longrightarrow
         X^{\ss T}
        \qquad\textit{in }\mathcal Z_T,\ \P^{\ss T}\text{-a.s., as }j\to\infty.
    \end{align}
    This proves \itref{IT:law_Xj_equals_law_Xnj_on_ZT} and
    \itref{IT:Xj_to_X_in_ZT}.

\textbf{\textit{Step $\mathbf{2}$.}}
    Let \[
        S:=C([0,T];H^{2\a})\cap \mathcal Z_T.
    \]
    By Lemma~\ref{LEM:Borel_sets_in_ZT}, $S\in\mathscr B_{\mathcal Z_T}$.
    Since $X^{n_j}\in C([0,T];H^{2\a})$ $\P$-a.s. for every $j$, then $\big((X^{n_j})_\ast\P\big)(S)=1$.
    By \eqref{EQ:law_tildeXj_equals_Xnj}, we have $\big(( X^{{\ss T},j})_\ast\P^{\ss T}\big)(S)=1$, hence
    \[
         X^{{\ss T},j}\in C([0,T];H^{2\a}),
        \qquad \P^{\ss T}\text{-a.s.,}\ \forall\, j\in\N.
    \]
    This proves \itref{IT:Xj_C_H2a}.
    We next establish uniform integrability estimates for $ X^{{\ss T},j}$.
    Define
    \[
        G_1(u):=\sup_{q\in\mathbb Q\cap[0,T]}\|u(q)\|_{H^1}^p,
        \qquad u\in\mathcal Z_T,
    \]
    for a fixed $p\ge2$.
    The map $\mathcal Z_T\ni u \mapsto u(q)\in H^\a_w$ is continuous for every
    rational $q$, because $\mathcal Z_T\hookrightarrow C([0,T];H^\a_w)$ continuously.
    Since $H^\a\hookrightarrow H^1$, $H^\a_w\hookrightarrow H^1_w$, and the norm
    $\|\,\cdot\,\|_{H^1}$ is weakly lower semicontinuous on $H^1$, the map
    $u\mapsto \|u(q)\|_{H^1}^p$ is Borel on $\mathcal Z_T$.
    Hence $G_1$ is Borel as a countable supremum of Borel maps.
    Similarly, define
    \[
        G_2(u):=\|u\|_{L^2(0,T;H^{\a+1})}^2,
        \qquad u\in\mathcal Z_T.
    \]
    Since $\mathcal Z_T\hookrightarrow L^2_w(0,T;H^{\a+1})$ and the norm
    on $L^2(0,T;H^{\a+1})$ is weakly lower semicontinuous, $G_2$ is Borel on
    $\mathcal Z_T$.
    If $u\in S$, then $u\in C([0,T];H^{2\a})\subset C([0,T];H^1)$, so by continuity
    of $t\mapsto \|u(t)\|_{H^1}$ one has
    \[
        G_1(u)=\|u\|_{C([0,T];H^1)}^p, \qquad \forall\, u\in S.
    \]
    Therefore, by \eqref{EQ:law_tildeXj_equals_Xnj},
    \begin{equation}
    \label{EQ:uniform_integrability_tildeXj}
        \tilde\E G_1( X^{{\ss T},j})+\tilde\E G_2( X^{{\ss T},j})
        =
        \E G_1(X^{n_j})+\E G_2(X^{n_j})
        =
        \E\sup_{t\in[0,T]}\|X_t^{n_j}\|_{H^1}^p
        +\E\int_0^T\|X_t^{n_j}\|_{H^{\a+1}}^2\d t
        \le C_p,
    \end{equation}
    where $C_p$ is independent of $j$, by the same estimates used in \textit{Step~1}.

    We now pass to the limit.
    Let $\Omega^{\ss T}_T\in\F^{\ss T}$ be such that $\P^{\ss T}(\Omega^{\ss T}_T)=1$ and
    \eqref{EQ:tildeXj_to_tildeX_ZT_as} holds for all $\omega\in\Omega^{\ss T}_T$.
    Fix $\omega\in\Omega^{\ss T}_T$.
    Since $ X^{{\ss T},j}(\omega)\to X^{\ss T}(\omega)$ in $\mathcal Z_T$, we have:
    \begin{align}
        & X^{{\ss T},j}(\omega)\to X^{\ss T}(\omega)\text{ in }C([0,T];H^\a_w),\\
        & X^{{\ss T},j}(\omega)\to X^{\ss T}(\omega)\text{ in }L^2_w(0,T;H^{\a+1}).
    \end{align}
    For every $q\in\mathbb Q\cap[0,T]$, the embedding $H^\a_w\hookrightarrow H^1_w$ and the weak lower semicontinuity of the $H^1$-norm yield
    \[
        \| X^{\ss T}_q(\omega)\|_{H^1}^p
        \le
        \liminf_{j\to\infty}\| X^{{\ss T},j}_q(\omega)\|_{H^1}^p.
    \]
    Taking the supremum over $q\in\mathbb Q\cap[0,T]$,
    \[
        G_1( X^{\ss T}(\omega))
        \le
        \liminf_{j\to\infty}G_1( X^{{\ss T},j}(\omega)).
    \]
    Since $ X^{\ss T}(\omega)\in C([0,T];H^\a_w)$ and $H^\a\hookrightarrow H^1$
    continuously, the path $t\mapsto  X^{\ss T}_t(\omega)$ is continuous in $H^1_w$.
    Hence $t\mapsto \| X^{\ss T}_t(\omega)\|_{H^1}$ is lower semicontinuous on the
    compact interval $[0,T]$, so its supremum over $\mathbb Q\cap[0,T]$ equals its supremum
    over $[0,T]$. Therefore
    \[
        \| X^{\ss T}(\omega)\|_{L^\infty(0,T;H^1)}^p
        =
        G_1( X^{\ss T}(\omega))
        \le
        \liminf_{j\to\infty}\| X^{{\ss T},j}(\omega)\|_{C([0,T];H^1)}^p.
    \]
    Likewise, by the weak lower semicontinuity of the norm on
    $L^2(0,T;H^{\a+1})$,
    \[
        \| X^{\ss T}(\omega)\|_{L^2(0,T;H^{\a+1})}^2
        \le
        \liminf_{j\to\infty}\| X^{{\ss T},j}(\omega)\|_{L^2(0,T;H^{\a+1})}^2.
    \]
    Fatou's lemma together with \eqref{EQ:uniform_integrability_tildeXj} gives
    \[
         X^{\ss T}\in
        \bigcap_{p\ge2}L^p\big(\Omega^{\ss T};L^\infty(0,T;H^1)\big)
        \cap
        L^2\big(\Omega^{\ss T}\times(0,T);H^{\a+1}\big).
    \]
    This proves \itref{IT:X_reg_limit}.

\textbf{\textit{Step $\mathbf{3}$.}}
    Fix $t\in[0,T]$.
    Since \eqref{EQ:ass_conv_mu_n_to_mu_bw} holds along the full sequence, it also
    holds along the extracted subsequence:
    \begin{equation}
    \label{EQ:ass_conv_subsequence}
        \iota_\ast\mu_{n_j}\longrightarrow \mu
        \qquad\textit{in }\mathscr P(H^{\a+1}_{bw}),
        \ \textit{as }j\to\infty.
    \end{equation}
    Recalling that $\upsilon$ is continuous, \eqref{EQ:ass_conv_subsequence} implies
    \begin{equation}
    \label{EQ:iota_bar_muj_to_iota_mu}
        \upsilon_\ast(\iota_\ast\mu_{n_j})\longrightarrow \upsilon_\ast\mu
        \qquad\textit{in }\mathscr P(U').
    \end{equation}
    
    On the other hand, since $\mathcal Z_T\hookrightarrow C([0,T];U')$ and $ X^{{\ss T},j}\to X^{\ss T}$ $\P^{\ss T}$-a.s.\ in $\mathcal Z_T$, we have $ X^{{\ss T},j}\longrightarrow X^{\ss T}$ in $C([0,T];U')$, $\P^{\ss T}\text{-a.s.}$.
    Therefore $  X^{{\ss T},j}_t\longrightarrow X^{\ss T}_t$ in $U'$, $\P^{\ss T}\text{-a.s.}$ and thus
    \begin{equation}
    \label{EQ:law_tildeXtj_to_tildeXt_in_U'}
        (  X^{{\ss T},j}_t)_\ast\P^{\ss T}
        \longrightarrow
        ( X^{\ss T}_t)_\ast\P^{\ss T},
        \qquad\textit{in }\mathscr P(U').
    \end{equation}

    We now identify $(  X^{{\ss T},j}_t)_*\P^{\ss T}$.
    Let $\kappa:\mathcal Z_T\to C([0,T];U')$ denote the natural embedding, and let $\mathrm{ev}_t:C([0,T];U')\ni u \mapsto u(t)\in U'$ be the evaluation map, which is continuous.
    Then $\widetilde{\mathrm{ev}}_t:=\mathrm{ev}_t\circ\kappa:\mathcal Z_T\to U'$ is Borel measurable.
    Hence, by \eqref{EQ:law_tildeXj_equals_Xnj},
    \begin{align}
    \label{EQ:law_tildeXtj_equals_Xnjt_in_U'}
        (  X^{{\ss T},j}_t)_*\P^{\ss T}
        &=
        (\widetilde{\mathrm{ev}}_t\circ X^{{\ss T},j})_*\P^{\ss T}\notag\\
        &=
        (\widetilde{\mathrm{ev}}_t)_*\big(( X^{{\ss T},j})_*\P^{\ss T}\big)\notag\\
        &=
        (\widetilde{\mathrm{ev}}_t)_*\big((X^{n_j})_*\P\big)\notag\\
        &=
        (\widetilde{\mathrm{ev}}_t\circ X^{n_j})_*\P
        =
        (X_t^{n_j})_*\P
        \qquad\text{in }\mathscr P(U').
    \end{align}

    Now $X_t^{n_j}$ is originally an $H^{2\a}$-valued random variable with law
    $\mu_{n_j}$, and the same random variable viewed in $U'$ is
    $(\upsilon\circ\iota)(X_t^{n_j})$.
    Therefore
    \begin{align}
    \label{EQ:law_Xnjt_as_Uprime}
        (X_t^{n_j})_*\P
        =
        (\upsilon\circ\iota)_\ast\big((X_t^{n_j})_\ast\P\big)
        =
        (\upsilon\circ\iota)_\ast\mu_{n_j}
        \qquad\text{in }\mathscr P(U').
    \end{align}
    Combining \eqref{EQ:law_tildeXtj_equals_Xnjt_in_U'} and
    \eqref{EQ:law_Xnjt_as_Uprime}, we obtain
    \begin{equation}
    \label{EQ:law_tildeXtj_equals_iota_bar_muj}
        (  X^{{\ss T},j}_t)_\ast\P^{\ss T}=\upsilon_\ast(\iota_\ast\mu_{n_j})
        \qquad\text{in }\mathscr P(U').
    \end{equation}

    Finally, \eqref{EQ:iota_bar_muj_to_iota_mu},
    \eqref{EQ:law_tildeXtj_to_tildeXt_in_U'}, and
    \eqref{EQ:law_tildeXtj_equals_iota_bar_muj} show that the sequence
    $\{(  X^{{\ss T},j}_t)_*\P^{\ss T}\}_{j\in\N}$ converges in $\mathscr P(U')$ both to
    $( X^{\ss T}_t)_*\P^{\ss T}$ and to $\upsilon_\ast\mu$.
    Since $U'$ is Polish, $\mathscr P(U')$ is Hausdorff, so the limit is unique.
    Hence
    \[
        ( X^{\ss T}_t)_\ast\P^{\ss T}=\upsilon_\ast\mu
        \qquad\text{in }\mathscr P(U').
    \]
    This proves \itref{IT:law_limit_time_Uprime}.

    We now prove the \(H^{\a+1}\)-valued formulation \itref{IT:law_limit_time_Ha+1}.
    We know from Lemma~\ref{LEM:Borel_X_cap_Y} that $H^{\a+1}\in\mathscr B_{U'}$, hence $(\upsilon_\ast\mu)(H^{\a+1})=\mu(H^{\a+1})=1$.
    Then the equality
    $( X^{\ss T}_t)_\ast\P^{\ss T}=\upsilon_\ast\mu$ yields $\P^{\ss T}( X^{\ss T}_t\in H^{\a+1})=1$.
    Hence the random variable
    \[
        \widehat X^{\ss T}_t:= X^{\ss T}_t\1_{\{ X^{\ss T}_t\in H^{\a+1}\}}
    \]
    is well defined as an $H^{\a+1}$-valued random variable, and $\widehat X^{\ss T}_t= X^{\ss T}_t$ $\P^{\ss T}\text{-a.s. in }U'$.
    For every $B\in\mathscr B_{H^{\a+1}}$,
    \[
        \P^{\ss T}\big(\widehat X^{\ss T}_t\in B\big)
        =
        \P^{\ss T}\big( X^{\ss T}_t\in \upsilon(B)\big)
        =
        \big(\upsilon_\ast\mu\big)(\upsilon(B))
        =
        \mu(B).
    \]
    Therefore $(\widehat X^{\ss T}_t)_\ast\P^{\ss T}=\mu$ in $\mathscr P(H^{\a+1})$, which proves \itref{IT:law_limit_time_Ha+1}.
\end{proof}

\begin{proposition}
\label{PROP:approx_inv_measures}
    Assume that $\a>1$, and let $\iota$ denote the natural embedding $H^{2\a}\xhookrightarrow{\iota} H^{\a+1}$.
    Then there exist:
    \begin{itemize}
        \item an infinitesimal strictly positive sequence $\{\nu_n\}_{n\in\N}$,
        \item invariant measures $\mu_n\in\mathscr P(H^{2\a})$ for the SHNS$_{\nu_n,\a}$ Equation~\eqref{EQ:NSnu} with kinematic viscosity $\nu_n$, for all $n\in\N$,
        \item a probability measure $\mu\in\mathscr P(H^{\a+1})$,
    \end{itemize}
    such that
    \[
        \iota_\ast\mu_n\longrightarrow\mu,
        \qquad \textit{in }\mathscr P(H^{\a+1}_{bw}),
        \ \textit{as }n\to\infty.
    \]
\end{proposition}

\begin{proof}
    Let $\{\bar\nu_k\}_{k\in\N}\in (0,+\infty)^\N$ be any infinitesimal sequence.
    For every $k\in\N$, by Theorem~\ref{TH:invariant_measure}, there exists an invariant measure $\bar\mu_k\in\mathscr P(H^{2\a})$ for the SHNS$_{\bar\nu_k,\a}$ Equation~\eqref{EQ:NSnu} with kinematic viscosity $\bar\nu_k$.
    By Theorem~\ref{TH:moment_estimates_mu_nu}, there exists a constant $C>0$, independent of $k$, such that
    \[
        \int_{H^{2\a}}\|x\|_{H^{\a+1}}^2\d\bar\mu_k(x)\le C,
        \qquad \forall\, k\in\N.
    \]

    We claim that $\{\iota_\ast\bar\mu_k\}_{k\in\N}$ is tight in $\mathscr P(H^{\a+1}_{bw})$.
    Indeed, fix $R>0$ and consider the closed ball
    \[
        B_R:=\big\{x\in H^{\a+1}:\ \|x\|_{H^{\a+1}}\le R\big\}.
    \]
    Since $H^{\a+1}$ is reflexive and separable, $B_R$ is weakly compact in $H^{\a+1}$.
    Moreover, by definition of the bounded weak topology, the topology induced by $H^{\a+1}_{bw}$ on
    $B_R$ coincides with the weak topology of $H^{\a+1}$ on $B_R$.
    Therefore $B_R$ is compact in $H^{\a+1}_{bw}$.
    Furthermore, by Chebyshev's inequality,
    \[
        (\iota_\ast\bar\mu_k)\big(H^{\a+1}\setminus B_R\big)
        =
        \bar\mu_k\big(\{x\in H^{2\a}:\ \|x\|_{H^{\a+1}}>R\}\big)
        \le
        \frac1{R^2}\int_{H^{2\a}}\|x\|_{H^{\a+1}}^2\d\bar\mu_k(x)
        \le
        \frac{C}{R^2},
    \]
    for all $k\in\N$.
    Thus, given $\e>0$, choosing $R$ so large that $C/R^2<\e$, we obtain $(\iota_\ast\bar\mu_k)(B_R)\ge 1-\e$ for all $k\in\N$.
    Since $B_R$ is compact in $H^{\a+1}_{bw}$, the claim follows.

    By Corollary~\ref{COR:Jakubowski} and Lemma~\ref{LEM:Z_T}, there exist a strictly increasing map $\N\ni n\mapsto k_n\in\N$ and a probability measure $\mu\in\mathscr P(H^{\a+1}_{bw})$ such that
    \[
        \iota_\ast\bar \mu_{k_n}\longrightarrow\mu,
        \qquad \textit{in }\mathscr P(H^{\a+1}_{bw}),
        \ \textit{as }n\to\infty.
    \]
    Since $H^{\a+1}$ is a separable Hilbert space, Lemma~\ref{LEM:weak_borel} implies equality of the strong and bounded-weak Borel $\S$-algebras; therefore the limiting measure $\mu\in\mathscr P(H^{\a+1}_{bw})$ can be regarded canonically as an element of $\mathscr P(H^{\a+1})$.
    The thesis follows by relabeling $\nu_n:=\bar \nu_{k_n}$ and $\mu_n:=\bar\mu_{k_n}$.
\end{proof}

\begin{theorem}
\label{TH:nu_to_0}
    Assume that $\s>2$.
    Then there exist an augmented filtered probability space $(\tilde\Omega,\tilde\F,\{\tilde\F_t\}_{t\ge0},\tilde\P)$ and an adapted $H^\s$-valued process $\tilde X$ such that:
    \begin{enumerate}[label=$(\roman*)$]
        \item for every $T>0$,
        \[
            \tilde X\in
            \bigcap_{p\ge1}L^p\big(\tilde\Omega;L^\infty(0,T;H^1)\big)
            \cap
            L^2\big(\tilde\Omega\times(0,T);H^\s\big);
        \]

        \item for $\tilde\P$-a.e.\ $\omega\in\tilde\Omega$, the path $\tilde X(\omega):\R_+\to H^\s$ belongs to $C(\R_+;H^\s)$ and is a solution to the Euler Equation~\eqref{EQ:euler};

        \item the process is marginally stationary, namely $(\tilde X_t)_\ast\tilde\P=(\tilde X_0)_\ast\tilde\P$ for all $t\geq 0$.
    \end{enumerate}
    In particular, $(\tilde X_0)_\ast\tilde\P\in\mathscr P(H^\s)$ is an invariant measure for the Euler Equation~\eqref{EQ:euler}, in the sense of Definition~\ref{DEF:inv_meas_euler}.
\end{theorem}
\begin{proof}
    Let us set $\a:=\s-1$, and let $\iota,\upsilon$ denote the natural embeddings $H^{2\a}\xhookrightarrow{\iota}H^\s\xhookrightarrow{\upsilon}U'$.

\textit{\textbf{Step $\mathbf{1}$.}}
    Proposition~\ref{PROP:approx_inv_measures} returns a strictly positive infinitesimal sequence $\{\nu_n\}_{n\in\N}$, invariant measures $\mu_n\in\mathscr P(H^{2\a})$ for the SHNS$_{\nu_n,\a}$ Equation~\eqref{EQ:NSnu}, for all $n\in\N$, and a probability measure $\mu\in\mathscr P(H^\s)$ such that
    \begin{equation}
    \label{EQ:conv_mu_n_to_mu_in_PHs_bw}
        \iota_\ast\mu_n\longrightarrow\mu,
        \qquad \textit{in }\mathscr P(H^\s_{bw}),
        \ \textit{as }n\to\infty.
    \end{equation}
    For every $n\in\N$, by Proposition~\ref{PROP:construct_stationary_solution},
    there exists a marginally stationary solution $(\Omega,\F,\{\F_t\}_{t\ge0},\P;W;X^n)$ to the SHNS$_{\nu_n,\a}$ Equation~\eqref{EQ:NSnu} such that
    \[
        (X_t^n)_\ast\P=\mu_n,
        \qquad \forall\, t\ge0.
    \]
    
\textit{\textbf{Step $\mathbf{2}$.}}
    Fix $T>0$ and apply Lemma~\ref{LEM:nu_to_0} to the sequence $\{(\nu_n,\mu_n,X^n)\}_{n\in\N}$.
    We obtain a strictly increasing map $\N\ni j\mapsto n_j\in\N$, a probability space $(\Omega^{\ss T},\F^{\ss T},\P^{\ss T})$, and random variables $X^{\ss T},X^{{\ss T},j}:\Omega^{\ss T}\to\mathcal Z_T$, for $j\in\N$, such that \itref{IT:law_Xj_equals_law_Xnj_on_ZT}--\itref{IT:law_limit_time_Ha+1}
    hold.

    For simplicity of notation, set
    \[
        \bar\nu_j:=\nu_{n_j},
        \qquad
        \bar X^j:=X^{n_j},
        \qquad
        \bar\mu_j:=\mu_{n_j},
        \qquad j\in\N.
    \]

    We claim that, for $\P^{\ss T}$-a.e.\ $\omega\in\Omega^{\ss T}$, the path $[0,T]\ni t\longmapsto X^{\ss T}_t(\omega)\in U'$ satisfies the Euler Equation~\eqref{EQ:euler} on $[0,T]$ in the sense
    \begin{equation}
    \label{EQ:euler_identity_Uprime}
        X^{\ss T}_t+\int_0^t B(X^{\ss T}_s)\d s=X^{\ss T}_0
        \qquad\text{in }H^{1-\a},
        \qquad \forall\, t\in[0,T].
    \end{equation}

    Fix $t\in[0,T]$ and $y\in H^{\a-1}$. Define $F^j_{t,y},F_{t,y}:\mathcal Z_T\to\R$ by
    \begin{align*}
        F^j_{t,y}(u)
        &:=
        \big\langle y,u(t)-u(0)\big\rangle
        +\bar\nu_j\int_0^t\big\langle A^\a u(s),y\big\rangle_{\ss{\!H^{1-\a}\!\times\!H^{\a-1}}}\d s
        +\int_0^t\big\langle B(u(s)),y\big\rangle_{\ss{\!H^{1-\a}\!\times\!H^{\a-1}}}\d s,\\
        F_{t,y}(u)
        &:=
        \big\langle y,u(t)-u(0)\big\rangle
        +\int_0^t\big\langle B(u(s)),y\big\rangle_{\ss{\!H^{1-\a}\!\times\!H^{\a-1}}}\d s.
    \end{align*}
    Here, in the first term, $y$ is identified with an element of $H^{-\a}$.

    We show that the map $F^j_{t,y}$ is Borel on $\mathcal Z_T$:
    \begin{itemize}
        \item The map $\mathcal Z_T\ni u\mapsto \big\langle u(t)-u(0),y\big\rangle\in\R$ is Borel because $\mathcal Z_T\hookrightarrow C([0,T];H^\a_w)$ continuously,
        fixed-time evaluation is continuous on $C([0,T];H^\a_w)$, and the pairing with
        $y$ is continuous on $H^\a_w$.

        \item The map $\mathcal Z_T\ni u\mapsto \int_0^t\big\langle A^\a u(s),y\big\rangle_{\ss{\!H^{1-\a}\!\times\!H^{\a-1}}}\d s\in\R$ is Borel because $\mathcal Z_T\hookrightarrow L^2_w(0,T;H^{\a+1})$
        continuously, $A^\a:L^2_w(0,T;H^{\a+1})\to L^2_w(0,T;H^{1-\a})$ is continuous,
        and $v\mapsto \int_0^t\langle v(s),y\rangle_{\ss{\!H^{1-\a}\!\times\!H^{\a-1}}}\d s\in\R$ is continuous on $L^2_w(0,T;H^{1-\a})$.

        \item The map $\mathcal Z_T\ni u\mapsto \int_0^t\big\langle B(u(s)),y\big\rangle_{\ss{\!H^{1-\a}\!\times\!H^{\a-1}}}\d s\in\R$ is Borel.
        Indeed, it is continuous on $L^2(0,T;H^\a)$ with the strong topology by Lemma~\ref{LEM:estimates_B}.
        Since $L^2(0,T;H^\a)$ is a separable Hilbert space, the weak and strong Borel $\sigma$-algebras coincide. 
        Therefore the map is Borel on  $L^2_w(0,T;H^\a)$, and thus on $\mathcal Z_T$ because $\mathcal Z_T\hookrightarrow L^2_w(0,T;H^\a)$ continuously.
    \end{itemize}

    Therefore, by Lemma~\ref{LEM:nu_to_0} \itref{IT:law_Xj_equals_law_Xnj_on_ZT},
    \begin{equation}
    \label{EQ:law_FjtY_tilde_equals_original}
        \big(F^j_{t,y}(X^{{\ss T},j})\big)_\ast\P^{\ss T}
        =
        \big(F^j_{t,y}(\bar X^j)\big)_\ast\P
        \qquad\text{in }\mathscr P(\R).
    \end{equation}

    Since $\bar X^j$ solves the SHNS$_{\bar\nu_j,\a}$ equation, we have $F^j_{t,y}(\bar X^j) = \big\langle \sqrt{\bar\nu_j}\,W_t,y\big\rangle$, $\P\text{-a.s.}$
    Hence, by \eqref{EQ:law_FjtY_tilde_equals_original},
    \begin{equation}
    \label{EQ:L2_bound_FjtY}
        \E^{\ss T}\big|F^j_{t,y}(X^{{\ss T},j})\big|^2
        =
        \bar\nu_j\,\E\big|\langle W_t,y\rangle\big|^2
        \longrightarrow0,
        \qquad \text{as }j\to\infty,
    \end{equation}
    where $\E^{\ss T}$ denotes expectation with respect to $\P^{\ss T}$.
    In particular,
    \[
        F^j_{t,y}(X^{{\ss T},j})\longrightarrow0
        \qquad\text{in }\P^{\ss T}\text{-probability.}
    \]

    We now show that
    \begin{equation}
    \label{EQ:one_more}
        F^j_{t,y}(X^{{\ss T},j})\longrightarrow F_{t,y}(X^{\ss T})
        \qquad\text{in }\P^{\ss T}\text{-probability.}
    \end{equation}
    By Lemma~\ref{LEM:nu_to_0} \itref{IT:Xj_to_X_in_ZT}, we have  $X^{{\ss T},j}\longrightarrow X^{\ss T}$ in $\mathcal Z_T$, $\P^{\ss T}\text{-a.s.}$.
    Hence $\big\langle X^{{\ss T},j}_t-X^{{\ss T},j}_0,y\big\rangle
        \longrightarrow
        \big\langle X^{\ss T}_t-X^{\ss T}_0,y\big\rangle$, $\P^{\ss T}\text{-a.s.}$
    Moreover, since $X^{{\ss T},j}\to X^{\ss T}$ in $L^2_w(0,T;H^{\a+1})$, we have
    \[
        \int_0^t\big\langle A^\a X^{{\ss T},j}_s,y\big\rangle_{\ss{\!H^{1-\a}\!\times\!H^{\a-1}}}\d s
        \longrightarrow
        \int_0^t\big\langle A^\a X^{\ss T}_s,y\big\rangle_{\ss{\!H^{1-\a}\!\times\!H^{\a-1}}}\d s
        \qquad \P^{\ss T}\text{-a.s.}
    \]
    Therefore
    \[
        \bar\nu_j\int_0^t\big\langle A^\a X^{{\ss T},j}_s,y\big\rangle_{\ss{\!H^{1-\a}\!\times\!H^{\a-1}}}\d s
        \longrightarrow0
        \qquad\text{in }L^1(\Omega^{\ss T}),
    \]
    because
    \begin{align}
        \E^{\ss T}\left|
        \bar\nu_j\int_0^t\big\langle A^\a X^{{\ss T},j}_s,y\big\rangle_{\ss{\!H^{1-\a}\!\times\!H^{\a-1}}}\d s
        \right|
        &\le
        \bar\nu_j\,\|y\|_{H^{\a-1}}\,t^{1/2}
        \left(
            \E^{\ss T}\int_0^t\|X^{{\ss T},j}_s\|_{H^{\a+1}}^2\d s
        \right)^{1/2}\\
        &=
        \bar\nu_j\,\|y\|_{H^{\a-1}}\,t
        \left(
            \int_{H^{2\a}}\|x\|_{H^{\a+1}}^2\d \bar\mu_j(x)
        \right)^{1/2}.
    \end{align}
    Here, for the equality, we used Lemma~\ref{LEM:nu_to_0} \itref{IT:law_Xj_equals_law_Xnj_on_ZT} together with the fact that $\bar X^j$ is marginally stationary with invariant measure $\bar\nu_j\in\mathscr P(H^{2\a})$.
    The last term is uniformly bounded in $j$ by Theorem~\ref{TH:moment_estimates_mu_nu}.

    Finally, since $X^{{\ss T},j}\to X^{\ss T}$ in $L^2(0,T;H^\a_{loc})$, $\P^{\ss T}$-a.s., and $X^{{\ss T},j}\to X^{\ss T}$ in $L^2_w(0,T;H^{\a+1})$, $\P^{\ss T}$-a.s., the sequence $\{X^{{\ss T},j}\}_{j\in\N}$ is $\P^{\ss T}$-a.s.\ bounded in $L^2(0,T;H^{\a+1})$, hence also in $L^2(0,T;H^\a)$.
    Therefore, by Lemma~\ref{LEM:nonlinear_loc_convergence}~\itref{ITEM:nonlinear_loc_convergence_dual} with $\b=\a$,
    \[
        \int_0^t
        \big\langle
            B(X^{{\ss T},j}_s)-B(X^{\ss T}_s),y
        \big\rangle_{\ss{\!H^{1-\a}\!\times\!H^{\a-1}}}
        \d s
        \longrightarrow 0,
        \qquad \P^{\ss T}\text{-a.s.}
    \]

    Thus we showed \eqref{EQ:one_more}.
    Combined with \eqref{EQ:L2_bound_FjtY}, this yields $F_{t,y}(X^{\ss T})=0$, $\P^{\ss T}\text{-a.s.}$, namely
    \begin{equation}
    \label{EQ:euler_identity_fixed_t_y}
        \big\langle X^{\ss T}_t-X^{\ss T}_0,y\big\rangle
        +
        \int_0^t\big\langle B(X^{\ss T}_s),y\big\rangle_{\ss{\!H^{1-\a}\!\times\!H^{\a-1}}}\d s
        =0,
        \qquad \P^{\ss T}\text{-a.s.}
    \end{equation}

    Let now $D\subset H^{\a-1}$ be a countable dense subset and let $Q_T:=\mathbb Q\cap[0,T]$.
    Intersecting the corresponding full-measure events, we find $\Omega^{{\ss T}\ast}\in\F^{\ss T}$ with $\P^{\ss T}(\Omega^{{\ss T}\ast})=1$ such that \eqref{EQ:euler_identity_fixed_t_y} holds for every $t\in Q_T$ and every $y\in D$.

    Fix $\omega\in\Omega^{{\ss T}\ast}$.
    By Lemma~\ref{LEM:nu_to_0} \itref{IT:X_reg_limit},
    \[
        X^{\ss T}(\omega)\in L^\infty(0,T;H^1)\cap L^2(0,T;H^\s).
    \]
    Hence, by Lemma~\ref{LEM:estimates_B} with $\s=\a+1$,
    \[
        B(X^{\ss T}(\omega))\in L^1(0,T;H^\a).
    \]
    Therefore the map
    \[
        t\mapsto X^{\ss T}_0(\omega)-\int_0^t B(X^{\ss T}_s(\omega))\d s
    \]
    belongs to $C([0,T];H^\a)$.
    Since also $t\mapsto X^{\ss T}_t(\omega)$ is continuous in $H^\a_w$, every pairing with $y\in H^{\a-1}$ is continuous in $t$.
    By continuity in $t$ and density of $D$ in $H^{\a-1}$, \eqref{EQ:euler_identity_fixed_t_y} extends to all $t\in[0,T]$ and all $y\in H^{\a-1}$.
    Hence, for every $\omega\in\Omega^{{\ss T}\ast}$,
    \[
        X^{\ss T}_t(\omega)+\int_0^t B(X^{\ss T}_s(\omega))\d s=X^{\ss T}_0(\omega)
        \qquad\text{in }H^{1-\a},
        \qquad \forall\, t\in[0,T].
    \]
    
\textit{\textbf{Step $\mathbf{3}$.}}
    By Lemma~\ref{LEM:nu_to_0} \itref{IT:law_limit_time_Ha+1} with $t=0$, there exists an $H^\s$-valued random variable $\widehat X^{\ss T}_0$ on $(\Omega^{\ss T},\F^{\ss T},\P^{\ss T})$ such that $\widehat X^{\ss T}_0=X^{\ss T}_0$, $\P^{\ss T}$-a.s. in $U'$ and $(\widehat X^{\ss T}_0)_\ast\P^{\ss T}=\mu$ in $\mathscr P(H^\s)$.
    
    Define
    \[
        Z^{\ss T}_t:=\widehat X^{\ss T}_0-\int_0^t B(X^{\ss T}(s))\d s,
        \qquad t\in[0,T].
    \]
    Since $B(X^{\ss T})\in L^1(0,T;H^\a)$ $\P^{\ss T}$-a.s., we have $Z^{\ss T}\in C([0,T];H^\a)$ $\P^{\ss T}$-a.s.
    Moreover, from \textit{Step~2} and the identity $\widehat X^{\ss T}_0=X^{\ss T}_0$ in $U'$, we obtain
    \[
        Z^{\ss T}_t=X^{\ss T}_t
        \qquad \P^{\ss T}\text{-a.s. in }U',
        \qquad \forall\, t\in[0,T].
    \]

    Let
    \[
        Y^{\ss T}_t:=\Phi(t,\widehat X^{\ss T}_0),
        \qquad t\in[0,T],
    \]
    where $\Phi$ is the Euler flow from Theorem~\ref{TH:solution_euler}.
    Then $Y^{\ss T}\in C([0,T];H^\s)$, $\P^{\ss T}$-a.s.,
    and
    \[
        Y^{\ss T}(t)+\int_0^t B(Y^{\ss T}(s))\d s=\widehat X^{\ss T}_0
        \qquad\text{in }H^\a,
        \qquad \forall\, t\in[0,T].
    \]

    Since $Z^{\ss T}=X^{\ss T}$ in $U'$ and both are $H^\a$-valued, they coincide almost everywhere as $H^\a$-valued functions, hence $B(Z^{\ss T})=B(X^{\ss T})$ almost everywhere on $(0,T)$.
    Therefore
    \[
        Z^{\ss T}_t+\int_0^t B(Z^{\ss T}_s)\d s=\widehat X^{\ss T}_0
        \qquad\text{in }H^\a,
        \qquad \forall\, t\in[0,T].
    \]
    Subtracting the two integral identities and using  \eqref{EQ:estim_B_1} from  Lemma~\ref{LEM:estimates_B} with $\s=\a+1$, we obtain $\P^{\ss T}$-a.s., for all $t\in[0,T]$,
    \[
        \|Z^{\ss T}_t-Y^{\ss T}_t\|_{H^\a}
        \le
        C\int_0^t
        \big(\|Z^{\ss T}_s\|_{H^\s}+\|Y^{\ss T}_s\|_{H^\s}\big)\,
        \|Z^{\ss T}_s-Y^{\ss T}_s\|_{H^\a}\d s.
    \]
    Since $Z^{\ss T},Y^{\ss T}\in L^1(0,T;H^\s)$ $\P^{\ss T}$-a.s., Gr\"onwall's lemma yields $Z^{\ss T}_t=Y^{\ss T}_t$ in $H^\a$, for all $t\in[0,T]$, $\P^{\ss T}$-a.s.
    Hence
    \begin{equation}
    \label{EQ:XT_equals_flow}
        X^{\ss T}_t=\Phi(t,\widehat X^{\ss T}_0)
        \qquad\text{in }U',
        \qquad \forall\, t\in[0,T],
        \qquad \P^{\ss T}\text{-a.s.}
    \end{equation}
    
\textit{\textbf{Step $\mathbf{4}$.}}
    We prove that $\mu$ is invariant for the Euler Equation.

    Fix $t\ge0$ and choose $T>t$.
    Perform the construction of \textit{Steps~2, 3} on the interval $[0,T]$.
    By Lemma~\ref{LEM:nu_to_0} \itref{IT:law_limit_time_Ha+1} at time $t$, there exists an $H^\s$-valued random variable $\widehat X^{\ss T}_t$ such that $\widehat X^{\ss T}_t=X^{\ss T}_t$, $\P^{\ss T}$-a.s. in $U'$, and $(\widehat X^{\ss T}_t)_\ast\P^{\ss T}=\mu$ in $\mathscr P(H^\s)$.
    From \eqref{EQ:XT_equals_flow}, we infer 
    \[
        \widehat X^{\ss T}_t=\Phi(t,\widehat X^{\ss T}_0)
        \qquad \P^{\ss T}\text{-a.s. in }H^\s,
    \]
    because the embedding $H^\s\hookrightarrow U'$ is injective.
    Hence
    \[
        (\Phi_t)_\ast\mu
        =
        (\Phi(t,\widehat X^{\ss T}_0))_\ast\P^{\ss T}
        =
        (\widehat X^{\ss T}_t)_\ast\P^{\ss T}
        =
        \mu.
    \]
    Since $t\ge0$ was arbitrary, $\mu$ is an invariant measure for the Euler
    Equation~\eqref{EQ:euler}.
    
\textit{\textbf{Step $\mathbf{5}$.}}
    We now construct the stationary Euler process.

    Set
    \[
        \tilde\Omega:=H^\s,
        \qquad
        \tilde\F:=\mathscr B_{H^\s},
        \qquad
        \tilde\P:=\mu,
    \]
    and let $\xi:\tilde\Omega\to H^\s$ be the identity map.
    Define
    \[
        \tilde X_t(\omega):=\Phi(t,\xi(\omega)),
        \qquad t\ge0,\ \omega\in\tilde\Omega.
    \]
    Since $\Phi:\R_+\times H^\s\to H^\s$ is continuous by Theorem~\ref{TH:solution_euler}, the process $\tilde X$ is measurable.
    Let $\{\tilde\F_t\}_{t\ge0}$ be the augmentation of its natural filtration.
    Then $\tilde X$ is adapted, every path belongs to $C(\R_+;H^\s)$, and every path solves the Euler Equation~\eqref{EQ:euler}.

    Moreover,
    \[
        (\tilde X_t)_\ast\tilde\P=(\Phi_t)_\ast\mu=\mu,
        \qquad \forall\, t\ge0.
    \]
    Thus the process is marginally stationary.

\medskip
\textit{\textbf{Step $\mathbf{6}$.}}
    It remains to prove the integrability estimates.

    Fix $T>0$ and consider again the auxiliary objects constructed on $[0,T]$.
    Since $Y^{\ss T}_t=X^{\ss T}_t$ in $U'$ for all $t\in[0,T]$ $\P^{\ss T}$-a.s., and both are $H^1$-valued almost everywhere, they coincide almost everywhere as $H^1$-valued functions; similarly, since both are $H^\s$-valued almost everywhere, they coincide almost everywhere as $H^\s$-valued functions.
    Hence, by Lemma~\ref{LEM:nu_to_0} \itref{IT:X_reg_limit},
    \[
        Y^{\ss T}\in
        \bigcap_{p\ge2}L^p\big(\Omega^{\ss T};L^\infty(0,T;H^1)\big)
        \cap
        L^2\big(\Omega^{\ss T}\times(0,T);H^\s\big).
    \]

    On the other hand, $(\widehat X^{\ss T}_0)_\ast\P^{\ss T}=\mu=\xi_\ast\tilde\P$, and the map
    \[
        \Psi^{\ss T}:H^\s\to C([0,T];H^\s),
        \qquad
        \Psi^{\ss T}(x):[0,T]\ni t\mapsto \Phi(t,x)\in H^\s,
    \]
    is continuous.
    Therefore the $C([0,T];H^\s)$-valued random variables $Y^{\ss T}=\Psi^{\ss T}(\widehat X^{\ss T}_0)$ and $\tilde X|_{[0,T]}=\Psi^{\ss T}(\xi)$ have the same law.
    Hence
    \[
        \tilde X\in
        \bigcap_{p\ge2}L^p\big(\tilde\Omega;L^\infty(0,T;H^1)\big)
        \cap
        L^2\big(\tilde\Omega\times(0,T);H^\s\big).
    \]
    The case $1\le p<2$ follows from H\"older's inequality.
\end{proof}

\subsection{Moment estimates}

This last part is devoted to providing some moment estimates to the invariant measure  for the Euler Equation constructed above.

\begin{theorem}
    Assume that $\s>2$.
    There exists a finite constant $C_\s>0$ such that, if $0<\b<(2 C_\s)^{-1}$, and $\mu\in\mathscr P(H^\s)$ is the invariant measure for the Euler Equation~\eqref{EQ:euler} provided by Theorem~\ref{TH:nu_to_0}, then 
    \begin{align}
        &\int_{H^{\s}}\|x\|^2_{H^{\s}}\d \mu(x)\leq C_\s,\\
        &\int_{H^{\s}}\|x\|^{2n}_{H^{1}}\d \mu(x)\leq (2n-1)!!\, C_\s^n, \qquad \forall\, n\in\N,\\
        &\int_{H^{\s}} e^{\b\|x\|^2_{H^{\s}}}\d \mu(x)\leq 2e^{\frac{2\b C_\s}{1-2\b C_\s}}.
    \end{align}
\end{theorem}
\begin{proof} 
    Let us set $\a:=\s-1$, so that $\a>1$, and let $\iota:H^{2\a}\hookrightarrow H^{\a+1}=H^\s$, $\upsilon:H^\s\hookrightarrow U'$ denote the natural embeddings. 
    We also set $\tilde\iota:=\upsilon\circ\iota:H^{2\a}\hookrightarrow U'$.
    Let $C_\a>0$ be the constant from Theorem~\ref{TH:moment_estimates_mu_nu}, and define $C_\s:=C_\a$.
    By Proposition~\ref{PROP:approx_inv_measures}, there exists an infinitesimal strictly positive sequence $\{\nu_n\}_{n\in\N}$, invariant measures $\mu_n\in\mathscr P(H^{2\a})$ for SHNS$_{\nu_n,\a}$, and an invariant measure $\mu\in\mathscr P(H^\s)$ for the Euler Equation~\eqref{EQ:euler} such that
    \[
        \iota_\ast\mu_n\longrightarrow\mu,
        \qquad \textit{in }\mathscr P(H^\s_{bw}).
    \]
    By Theorem~\ref{TH:nu_to_0}, the measure $\mu$ is an invariant measure for the Euler Equation~\eqref{EQ:euler}.
    For every $n\in\N$,  let $(\Omega,\F,\{\F_t\}_{t\ge0},\P;W;X^n)$ be a marginally stationary solution such that $(X_t^n)_\ast\P=\mu_n$ for all $t\geq 0$.
    Fix $T>0$ and apply Lemma~\ref{LEM:nu_to_0} to the sequence $\{(\nu_n,\mu_n,X^n)\}_{n\in\N}$.
    We obtain a strictly increasing map $\N\ni j\mapsto n_j\in\N$, a probability space $(\Omega^{\ss T},\F^{\ss T},\P^{\ss T})$, random variables $X^{\ss T},X^{{\ss T},j}:\Omega^{\ss T}\to\mathcal Z_T$, for $j\in\N$, and,  for $t=0$, an $H^\s$-valued random variable $\widehat X^{\ss T}_0$ such that $\widehat X^{\ss T}_0=X^{\ss T}_0$, $\P^{\ss T}$-a.s. in $U'$, and  $(\widehat X^{\ss T}_0)_\ast\P^{\ss T}=\mu$ in $\mathscr P(H^\s)$.
    
    Let us now fix $m\in\N$ and $\b>0$.
    We define the Hilbert spaces
    \[
        E_1=E_3:=H^{\s}, \quad E_2:=H^{1},
    \]
    with the respective norms $\|\,\cdot\,\|^{}_{E_i}$, $i=1, 2, 3$, 
    and the functions $\varphi_i:\R_+\to\R_+$, $i=1, 2, 3$, by 
    \[
        \varphi_1(r):=r^2, \quad \varphi_2(r):=r^{2n}, \quad \varphi_3(r):=e^{\b r^2}, \qquad \forall\, r\geq 0.
    \]
    Fix $i\in\{1,2,3\}$.
    Since each $E_i$ is compactly embedded into $U'$, Example~2.12
    in \cite{Brzezniak+Serrano_2013_Optimal_relaxed_control_dissipative_stochastic_partial_differential_equations_Banach_spaces}
    implies that the functions $\eta_i:U'\to[0,+\infty]$ defined by
    \[
        \eta_i(x):=
        \begin{cases}
            \varphi_i\big(\|x\|_{E_i}\big), & x\in E_i,\\
            +\infty, & x\in U'\setminus E_i,
        \end{cases}
    \]
    are lower semicontinuous on $U'$. 
    see \cite[Section $1.3$]{brezis_2010_functional_analysis_Sobolev} and the references therein for the definition and the main properties.
    In particular, if $x\in U'$ and if $\{x_n\}_{n\in\N}$ is a sequence in $U'$ convergent to $x$ in $U'$, then $\eta_i(x)\leq \liminf_{n\to\infty}\eta_i(x_n)$.

    By Lemma~\ref{LEM:nu_to_0} \itref{IT:Xj_to_X_in_ZT}, we have $X^{{\ss T},j}\longrightarrow X^{\ss T}$ in $\mathcal Z_T$, $\P^{\ss T}$-a.s..
    Since $\mathcal Z_T\hookrightarrow C\big([0,T];U'\big)$, it follows that 
    \[
        X^{{\ss T},j}_0\longrightarrow X^{\ss T}_0, \qquad \textit{in }U', \ \P^{\ss T}-a.s.,\  \textit{as }j\to\infty.
    \]
    Therefore, by lower semicontinuity of $\eta_i$,
    \[
        \eta_i(\upsilon X^{\ss T}_0)\leq \liminf_{j\to\infty}\eta_i(\tilde\iota X^{{\ss T},j}_0), \qquad \tilde\P-a.s.
    \]
    Let us take the expectation with respect to $\P^{\ss T}$, that we denote by $\E^{\ss T}$, to both members of this inequality and use Fatou's Lemma. 
    We obtain
    \begin{equation}
    \label{EQ:fatou_eta_euler_mom}
        \E^{\ss T}\big[\eta_i(\upsilon\widehat X^{\ss T}_0)\big]
        \le
        \liminf_{j\to\infty}
        \E^{\ss T}\big[\eta_i(\tilde\iota X_0^{{\ss T},j})\big].
    \end{equation}

    We now identify the laws of the random variables appearing above.
    By Lemma~\ref{LEM:nu_to_0} \itref{IT:law_Xj_equals_law_Xnj_on_ZT} and the continuity of the evaluation map $e_0:\mathcal Z_T\to U'$, $e_0(u):=u(0)$, we have
    \[
        (\tilde\iota X_0^{{\ss T},j})_\ast\P^{\ss T}
        =
        (e_0)_\ast (X^{{\ss T},j})_\ast\P^{\ss T}
        =
        (e_0)_\ast (X^{n_j})_\ast\P
        =
        \tilde\iota_\ast\mu_{n_j},
        \qquad \forall\, j\in\N.
    \]
    Moreover, by Lemma~\ref{LEM:nu_to_0} \itref{IT:law_limit_time_Ha+1}, $(\widehat X^{\ss T}_0)_\ast\P^{\ss T}=\mu$ in $\mathscr P(H^\s)$, hence $(\upsilon\widehat X^{\ss T}_0)_\ast\P^{\ss T}=\upsilon_\ast\mu$ in $\mathscr P(U')$.

    Since $\eta_i\circ\upsilon(x)=\varphi_i(\|x\|_{E_i})$ for every $x\in H^\s$,
    and $\eta_i\circ\tilde\iota(x)=\varphi_i(\|x\|_{E_i})$ for every
    $x\in H^{2\a}$, from \eqref{EQ:fatou_eta_euler_mom} we conclude that
    \begin{equation}
    \label{EQ:master_moment_bound_euler}
        \int_{H^\s}\varphi_i\big(\|x\|_{E_i}\big)\d\mu(x)
        \le
        \liminf_{j\to\infty}
        \int_{H^{2\a}}\varphi_i\big(\|x\|_{E_i}\big)\d\mu_{n_j}(x).
    \end{equation}

    Finally apply Theorem~\ref{TH:moment_estimates_mu_nu} to conclude.
\end{proof}


\begin{appendices}

\section{Useful topological results}
\label{APP:proofs}

\begin{remark}
    If a function $f:\mathcal X\to \mathcal Y$ between topological spaces is continuous, then it is sequentially continuous, while the converse is generally false, \textit{i.e.} $C(\mathcal X;\mathcal Y)\subsetneq SC(\mathcal X;\mathcal Y)$, see \cite[Example $1$, Chapter $3$]{Arkhangelskii+Fedorchuk_1990_General_Topology_I}.

    A topological space $\mathcal X$ is said to be sequential if sequentially closed subsets are topologically closed.
    Therefore, under the additional assumption of $\mathcal X$ being sequential, a function $f:\mathcal X\to Y$ is sequentially continuous if and only if it is continuous, see \cite[Proposition $3$, Chapter $3$]{Arkhangelskii+Fedorchuk_1990_General_Topology_I}.
    As a matter of fact, this universal property characterises sequential spaces: a topological space $\mathcal X$ is sequential if and only if for any topological space $\mathcal Y$, we have $C(\mathcal X;\mathcal Y)=SC(\mathcal X;\mathcal Y)$.
    Let us recall, see, \textit{e.g.}, \cite[{Chapters $1$, $2$}]{Arkhangelskii+Fedorchuk_1990_General_Topology_I}, that metrizable spaces are sequential spaces.
\end{remark}

\begin{remark}
\label{REM:weak_topology_not_sequential}
    If $\mathcal X$ is an infinite-dimensional Banach space, then the topological space $\mathcal X_w$ is not generally sequential, see \cite[Theorem $1.5$]{Gabriyelyan+Kakol+Plebanek_2016_Ascoli_property}, in particular, it is not generally metrizable.
\end{remark}

\begin{lemma}
\label{LEM:topologies_inclusions}
    Assuming $\mathcal X$ to be a Banach space, we have:
    \begin{equation}
    \label{EQ:inclusions_topologies}
        \tau_{\mathcal X}^{w}\subset \tau_{\mathcal X}^{bw}\subset \tau_{\mathcal X}^s.
    \end{equation}
    In particular, for any topological space $\mathcal Y$, we have: $C(\mathcal X_{w};\mathcal Y)\subset C(\mathcal X_{bw};\mathcal Y)\subset C(\mathcal X;\mathcal Y)$. 
\end{lemma}
\begin{proof}
    We show the first inclusion in equation~\eqref{EQ:inclusions_topologies}. 
    Fix $O\in \tau^{w}_{\mathcal X}$, and a closed ball $\bar B^\mathcal X$ in $\mathcal X$.
    $\bar B^\mathcal X$ is closed in $\mathcal X_w$ by \cite[Theorem $3.7$]{brezis_2010_functional_analysis_Sobolev}. 
    Moreover, $O^c$ is closed in $\mathcal X_w$ because $O\in \tau^w_{\mathcal X}$.
    Therefore, $O^c\cap \bar B^\mathcal X$ is the intersection of closed sets in $\mathcal X_w$, hence it is closed in $\mathcal X_w$.
    By arbitrariness of $\bar B^\mathcal X$ and definition of bounded weak topology, this implies that $O^c$ is closed in $\mathcal X_{bw}$, hence $O\in \tau_{\mathcal X}^{bw}$.

    We now prove the second inclusion in \eqref{EQ:inclusions_topologies}.
    Fix $O\in \tau^{bw}_{\mathcal X}$, then we will show that $O^c$ is sequentially closed in $\mathcal X$, which implies that $O\in \tau^s_{\mathcal X}$ because $\mathcal X$ is sequential.
    Fix a sequence $\{x_n\}_{n\in\N}$ of points in $O^c$ convergent in $\mathcal X$ to $x\in\mathcal X$.
    Since the strong topology is larger than the weak topology, the sequence $\{x_n\}_{n\in\N}$ converges to $x$ in $\mathcal X_w$.
    Moreover, assuming $\|\,\cdot\,\|$ to be the norm on $\mathcal X$, and letting $R:=\sup_{n\in\N}\|x_n\|$, we consider the strongly closed ball $\bar B^{\mathcal X}_R:=\{y\in\mathcal X \ : \ \|y\|\leq R\}$.
    By direct inspection, $x_n\in O^c\cap \bar B^{\mathcal X}_R$ for all $n\in\N$ and, by definition of bounded weak topology, $O^c\cap \bar B^{\mathcal X}_R$ is closed in $\mathcal X_w$.
    Therefore, $x\in \big(O^c\cap \bar B^{\mathcal X}_R\big)\subset O^c$.
\end{proof}

\begin{lemma}
\label{LEM:weak_compacts_inclusions}
    If $\mathcal X$ is a Banach space, then compact sets in $\mathcal X_w$ are compact in $\mathcal X_{bw}$.
\end{lemma}
\begin{proof}
    Fix a compact set $K$ in $\mathcal X_w$ and a collection $\mathcal C$ of open sets in $\mathcal X_{bw}$ that cover $K$.
    $K$ is bounded in $\mathcal X$ by \cite[Corollary $2.4$ and Exercise $3.1$]{brezis_2010_functional_analysis_Sobolev}, hence there exists a strongly closed ball $\bar B^\mathcal X\subset \mathcal X$ such that $K\subset \bar B^\mathcal X$.
    For any $O\in\mathcal C$, we know that $O^c\cap \bar B^\mathcal X$ is closed in $\mathcal X_w$ by definition of bounded weak topology, hence $\big(O^c\cap \bar B^\mathcal X\big)^c=O\cup \big(\bar B^\mathcal X\big)^c$ is open in $\mathcal X_w$.
    Moreover, the collection $\big\{O\cup \big(\bar B^\mathcal X\big)^c\big\}_{O\in\mathcal C}$ covers $K$ because $O\subseteq O\cup \big(\bar B^\mathcal X\big)^c$ for any $O\in\mathcal C$ and $K$ is covered by $\mathcal C$.
    Therefore, by definition of compact set in $\mathcal X_w$, we can extract $O_1,\dots,O_N\in\mathcal C$ such that $\big\{O_n\cup \big(\bar B^\mathcal X\big)^c\big\}_{n=1}^N$ still covers $K$.
    However, $K\subset \bar B^\mathcal X$, hence 
    \[
        K=\bar B^\mathcal X\cap K
        \subseteq
        \bar B^\mathcal X\cap \bigcup_{n=1}^N\Big(O_n\cup (\bar B^\mathcal X)^c\Big)
        =
        \bigcup_{n=1}^N\big(\bar B^\mathcal X\cap O_n\big)
        \subseteq
        \bigcup_{n=1}^N O_n,
    \]
    which proves that $K$ is compact in $\mathcal X_{bw}$.
\end{proof}

\begin{lemma}
\label{LEM:SCw=Cbw}
    If $\mathcal H$ is a separable Hilbert space, then $SC(\mathcal H_w;\mathcal Y)=C(\mathcal H_{bw};\mathcal Y)$ for any topological space $\mathcal Y$.
\end{lemma}
\begin{proof}
    Let $f\in C(\mathcal H_{bw};\mathcal Y)$ and fix a convergent sequence $\{x_n\}_{n\in\N}$ in $\mathcal H_w$ to some limit $x\in \mathcal H$.
    Then, denoting $\|\,\cdot\,\|$ the norm in $\mathcal H$, there exists $k\in\N$ such that the strongly closed ball $\bar B:=\{y\in\mathcal H \ : \ \|y\|\leq k\}$ contains both $x$ and $x_n$ for all $n\in\N$.
    If $C$ is a closed set in $\mathcal Y$, then its inverse image $f^{-1}(C)$  is closed in $\mathcal H_{bw}$, hence $f^{-1}(C)\cap \bar B$ is weakly closed by definition of bounded-weak topology. 
    However $f^{-1}(C)\cap \bar B$ is the inverse image of $C$ via the restriction $f\big|^{}_{\bar B}$ of $f$ to $\bar B$.
    Thus we proved that $f\big|^{}_{\bar B}\in C(\bar B_w;\mathcal Y)$.
    In particular, $f(x_n)=f\big|^{}_{\bar B}(x_n)$ tends to $f\big|^{}_{\bar B}(x)=f(x)$ in $\mathcal Y$ as $n\to\infty$. 
    Therefore, $f\in SC(\mathcal H_w;\mathcal Y)$.

    On the other hand, fix $f\in SC(\mathcal H_w;\mathcal Y)$, consider a strongly closed ball $\bar B\subset \mathcal H$ and the restriction $f\big|^{}_{\bar B}$. 
    If $\{x_n\, : \, n\in\N\}\subset \bar B$ is a weakly convergent sequence, then its limit belongs to $\bar B$ by \cite[Theorem $3.7$]{brezis_2010_functional_analysis_Sobolev}, thus $f\big|^{}_{\bar B}\in SC(\bar B_w;\mathcal Y)$.
    However, $\bar B_w$ is metrizable by the Banach-Alaoglu Theorem, hence sequential.
    Therefore, $f\big|^{}_{\bar B}\in C(\bar B_w;\mathcal Y)$.
    This implies that, assuming $C$ to be a closed set in $\mathcal Y$, the inverse image $\big(f\big|^{}_{\bar B}\big)^{-1}(C)=f^{-1}(C)\cap \bar B$ is weakly closed, thus, by arbitrariness of $\bar B$,  $f^{-1}(C)$ is bounded-weakly closed.
    Being $C$ an arbitrary closed set in $\mathcal Y$, we infer that $f\in C(\mathcal H_{bw};\mathcal Y)$.
\end{proof}

\begin{lemma}
\label{LEM:weak_borel}
    If $\mathcal X$ is a separable Banach space, then the Borel $\s$-algebras generated by the strong, bounded weak, and weak topologies coincide.
    \[
        \mathscr B_{\mathcal X_w}=\mathscr B_{\mathcal X_{bw}}=\mathscr B_{\mathcal X}.
    \]    
    In particular, assuming $\mathcal H$ to be a separable Hilbert space, if a function $f:\mathcal H_w\to \R$ is sequentially continuous, then it is $\mathscr B_\mathcal H$-measurable. 
\end{lemma}
\begin{proof}
    For the proof of the first statement, we refer to \cite[Theorem $7.19$]{Zizler_2003_Nonseparable_Banch_spaces} or \cite[Corollary $2.4$]{Edgar_1979_Measurability_Banach_space_II} for the proof of this statement. 
    See also \cite[Introduction]{Maslowski+Seidler_2001_Strong_Feller_solutions}.

    For the second statement, it is sufficient to notice that if $f\in SC(\mathcal H_w)$, then $f\in C(\mathcal H_{bw})$ by Lemma~\ref{LEM:SCw=Cbw}, then  $f$ is $\mathscr B_{\mathcal H_{bw}}$-measurable.
    However $\mathscr B_{\mathcal H_{bw}}=\mathscr B_\mathcal H$.
\end{proof}

\begin{lemma}
\label{LEM:H_w_embedded_U}
    Assume that $\mathcal H$ is a separable Hilbert space and let $\mathcal Y$ be a Banach space.
    If $\mathcal H\doublehookrightarrow \mathcal Y$, then $\mathcal H_{bw}\hookrightarrow \mathcal Y$.    
\end{lemma}
\begin{proof}
    If indeed $\iota:\mathcal H\to \mathcal Y$ denotes the natural embedding, the assertion is equivalent to proving that $\iota\in C(\mathcal H_{bw};\mathcal Y)$, which is again equivalent to showing that $\iota\in SC(\mathcal H_{w};\mathcal Y)$ by Lemma~\ref{LEM:SCw=Cbw}.
    Therefore, let us fix a weakly convergent sequence $\{x_n\}_{n\in\N}$ in $\mathcal H_w$, then the sequence is convergent in $\mathcal Y$ because of the compact embedding $\mathcal H\doublehookrightarrow \mathcal Y$, and the assertion follows.
\end{proof}

\begin{lemma}
\label{LEM:Borel_X_cap_Y}
    If $\mathcal X,\mathcal Y$ are Polish spaces (\textit{i.e.}, separable completely metrizable topological spaces) such that $\mathcal X\hookrightarrow \mathcal Y$, then 
    \begin{equation}
    \label{EQ:Borel:X:cap_Y}
        \mathscr B_{\mathcal X}=\mathscr B_{\mathcal Y}\cap \mathcal X,
    \end{equation}
    where $\mathscr B_{\mathcal Y}\cap \mathcal X:=\{U\cap \mathcal X \ : \ U\in\mathscr B_{\mathcal Y}\}$.
    In particular, $\mathscr B_{\mathcal X}\subset \mathscr B_{\mathcal Y}$.
\end{lemma}
\begin{proof}
    If we denote by $\iota:\mathcal X\to \mathcal Y$ the natural embedding,  Kuratowski's Theorem \cite[Theorem $1$, Section V, Chapter $39$]{Kuratowski_1966_topology_1} implies that $\iota(\mathcal X)\in\mathscr B_{\mathcal Y}$, see also \cite[Theorem $15.1$]{Kechris_1995_classical_descriptive_set_theory}. 
    Then $\mathcal X\in\mathscr B_\mathcal Y$ because $\iota(\mathcal X)=\mathcal X$ by direct inspection.
    Then \cite[Theorem $4.3.3$]{Lojasiewicz_1988_an_introduction_to_the_theory_of_real_functions} yields $\mathscr B_{\mathcal X}=\mathscr B_{\mathcal Y}\cap \mathcal X$.

    Let now $E\in\mathscr B_{\mathcal X}$, then, by \eqref{EQ:Borel:X:cap_Y}, there exists $F\in\mathscr B_{\mathcal Y}$ such that $E=F\cap \mathcal X$. 
    Since $\mathcal X\in\mathscr B_{\mathcal Y}$ as discussed, we infer that $E\in\mathscr B_{\mathcal Y}$.
    Hence $\mathscr B_{\mathcal X}\subset \mathscr B_{\mathcal Y}$.
\end{proof}

\begin{remark}
    The function $\mathscr F[f]:\R^d\ni\xi\mapsto\mathscr F[f](\xi)\in\mathcal H$ is bounded and continuous, see Notation \ref{NOT:Fourier_trasnform}.
    By the Plancherel Theorem, the linear operator $L^1(\R^d;\mathcal H)\cap L^2(\R^d;\mathcal H)\ni f\mapsto \mathscr F[f]\in L^2(\R^d;\mathcal H)$    is a linear isometry with respect to the $L^2(\R^d;\mathcal H)$-norm. 
    By a density argument, this operator admits a unique extension to a unitary operator \[
        \mathscr F:L^2(\R^d;\mathcal H)\to L^2(\R^d;\mathcal H).
    \]
    In particular, for any $f,g\in L^2(\R^d;\mathcal H)$ we have:
    \begin{gather}
        \|f\|_{L^2(\R^d;\mathcal H)}=\big\|\mathscr F[f]\big\|_{L^2(\R^d;\mathcal H)}
    \label{EQ:unitary_fourier_1},\\
        \int_{\R^d}\big\lan f(x),g(x)\big\ran^{}_{\mathcal H}\d x =\int_{\R^d}\big\lan \mathscr F[f](\xi),\mathscr F[g](\xi)\big\ran^{}_{\mathcal H}\d \xi
    \label{EQ:unitary_fourier_2}.
    \end{gather}
    In addition, $\mathscr F$ is a continuous endomorphism when restricted to the Schwartz space $\mathcal S(\R^d;\mathcal H)$, endowed with its canonical LF topology. 
    Therefore, its unique transpose is well-defined and it is a continuous endomorphism, still denoted by $\mathscr F$, acting on the space of tempered distributions $\mathcal S'(\R^d;\mathcal H)$, where $\mathcal H$ is identified with its topological dual by the usual Riesz identification. 
    Moreover, by denoting ${\lan}\,\cdot\,,\,\cdot\,\ran^{}_{\ss{\!\mathcal H'\!\times\!\mathcal H}}$ the duality product, we have: 
    \begin{equation}
    \label{EQ:properties_Fourier_duality}
    \mathscr F\Big[{\big\lan}x,f\big\ran^{}_{\ss{\!\mathcal H'\!\times\!\mathcal H}}\Big]=
        {\big\lan}x,\mathscr F[f]\big\ran^{}_{\ss{\!\mathcal H'\!\times\!\mathcal H}}, \qquad \forall\, f\in L^2(\R^d;\mathcal H), \ \forall\, x\in \mathcal H'.
    \end{equation} 
\end{remark}

\section{Proofs and auxiliary results}
\label{APP:aux}

\begin{proof}[Proof of Lemma~\ref{LEM:stokes_operator_semigroup}]
    Set $m_a(\xi):=(1+|\xi|^2)^a$ for $\xi\in\R^2$ and $a\in\R$.
    We divide the proof into five steps.
    
\textit{\textbf{Step }$\mathbf{1.}$ Well-posedness, boundedness, bijectivity, and identities
    \eqref{EQ:LEM:stokes_operator_semigroup:graph_norm_Aa}.}
    Let $u\in H^{2\a+s}$.
    By definition of $A^\a$ and of the Sobolev norm,
    \[
        \|A^\a u\|_{H^s(\R^2;\R^2)}^2
        =\int_{\R^2}m_s(\xi) \big|m_\a(\xi)\mathscr F[u](\xi)\big|^2\d\xi
        = \int_{\R^2}m_{s+2\a}(\xi)|\mathscr F[u](\xi)|^2\d\xi
        =\|u\|_{H^{2\a+s}}^2.
    \]
    Hence $A^\a u \in H^s(\R^2;\R^2)$.
    Moreover, the divergence-free costraint is preserved because the multiplier $m_\a$ is scalar:
    \[
        \div(A^\a u)
        =i\mathscr F^{-1}\Big[x\cdot \mathscr F\big[A^\a u\big](x)\Big]
        =i\mathscr F^{-1}\Big[m_\a(x)x\cdot\mathscr F[u](x)\Big],\qquad \text{in }\mathcal S'(\R^2;\R^2),
    \]
    where we used Definitions~\ref{DEF:divergencefree_Sobolev_spaces},  \ref{DEF:stokes_operator}.
    The right-hand side vanishes because the condition $u\in H^{2\a+s}$ implies $\div u=0$, which entails $x\cdot\mathscr F[u](x)=0$ in the distributional sense in $\mathcal S'(\R^2;\R^2)$.
    Consequently, $A^\a:H^{2\a+s}\to H^s$ is well-defined and bounded,
    and the first identity in
    \eqref{EQ:LEM:stokes_operator_semigroup:graph_norm_Aa} holds.
    
    Now let $u,v\in H^{2\a+s}$. 
    Then, denoting the $2$-dimensional Lebesgue measure on $\R^2$ by $\mathscr L^2$:
    \[
    \begin{aligned}
        \lan A^\a u,A^\a v\ran_{H^s}
        = \int_{\R^2}m_s m_\a^2\, \mathscr F[u]\cdot \overline{\mathscr F[v]}\d\mathscr L^2
        = \int_{\R^2}m_{s+2\a} \mathscr F[u]\cdot \overline{\mathscr F[v]}\d\mathscr L^2
        = \lan u,v\ran_{H^{2\a+s}},
    \end{aligned}
    \]
    which proves the second identity in \eqref{EQ:LEM:stokes_operator_semigroup:graph_norm_Aa}.
    
    To prove bijectivity, let $y\in H^s$ and define $u:=\mathscr F^{-1}\big[m_\a^{-1}\mathscr F[y]\big]$.
    Since $m_\a^{-1}=m_{-\a}$, we have
    \[
        \|u\|_{H^{2\a+s}}^2
        = \int_{\R^2}m_{s+2\a}(\xi) \big|m_\a^{-1}(\xi)\mathscr F[y](\xi)\big|^2\d\xi
        = \|y\|_{H^s}^2<+\infty.
    \]
    Hence $u\in H^{2\a+s}$ and $A^\a u= \mathscr F^{-1}\big[m_\a\,m_\a^{-1}\mathscr F[y]\big] = y$.
    Therefore $A^\a$ is surjective. 
    Injectivity follows immediately from the first identity
    in \eqref{EQ:LEM:stokes_operator_semigroup:graph_norm_Aa}. Thus $A^\a$ is bijective.
    
    If $\a=1$, then for $u\in H^{2+s}$, $\mathscr F[Au](\xi) = (1+|\xi|^2)\mathscr F[u](\xi) = \mathscr F[u](\xi)-\mathscr F[\Delta u](\xi)$, because $\mathscr F[-\Delta u](\xi)=|\xi|^2\mathscr F[u](\xi)$. 
    Hence \eqref{EQ:LEM:stokes_operator_semigroup:A_Delta}.
    
\textit{\textbf{Step }$\mathbf{2.}$ Negativity and self-adjointness.}
    For $u,v\in H^{2\a+s}$, we have
    \[
    \begin{aligned}
        \lan  A^\a u,v\ran_{H^s}
        = \int_{\R^2}m_s m_\a\, \mathscr F[u]\cdot \overline{\mathscr F[v]}\d\mathscr L^2
        = \int_{\R^2}m_s \mathscr F[u]\cdot \overline{m_\a\mathscr F[v]}\d\mathscr L^2
        = \lan u, A^\a v\ran_{H^s}.
    \end{aligned}
    \]
    Thus $-\nu A^\a$, viewed as an unbounded operator on $H^s$ with domain $D_{H^s}(A^\a)$, is symmetric.
    For every $u\in H^{2\a+s}$,
    \[
        \lan  A^\a u,u\ran_{H^s}
        =  \int_{\R^2}m_s m_\a\big|\mathscr F[u]\big|^2\d\mathscr L^2
        \geq \int_{\R^2}m_s\big|\mathscr F[u]\big|^2\d\mathscr L^2
        = \|u\|_{H^s}^2,
    \]
    because $m_\a\ge1$. 
    Hence $A^\a$ is positive definite and $-\nu A^\a$ is negative definite. 
    
    We now prove self-adjointness. 
    Let $\lambda\in\C\setminus(-\infty,\nu]$. 
    Since $\operatorname{dist}\big(\lambda,(-\infty,\nu]\big)>0$, the functions $(-\nu m_\a-\lambda)^{-1}$ and $m_\a(-\nu m_\a-\lambda)^{-1}$ are bounded on $\R^2$. 
    Therefore, for every $y\in H^s$, the element 
    \[
        R_\lambda y
        :=
        \mathscr F^{-1}\Big[(-\nu m_\a-\lambda)^{-1}\mathscr F[y]\Big]
    \]
    belongs to $H^{2\a+s}$ and satisfies $(-\nu A^\a-\lambda \id_{H^{2s+\a}})R_\lambda y=y$, $R_\lambda (-\nu A^\a-\lambda \id_{H^{2s+\a}})u=u$ for all $u\in H^{2\a+s}$.
    Thus $\lambda$ is in the resolvent set of $-\nu A^\a$. 
    Since $-\nu A^\a$ is symmetric and has nonempty resolvent outside the real axis, it is self-adjoint on $H^s$.
    
\textit{\textbf{Step }$\mathbf{3.}$ Computation of the spectrum.}
    We already know from the previous step that the spectrum of $-\nu A^\a$ is contained in $(-\infty,-\nu]$.
    We now prove that it is exactly $(-\infty,-\nu]$, by exhibiting, for any $\lambda\in(-\infty,-\nu]$, a sequence $\{u_n\}_{n\in\N}$ in $H^{s+2\a}$ such that $\|u_n\|_{H^s}=1$ for all $n\in\N$ and $(-\nu A^\a-\lambda\id_{H^{s+2\a}})u_n\longrightarrow 0$ in $H^s$.
    This shows that $-\nu A^\a-\lambda \id_{H^{s+2\alpha}}$ cannot be invertible, because otherwise $u_n=(-\nu A^\a-\lambda \id_{H^{s+2\alpha}})^{-1}(-\nu A^\a-\lambda \id_{H^{s+2\alpha}})u_n$ would imply $u_n\to 0$ in $H^s$, contradicting $\|u_n\|_{H^s}=1$.
    
    If $\lambda<-\nu$, and $n\in\N$,
    let $\psi_n\in C_c^\infty(\R^2)$ have nonempty support satisfying 
    \[
        \supp\psi_n\subset \Big\{\xi\in\R^2:\ \big|-\nu m_\a(\xi)-\lambda\big|<1/n\Big\}.
    \]
    Define $\Phi_n(\xi):=(-\xi_2,\xi_1)\,\psi_n(\xi)$ so that $\xi\cdot \Phi_n(\xi)=0$ pointwise, and set
    \[
        \widehat{u_n}
        :=
        \frac{m_{-s/2}\Phi_n}
        {\|\Phi_n\|_{L^2(\R^2;\R^2)}}.
    \]
    Then, letting $u_n:=\mathscr F^{-1}[\widehat{u_n}]$, we have $\div u_n=0$, $u_n\in H^{s+2\a}$, and $\|u_n\|_{H^s}=1$.
    Moreover,
    \[
    \begin{aligned}
        \|(-\nu A^\a-\lambda \id_{H^{s+2\a}})u_n\|_{H^s}^2
        = \int_{\R^2}m_s \big|-\nu m_\a-\lambda\big|^2 |\widehat{u_n}|^2\d\mathscr L^2
        \le \frac{1}{n^2}.
    \end{aligned}
    \]
    Hence $(-\nu A^\a-\lambda I)u_n\longrightarrow0$ in $H^s$.
    
    If $\lambda=-\nu$, and $n\in\N$, let $\psi_n\in C_c^\infty(\R^2)$  have nonempty support contained in $\big\{\xi\in\R^2:\ |\xi|<1/n\big\}$ and define $\Phi_n$ and $u_n$ exactly as above. 
    Since $|-\nu m_\a(\xi)+\nu|=\nu\big((1+|\xi|^2)^\a-1\big)\longrightarrow0$ as $\xi\to0$, the same computation gives $\|(-\nu A^\a+\nu \id_{H^{s+2\a}})u_n\|_{H^s}\longrightarrow0$.
    
\textit{\textbf{Step }$\mathbf{4.}$ Generation of the semigroup and analyticity.}
    Since $\mathcal S(\R^2;\R^2)\cap H^s$ is contained in $H^{2\a+s}$ and is dense in $H^s$, the domain $D_{H^s}(A^\a)=H^{2\a+s}$ is dense in $H^s$.
    For $t\ge0$ and $x\in H^s$, define
    \[
        T(t)x
        :=
        \mathscr F^{-1}\big[e^{-\nu t m_\a}\mathscr F[x]\big].
    \]
    Since $0<e^{-\nu t m_\a(\xi)}\le1$, we have
    \[
    \begin{aligned}
        \|T(t)x\|_{H^s}^2
        = \int_{\R^2}m_s e^{-2\nu t m_\a} \big|\mathscr F[x]\big|^2\mathscr L^2 
        \leq  \|x\|_{H^s}^2.
    \end{aligned}
    \]
    Thus each $T(t)$ is a contraction on $H^s$. Since the multiplier is scalar, $T(t)$
    preserves the divergence-free constraint.
    The semigroup property follows from $ e^{-\nu (t+r)m_\a}=e^{-\nu t m_\a}e^{-\nu r m_\a}$, hence $T(t+r)=T(t)T(r)$ and $T(0)=\id_{H^s}$.
    
    We now prove strong continuity at $0$. 
    Let $x\in H^s$, then, by the Dominated Convergence Theorem:
    \[
        \|T(t)x-x\|_{H^s}^2
        = \int_{\R^2}m_s(\xi) \big |e^{-\nu t m_\a(\xi)}-1\big |^2 \big |\mathscr F[x](\xi)\big |^2\d\xi\longrightarrow 0,
    \]
    thanks to $e^{-\nu t m_\a}\longrightarrow1$ pointwise as $t\to 0^+$, and $|e^{-\nu t m_\a(\xi)}-1|^2\leq 4$.
    
    We now identify its generator. 
    Let $x\in H^{2\a+s}$, then
    \[
    \begin{aligned}
        \left\| \frac{T(t)x-x}{t}+\nu A^\a x \right\|_{H^s}^2
        = \int_{\R^2}m_s(\xi) \left| \frac{e^{-\nu t m_\a(\xi)}-1}{t}+\nu m_\a(\xi) \right|^2 |\mathscr F[x](\xi)|^2\d\xi.
    \end{aligned}
    \]
    For every $\xi\in\R^2$, the integrand converges to $0$ as $t\to 0^+$.
    In addition, by the mean value theorem, we have $|e^{-\nu t m_\a(\xi)}-1|\leq t\nu m_\a(\xi)$.
    Thus the integrand is controlled by $4\nu^2 m_sm_\a^2\big|\mathscr F[x]\big|^2$, which is integrable because $x\in H^{s+2\a}$.
    Then the Dominated Convergence Theorem yields
    \[
        \frac{T(t)x-x}{t}\longrightarrow -\nu A^\a x
        \qquad\text{in }H^s.
    \]
    Thus the generator of $\{T(t)\}_{t\ge0}$ is $-\nu A^\a$, and therefore
    \[
        T(t)=e^{-\nu tA^\a},
        \qquad t\ge0.
    \]
    
    Finally, we prove analyticity. 
    Let $\Sigma_{\pi/2}:=\{z\in\C:\ \Re z>0\}$.
    For $z\in\Sigma_{\pi/2}$ and $x\in H^{s}$, define $T(z)x:=\mathscr F^{-1}\big[e^{-\nu z m_\a}\mathscr F[x]\big]$.
    Since $\Re z>0$, we have $|e^{-\nu z m_\a(\xi)}|=e^{-\nu (\Re z)m_\a(\xi)}\le 1$, so $T(z)\in\mathcal L(H^s)$ and $\|T(z)x\|_{H^s}\leq \|x\|_{H^s}$.
    Moreover, for fixed $x,y\in H^s$,
    \[
        \lan T(z)x,y\ran_{H^s}
        =
        \int_{\R^2}m_s(\xi)
        e^{-\nu z m_\a(\xi)}
        \mathscr F[x](\xi)\cdot \overline{\mathscr F[y](\xi)} \d\xi.
    \]
    The integrand is holomorphic in $z$ and dominated by $m_s\big|\mathscr F[x]\big|\big|\mathscr F[y]\big|$, which belongs to $L^1(\R^2)$ by the Cauchy-Schwarz inequality.
    Hence $z\mapsto \lan T(z)x,y\ran_{H^s}$ is holomorphic on $\Sigma_{\pi/2}$.
    Therefore $\{e^{-\nu tA^\a}\}_{t\ge0}$ is an analytic semigroup on $H^s$.
    
\textit{\textbf{Step }$\mathbf{5.}$ Proof of \eqref{EQ:LEM:stokes_operator_semigroup:smoothing_estimate}.}
    Let $\theta\ge0$, $t>0$, and $x\in H^s$. 
    Then
    \[
    \begin{aligned}
        \|A^\theta e^{-\nu tA^\a}x\|_{H^s}^2
        &=
        \int_{\R^2}m_s(\xi)
        m_{\theta}^2(\xi)
        e^{-2\nu t m_\a(\xi)}
        \big|\mathscr F[x](\xi)\big|^2\d\xi.
    \end{aligned}
    \]
    Since $m_\theta =  m_\a^{\theta/\a}$, and
    \[
        \sup_{r\ge0} r^{\theta/\a}e^{-\nu tr}
        =
        t^{-\theta/\a}\sup_{\rho\ge0}\rho^{\theta/\a}e^{-\nu \rho}
        \le C_{\nu,\a,\theta}\, t^{-\theta/\a},
    \]
    we have $m_\theta(\xi) e^{-\nu t m_\a(\xi)} \le C_{\nu,\a,\theta}\, t^{-\theta/\a}$ for all $\xi\in\R^2$.
    Substituting this bound into the previous integral yields  \eqref{EQ:LEM:stokes_operator_semigroup:smoothing_estimate} and completes the proof.
\end{proof}

\begin{proof}[Proof of Lemma~\ref{LEM:b}]
    The well-posedness of the trilinear form comes from \cite[Lemma $2.1$]{Temam_1995_Navier-Stokes_equations_nonlinear_functional_analysis}.
    The two properties in equations \eqref{EQ:b_antysimmetry_1}, \eqref{EQ:b_antysimmetry_2} follow by a density argument from \cite[equations $(2.33)$, $(2.34)$]{Temam_1995_Navier-Stokes_equations_nonlinear_functional_analysis}.
    
    The last property \eqref{EQ:b(x,x,Ay)=0} is proved in \cite[Lemma $3.1$]{Temam_1995_Navier-Stokes_equations_nonlinear_functional_analysis} in the case of periodic boundary conditions, however, its proof can be adapted also to the case of the unbounded domain $\R^2$ {as follows}.
    Let us suppose that $u\in C_c^{\infty}(\R^2;\R^2)$ with $\div u=0$.
    Then, the identity~\eqref{EQ:b_antysimmetry_2} and integration by parts yield:
    \begin{align}
        b(u,u,Au)
        &=b(u,u,u)-\sum_{i,j,k=1}^2\int_{\R^2}\big(\partial_k^2u_j\big)u_i(\partial_iu_j)\d\mathscr L^2\\
        &=\sum_{i,j,k=1}^2\int_{\R^2}(\partial_ku_j)(\partial_ku_i)(\partial_iu_j)\d\mathscr L^2\
        +\sum_{i,j,k=1}^2\int_{\R^2}(\partial_ku_j)u_i\big(\partial_{ik}^2u_j\big)\d\mathscr L^2.
    \end{align}
    The first sum vanishes because a simple calculation shows:
    \[
        \sum_{i,j,k=1}^2(\partial_ku_j)(\partial_ku_i)(\partial_iu_j)
        =(\div u)\Bigg(\sum_{i,j=1}^2(\partial_iu_j)^2\Bigg)=0.
    \]
    The second sum vanishes because $(\partial_ku_j)\big(\partial_{ik}^2u_j)=\partial_i\big((\partial_ku_j)^2\big)/2$. Hence, integrating by parts:
    \begin{align}
        \sum_{i,j,k=1}^2\int_{\R^2}(\partial_ku_j)u_i(\partial^2_{ik}u_j)\d\mathscr L^2
        &=\frac12\sum_{i,j,k=1}^2\int_{\R^2}u_i\partial_i\big((\partial_ku_j)^2\big)\d\mathscr L^2\\
        &=-\frac12\int_{\R^2}(\div u)\Bigg(\sum_{j,k=1}^2(\partial_ku_j)^2\Bigg)\d\mathscr L^2=0.
    \end{align}
    The general case is obtained by recalling that compactly supported smooth vector fields are dense in $H^s(\R^2;\R^2)$, for any $s\in\R$, hence also in its subspace $H^s$.
\end{proof}

\begin{proof}[Proof of Lemma~\ref{LEM:estimates_B}]
    The first estimate is taken from \cite[Lemma $1.3$]{Kato+Ponce_1986_well_posedness_Euler_N-S}.
    Then the continuity of $B:H^\s\to H^{\s-1}$ follows directly from
    \begin{align}
        \|B(u)-B(y)\|^{}_{H^{\s-1}}&=
        \|B(u-v,u)+B(v,u-v)\|^{}_{H^{\s-1}}\\
        &\leq c\|u-v\|^{}_{H^{\s-1}}\|u\|^{}_{H^\s}+c\|v\|^{}_{H^{\s-1}}\|u-v\|^{}_{H^\s}\\
        &\leq c\big(\|u\|^{}_{H^{\s}}+\|v\|^{}_{H^{\s}}\big)\|u-v\|^{}_{H^\s}.
    \end{align}
    In particular, $B:H^\s\to H^{\s-1}$ is Lipschitz continuous on bounded sets of $H^\s$.
    The second estimate comes from the continuity of the bilinear operator $B$, see Definition~\ref{DEF:B} with $p=1$, $q=\e$, $r=0$.
    The third and fourth estimates follow from the continuity of the operators $b$ and $B$, see Definition~\ref{DEF:B} and Lemma~\ref{LEM:b} with $p=1$, $q=0$, $r=\e$.
    The continuity of $B:H^{\s-1}\to H^{2-\s}$ follows from the second estimate with $\varepsilon=\s-2$.
\end{proof}

\begin{proof}[Proof of Theorem~\ref{TH:z}]
    Adapting the notations from \cite[Section $5.1.1$]{DaPrato+Zabczyk_2014_stochastic_equations_infinite}, we see that the abstract Hilbert spaces $U$ and $H$ in the reference are both supposed to be the Sobolev-type space $H^\b$, and that the operator $B$ is, in our setting, simply $\sqrt\nu  \,\id_{H^\b}$.
    In addition, the abstract linear operator $A$ is replaced by the hyperviscous Stokes operator $-\nu A^{\a}$, which generates a contraction analytic semigroup $\big\{e^{- \nu t A^{\a}}\big\}_{t\geq 0}$ of linear bounded operators in $H^\b$. 
    The deterministic function $f$, and the initial random variable $\xi$ are both chosen to be $0$.
    Finally, we work under the assumption of the $H^\b$-valued Wiener process $W$, whose reproducing kernel Hilbert space and covariance operator will be denoted by $U_0$ and $Q_\b$ respectively.
    We recall that $Q_\b\in\mathcal L_1(H^\b)$, and that $U_0=Q_\b^{1/2}(H^\b)$, hence $U_0$ is Hilbert-Schmidt embedded into $H^\b$.
    
    The pathwise uniqueness result for the weak formulation and the validity of the mild formulation \eqref{EQ:mild_Z}, $\P-a.s.$ in $H^\b$, follow from \cite[Theorem $5.4$]{DaPrato+Zabczyk_2014_stochastic_equations_infinite}, while the pathwise continuity in $H^\b$ is stated in \cite[Theorem $5.11$]{DaPrato+Zabczyk_2014_stochastic_equations_infinite}.
    In order to apply these theorems, the following condition has to be verified: there exists $\g\in(0,1/2)$ such that
    \[
        \int_0^1t^{-2\g}\Tr\big[e^{-2\nu tA^{\a}}Q_\b\big]\d t <+\infty.
    \]
    Indeed, it is sufficient to notice that 
    \[
        \Tr\big[e^{-2\nu tA^{\a}}Q_\b\big]
        \leq \big\|e^{-2\nu tA^{\a}}\big\|_{\mathcal L(H^\b)}\Tr[Q_\b]\leq \Tr[Q_\b]<+\infty,
    \]
    because $Q_\b\in\mathcal L_1(H^\b)$ and the semigroup is contractive.

    As far as the strong formulation in equation~\eqref{EQ:strong_Z} is concerned, we refer to \cite[Section $5.6$]{DaPrato+Zabczyk_2014_stochastic_equations_infinite}.
    We make the following remarks to show that our situation can be reduced to that of \cite[Section $5.6$]{DaPrato+Zabczyk_2014_stochastic_equations_infinite}.
    If $\iota:H^\b\to H^{-2\a}$ denotes the natural embedding under the Riesz identification $H=H'$,  then $\iota W$ is an $H^{-2\a}$-valued Wiener process with covariance operator $Q:=\iota Q_\b\iota^\ast\in\mathcal L_1\big(H^{-2\a}\big)$ and with the same reproducing kernel Hilbert space  $Q^{1/2}\big(H^{-2\a}\big)=Q_\b^{1/2}(H^\b)=:U_0$ as $W$. 
    Hence, the abstract Hilbert space $H$ from \cite[Section $5.6$]{DaPrato+Zabczyk_2014_stochastic_equations_infinite} coincides with the Sobolev-type Hilbert space $H^{-2\a}$.
    We also need to set the hyperviscous Stokes operator on $H^{-2\a}$, that is to say
    \[
        A^\a:D_{H^{-2\a}}(A^\a)=H\subset H^{-2\a}\to H^{-2\a}.
    \]
    
    With these conventions, we seek to apply \cite[Theorem $5.38$]{DaPrato+Zabczyk_2014_stochastic_equations_infinite} in space $H^{-2\a}$. 
    Indeed, we show that the hypotheses of the cited theorem are satisfied.
    First, we need to check that $Q^{1/2}(H^{-2\a})=U_0$ is contained into $D_{H^{-2\a}}(A^\a)=H$: this is true because $U_0$ is embedded into $H^\b$, which is again embedded into $H$.
    Moreover, since $U_0$ is Hilbert-Schmidt embedded into $H^\b$, we also have the Hilbert-Schmidt embedding $\tilde\iota: U_0\hookrightarrow H$ and thus the Hilbert-Schmidt property for the operator $A^{\a}\tilde\iota  Q^{1/2}:H^{-2\a}\to H^{-2\a}$.
    Eventually, the validity of equation~\eqref{EQ:strong_Z} in $H^{-2\a}$ follows from \cite[Theorem $5.38$]{DaPrato+Zabczyk_2014_stochastic_equations_infinite}.
    In addition, the $\P-a.s.$ regularity $Z\in C\big([0,T];H^\b\big)$ discussed above, allows us to conclude that the equation~\eqref{EQ:strong_Z} is satisfied in $H^{\b-\a}$.
    
    We will now show the regularity in equation~\eqref{EQ:regularity_Z}.
    It follows from the It\^o formula, see \cite[Theorem $4.32$]{DaPrato+Zabczyk_2014_stochastic_equations_infinite} or \cite[Theorem $1.2$]{Pardoux_1979_stochastic_Partial_Differential_Equations}, applied to the $H^\b$-valued It\^o process $Z$ and to the function $F:H^\b\ni x\mapsto \|x\|^2_{H^\b}\in\R_+$.
    This $F$ is twice Fréchet differentiable, its differential at $x\in H^\b$ is the vector $F'(x)=2x\in H^{\b}$ and its second derivative at $x\in H^\b$ is the operator $F''(x)=2\id_{H^\b}\in\mathcal L(H^\b)$.
    In particular, $F\in C^2(H^\b)$. This allows us to apply the It\^o formula to the $H^\b$-valued It\^o process $Z$ and obtain, $\P-a.s.$
    \begin{equation}
        \|Z_t\|^2_{H^\b}
        + 2\nu\int_0^t\big\lan Z_s, A^{\a}Z_s\big\ran^{}_{\! H^\b}\d s
        = 2\sqrt \nu \int_0^t \big\lan Z_s,\d W_s\big\ran^{}_{\! H^\b}
        + {\nu} \Tr[Q_\b]t, \qquad \forall\, t\geq 0.
    \end{equation}
    Assume that $T>0$ and take the supremum for $t\in[0,T]$ to both members of the last equality, then compute the expectation. 
    We reach
    \begin{equation}
    \label{EQ:Ito_Z}
        \E\sup_{t\in[0,T]}\|Z_t\|^2_{H^\b} + 2\nu\E\int_0^T\big\|Z_s\|^2_{H^{\b+\a}}\d s
        \leq  2\sqrt \nu \, \E\sup_{t\in[0,T]}\bigg|\int_0^t \big\lan Z_s,\d W_s\big\ran^{}_{\! H^\b}\bigg|
        + {\nu} \Tr[Q_\b]T.
    \end{equation}
    We estimate the first term on the right-hand side, thanks to the Burkholder-Davis-Gundy inequality, see \cite[Theorem $48$, Chapter IV]{Protter_2005_Stochastic_integration_differential_equations}, and the Young inequality. 
    We recall indeed that $\int_0^{\,\cdot\,}\lan Z_s, \d W_s\ran ^{}_{H^\b}$ is a real-valued continuous martingale with quadratic variation process $\int_0^{\,\cdot\,}\|Z_s\|^{2}_{H^\b}\d s$.
    For an absolute constant $c>0$:
    \begin{align}
        2\sqrt \nu\, \E\sup_{t\in[0,T]} \bigg|\int_0^t\lan Z_s, \d W_s\ran ^{}_{H^\b} \bigg|
        &\leq  2c\sqrt\nu\,\E\bigg(\int_0^T\|Z_s\|^{2}_{H^\b}\d s\bigg)^{1/2} \\
        &\leq 2c\sqrt{\nu T}\,\E\sup_{t\in[0,T]}\|Z_t\|^{}_{H^\b}\\
        &\leq {2c^2\nu T}+\frac12 \E\sup_{t\in[0,T]}\|Z_t\|^2_{H^\b}.
    \end{align}
    We insert this estimate back into equation~\eqref{EQ:Ito_Z} and obtain
    \begin{align}
        \frac12\E\sup_{t\in[0,T]}\|Z_t\|^2_{H^\b}
        +
        2\nu\E\int_0^T\big\|Z_s\|^2_{H^{\b+\a}}\d s
        \leq  \big({2c^2}
        +  \Tr[Q_\b]\big)\nu T.
    \end{align}
    The sought regularity in equation~\eqref{EQ:regularity_Z} follows from this last estimate and the pathwise continuity, which we already discussed.
\end{proof}

We here present a simple lemma which clarifies the definition of solution for the deterministic problem~\eqref{EQ:v_nu_weak}.

\begin{lemma}
\label{LEM:v_mild_to_strong}
    Assume that $\nu>0$, $\a>1$, $x\in H^{2\a}$, $z\in C(\R_+;H^{2\a})$, and $v\in C(\R_+;H^{2\a})\cap L^2_{loc}(\R_+;H^{3\a})$.
    Then the following are equivalent:
    \begin{enumerate}[label=(\roman*), noitemsep]
        \item 
            $v$ satisfies the mild formulation in $H^{2\a}$, namely
            \begin{equation}
            \label{EQ:vmild}
                v(t)+\int_0^t e^{-\nu (t-s)A^\a}B\big(v(s)+z(s)\big)\d s
                =
                e^{-\nu tA^\a}x,
                \qquad \forall\, t\ge 0;
            \end{equation}
        \item 
            $v$ satisfies the integral strong formulation in $H^\a$, namely
            \begin{equation}
            \label{EQ:vstrong}
                v(t)+\nu\int_0^t A^\a v(s)\d s+\int_0^t B\big(v(s)+z(s)\big)\d s
                =
                x,
                \qquad \forall\, t\ge 0.
            \end{equation}
    \end{enumerate}
    Moreover, whenever one of the two equivalent conditions above holds, one has $v\in H^1_{loc}(\R_+;H^\a)\cap C^1(\R_+;H)$, and
    \begin{equation}
    \label{EQ:v_diff_Ha_ae}
        v'(t)+\nu A^\a v(t)+B\big(v(t)+z(t)\big)=0,
    \end{equation}
    for a.e. $t\geq 0$ in $H^\a$, and all $t\geq 0$ in $H$.
\end{lemma}
\begin{proof}
    Set
    \[
        S(t):=e^{-\nu tA^\a},
        \qquad
        g(t):=B\big(v(t)+z(t)\big),
        \qquad t\ge0.
    \]
    Since $v,z\in C(\R_+;H^{2\a})$, Lemma~\ref{LEM:estimates_B} yields $g\in C(\R_+;H^{2\a-1})$.
    Because $\a>1$, one has $g\in C(\R_+;H^\a)\subset L^2_{loc}(\R_+;H^\a)$.
    Moreover, since $v\in L^2_{loc}(\R_+;H^{3\a})$ and
    $A^\a:H^{3\a}\to H^\a$ is bounded by
    Lemma~\ref{LEM:stokes_operator_semigroup}, we also have $A^\a v\in L^2_{loc}(\R_+;H^\a)$.
    Hence \eqref{EQ:vstrong} implies $v\in H^1_{loc}(\R_+;H^\a)$ and \eqref{EQ:v_diff_Ha_ae} holds a.e. in $H^\a$.
    The $C^1(\R_+;H)$ regularity and the pointwise identity in $H$ will be proved later in the second step.
    
\textit{\textbf{Step }$\mathbf{1.}$ (i) implies (ii).}
    Assume that \eqref{EQ:vmild} holds. Then
    \[
        v(t)=S(t)x-w(t),
        \qquad
        w(t):=\int_0^t S(t-s)g(s)\d s,
        \qquad t\ge0.
    \]
    We first prove that
    \begin{equation}
    \label{EQ:LEM:v_mild_to_strong:semigroup_x}
        S(t)x+\nu\int_0^t A^\a S(r)x\d r=x
        \qquad\text{in }H^\a,\quad \forall\, t\ge0.
    \end{equation}
    Fix $t>0$. Since $x\in H^{2\a}$, Lemma~\ref{LEM:stokes_operator_semigroup}, applied with $s=2\a$ and $\theta=\a/2$, gives
    \[
        \|A^\a S(r)x\|_{H^\a}
        =
        \|S(r)x\|_{H^{3\a}}
        =
        \|A^{\a/2}S(r)x\|_{H^{2\a}}
        \le C r^{-1/2}\|x\|_{H^{2\a}},
        \qquad r>0.
    \]
    Hence $r\mapsto A^\a S(r)x$ belongs to $L^1(0,t;H^\a)$.
    Moreover, for every $\e\in(0,t)$, the map
    $r\mapsto S(r)x$ belongs to $C^1([\e,t];H^\a)$ and satisfies
    \[
        \frac{\d}{\d r}S(r)x=-\nu A^\a S(r)x
        \qquad\text{in }H^\a,\quad r\in[\e,t].
    \]
    Therefore
    \[
        S(t)x-S(\e)x=-\nu\int_\e^t A^\a S(r)x\d r
        \qquad\text{in }H^\a.
    \]
    Letting $\e\to 0^+$ and using the strong continuity of $S(\cdot)$ on
    $H^\a$, we obtain \eqref{EQ:LEM:v_mild_to_strong:semigroup_x}.
    
    We now prove that
    \begin{equation}
    \label{EQ:LEM:v_mild_to_strong:w_identity}
        w(t)+\nu\int_0^t A^\a w(r)\d r=\int_0^t g(s)\d s
        \qquad\text{in }H^\a,\quad \forall\, t\ge0.
    \end{equation}
    Fix $t>0$. 
    By Lemma~\ref{LEM:stokes_operator_semigroup}, applied with $s=2\a-1$ and $\theta=(\a+1)/2$, we have
    \[
        \|A^\a S(r-s)g(s)\|_{H^\a}
        =
        \|S(r-s)g(s)\|_{H^{3\a}}
        =
        \|A^{(\a+1)/2}S(r-s)g(s)\|_{H^{2\a-1}}
        \le
        C(r-s)^{-(\a+1)/(2\a)}\|g(s)\|_{H^{2\a-1}},
    \]
    for every $0\le s<r\le t$. 
    Since $(\a+1)/(2\a)<1$ and $g\in C([0,t];H^{2\a-1})$, the right-hand side is integrable on the triangle $\{(s,r):0\le s\le r\le t\}$. 
    Hence Fubini's theorem yields
    \begin{equation}
    \label{EQ:LEM:v_mild_to_strong:Fubini}
        \int_0^t A^\a w(r)\d r
        =
        \int_0^t\int_0^r A^\a S(r-s)g(s)\d s\d r
        =
        \int_0^t\int_s^t A^\a S(r-s)g(s)\d r\d s
    \end{equation}
    in $H^\a$. 
    Fix now $s\in[0,t)$. 
    Since $g(s)\in H^{2\a-1}$, the same estimate shows that $\rho\mapsto A^\a S(\rho)g(s)$ belongs to $L^1(0,t-s;H^\a)$. 
    Moreover, for every $\e\in(0,t-s)$, the map $\rho\mapsto S(\rho)g(s)$ belongs to $C^1([\e,t-s];H^\a)$ and satisfies
    \[
        \frac{\d}{\d \rho}S(\rho)g(s)=-\nu A^\a S(\rho)g(s)
        \qquad\text{in }H^\a.
    \]
    Hence
    \[
        S(t-s)g(s)-S(\e)g(s)
        =
        -\nu\int_\e^{t-s} A^\a S(\rho)g(s)\d \rho
        \qquad\text{in }H^\a.
    \]
    Letting $\e\to 0^+$ and using again the strong continuity of $S(\cdot)$ on $H^\a$, we obtain
    \[
        S(t-s)g(s)+\nu\int_0^{t-s} A^\a S(\rho)g(s)\d \rho=g(s)
        \qquad\text{in }H^\a.
    \]
    Integrating this identity with respect to $s\in[0,t]$ and using \eqref{EQ:LEM:v_mild_to_strong:Fubini}, we get
    \[
        w(t)+\nu\int_0^t A^\a w(r)\d r=\int_0^t g(s)\d s
        \qquad\text{in }H^\a,
    \]
    which is \eqref{EQ:LEM:v_mild_to_strong:w_identity}.
    
    Combining \eqref{EQ:LEM:v_mild_to_strong:semigroup_x},
    \eqref{EQ:LEM:v_mild_to_strong:w_identity}, and the identity
    $v=S(\cdot)x-w$, we conclude that, for every $t\ge0$,
    \[
    \begin{aligned}
        &v(t)+\nu\int_0^t A^\a v(r)\d r+\int_0^t g(r)\d r\\
        =&
        \big(S(t)x-w(t)\big)
        +\nu\int_0^t A^\a\big(S(r)x-w(r)\big)\d r
        +\int_0^t g(r)\d r\\
        =&
        \bigg(S(t)x+\nu\int_0^t A^\a S(r)x\d r\bigg)
        -
        \bigg(w(t)+\nu\int_0^t A^\a w(r)\d r\bigg)
        +\int_0^t g(r)\d r\\
        =&\,x
    \end{aligned}
    \]
    in $H^\a$. This proves \eqref{EQ:vstrong}.

\textit{\textbf{Step }$\mathbf{2.}$ (ii) implies (i).}    
    Assume now that \eqref{EQ:vstrong} holds. 
    Evaluating it at $t=0$, we get $v(0)=x$.
    Since $A^\a:H^{2\a}\to H$ is bounded by Lemma~\ref{LEM:stokes_operator_semigroup}, and $v\in C(\R_+;H^{2\a})$, one has $A^\a v\in C(\R_+;H)$.
    Moreover, $g=B(v+z)\in C(\R_+;H^{2\a-1})\subset C(\R_+;H)$.
    Therefore \eqref{EQ:vstrong}, viewed in $H$, reads
    \[
        v(t)=x-\int_0^t\big(\nu A^\a v(s)+g(s)\big)\d s,
        \qquad t\ge0,
    \]
    and the right-hand side is a $C^1(\R_+;H)$ function. 
    Hence $v\in C^1(\R_+;H)$ and 
    \begin{equation}
    \label{EQ:LEM:v_mild_to_strong:classical_H}
        v'(t)+\nu A^\a v(t)+g(t)=0,
        \qquad \text{ in }H, \quad \forall\, t\ge0.
    \end{equation}
    
    Fix $t>0$ and define
    \[
        \Phi(s):=S(t-s)v(s),
        \qquad s\in[0,t].
    \]
    Since $v\in C^1(\R_+;H)$ and $v(s)\in H^{2\a}=D_H(A^\a)$ for every $s\ge0$, the map $\Phi$ belongs to $C^1([0,t];H)$ and
    \[
        \Phi'(s)=\nu A^\a S(t-s)v(s)+S(t-s)v'(s),
        \qquad s\in[0,t].
    \]
    Because $v(s)\in D_H(A^\a)$, the semigroup commutes with $A^\a$ on $v(s)$, so, using \eqref{EQ:LEM:v_mild_to_strong:classical_H}, we obtain
    \[
    \begin{aligned}
        \Phi'(s)
        = S(t-s)\bigl(\nu A^\a v(s)+v'(s)\bigr)
        = -S(t-s)g(s),\qquad \text{ in }H, \quad \forall\, s\in[0,t].
    \end{aligned}
    \]
    Apply the Fundamental Theorem of Calculus in $H$ on the interval $[0,t]$, and recall the definition of $\Phi$ to obtain
    \begin{equation}
    \label{EQ:LEM:v_mild_to_strong:mild_H}
        v(t)-S(t)x=-\int_0^t S(t-s)g(s)\d s
        \qquad\text{in }H.
    \end{equation}
    
    It remains to upgrade \eqref{EQ:LEM:v_mild_to_strong:mild_H} to $H^{2\a}$.
    Fix $t>0$. 
    By Lemma~\ref{LEM:stokes_operator_semigroup}, applied with $s=2\a-1$ and $\theta=1/2$, we have
    \[
        \|S(t-s)g(s)\|_{H^{2\a}}
        =
        \|A^{1/2}S(t-s)g(s)\|_{H^{2\a-1}}
        \le
        C (t-s)^{-1/(2\a)}\|g(s)\|_{H^{2\a-1}},
        \qquad 0\le s<t.
    \]
    Since $1/(2\a)<1$ and $g\in C([0,t];H^{2\a-1})$, the right-hand side is integrable on $(0,t)$.
    Hence the right-hand side of \eqref{EQ:LEM:v_mild_to_strong:mild_H} is well-defined in $H^{2\a}$. 
    As also $v(t)\in H^{2\a}$ and $S(t)x\in H^{2\a}$,
    identity \eqref{EQ:LEM:v_mild_to_strong:mild_H} actually holds in $H^{2\a}$.
    This is exactly \eqref{EQ:vmild}.
\end{proof}

 We here present a version of It\^o's Lemma tailored to the SHNS$_{\nu,\a}$ Equation~\eqref{EQ:NSnu}.
 
\begin{lemma}
\label{LEM:Ito_formula_h(|x|^2_b)}
    Assume that $\b\geq 0$ and let $(\Omega, \F, \{\F_t\}_{t\geq 0}, \P)$ be an augmented filtered probability space with an adapted $H^\b$-valued Wiener process $W$.
    Let $Q_\b\in\mathcal L_1(H^\b)$ denote the covariance operator of $W_1$.
    Assume that $\g\in[0,\b]$, $\nu>0$, $\a>1$ and further assume that there exists an $H^\g$-valued process $X$ satisfying $\P-a.s.$ in $H^\g$:
    \begin{equation}
            X_t+ \nu \int_0^t A^{\a}X_s\d s +\int_0^tB(X_s)\d s=X_0+\sqrt \nu W_t,\qquad \forall\,t\geq0.
    \end{equation}
    Let $\iota:H^\b\to H^\g$ be the natural embedding and let $Q_\g:=\iota Q_\b\iota^\ast\in\mathcal L_1(H^\g)$.
    Assume that $h\in C^2(\R_+)$, then  $\P-a.s.$, for all $t\geq 0$
    \begin{equation}
    \label{EQ:ito_X_xi_h}
    \begin{aligned}
        h\big(\|X_t\|^2_{H^\g}\big)
        = h\big(\|X_0\|^2_{H^\g}\big)
        &-2\nu\int_0^t\big\|X_s\|^2_{H^{\g+\a}}h'\big(\|X_s\|^2_{H^\g}\big)\d s \\
        &-2\int_0^t\big\lan B(X_s), X_s\big\ran^{}_{\! H^\g}h'\big(\|X_s\|^2_{H^\g}\big)\d s \\
        &+ 2\sqrt \nu \int_0^th'\big(\|X_s\|^2_{H^\g}\big)\lan X_s, \d W_s\ran^{}_{\! H^\g} \\
        &+ {2\nu} \int_0^t\big\|Q_\g^{1/2}X_s\big\|^{2}_{H^\g}h''\big(\|X_s\|^2_{H^\g}\big)\d s\\
        &+ \nu \Tr[Q_\g]\int_0^th'\big(\|X_s\|^2_{H^\g}\big)\d s.
    \end{aligned}
    \end{equation}
\end{lemma}
\begin{proof}
    The thesis follows from the It\^o formula, see \cite[Theorem $4.32$]{DaPrato+Zabczyk_2014_stochastic_equations_infinite} or \cite[Theorem $1.2$]{Pardoux_1979_stochastic_Partial_Differential_Equations}, applied  to the It\^o process $X$ and to the function 
    \[
        F:H^\g\in x\mapsto F(x):=h\big(\|x\|^2_{H^\g}\big)\in\R_+.
    \]
    
    $F$ is Fréchet differentiable and its differential at $x\in H^\g$ is the vector $F'(x)\in H^{\g}$ such that
    \[
        F'(x)=2h'\big(\|x\|^2_{H^\g}\big)x.
    \]
    The function $F':H^\g\ni x\mapsto F'(x)\in H^\g$ is again Fréchet differentiable and its derivative at $x\in H^\g$ is the operator $F''(x)\in\mathcal L(H^\g)$ such that for all $y\in H^\g$
    \begin{align}
        F''(x)y
        &=4h''\big(\|x\|^2_{H^\g}\big)\lan x,y\ran^{}_{\! H^\g}x+2h'\big(\|x\|^2_{H^\g}\big)y\\
         &=
         \Big[4h''\big(\|x\|^2_{H^\g}\big)\|x\|^2_{H^\g}\Pi_x
         +2h'\big(\|x\|^2_{H^\g}\big)\id_{H^\g}\Big] y,
    \end{align}
    where $\id_{H^\g}$ is the identity operator in $H^\g$ and $\Pi_x$ denotes the projection onto the subspace generated by $x$, \textit{i.e.}:
    \[
        \Pi_x:H^\g\to H^\g, \qquad \Pi_xy:=
        \begin{cases}
            \ds \frac{\lan x,y\ran^{}_{\! H^\g}}{\|x\|^2_{H^\g}}x, \qquad &\textit{if }x\in H^\g\setminus\{0\}\\
            0&\textit{if }x=0
        \end{cases}.
    \] 
    Hence, the function $F'':H^\g\ni x\mapsto F''(x)\in \mathcal L(H^\g)$ is continuous, in particular, $F\in C^2(H^\g)$. This allows us to apply the It\^o formula and obtain, $\P-a.s.$ for all $t\geq 0$ in $\R$:
    \begin{equation}
    \begin{aligned}
        F(X_t)
        = F(X_0)
        &-\int_0^t\big\lan F'(X_s), \nu A^{\a}X_s+B(X_s)\big\ran^{}_{\! H^\g}\d s\\
        &+ \sqrt \nu \int_0^t \big\lan F'(X_s),\d W_s\big\ran^{}_{\! H^\g}\\
        &+ \frac{\nu}{2} \int_0^t\Tr\big[F''(X_s)Q_\g\big]\d s,
    \end{aligned}
    \end{equation}
    where $Q_\g:=\iota Q_\b\iota^\ast\in\mathcal L_1(H^\g)$ and $\iota:H^\b\to H^\g$ is the Sobolev embedding.
    The thesis follows after expanding $F, F'$ and $F''$ and observing that $\|X_s\|^2_{H^\g}\!\Tr[\Pi_{X_s}Q_\g]=\big\|Q_\g^{1/2}X_s\big\|^2_{H^\g}$. 
    Indeed, if $\{e_n\}_{n\in\N}$ is a complete orthonormal system for the separable Hilbert space $H^\g$, then by the self-adjointness of $\Pi_x$, $Q_\g$, and by the usual properties of Hilbert spaces, we have for all $x\in H^\g$:
    \begin{equation}
    \begin{aligned}
    \label{EQ:rewrite_Tr}
        \|x\|^2_{H^\g}\Tr[\Pi_xQ_\g]
        &=\|x\|^2_{H^\g}\sum_{n\in\N}\lan \Pi_xQ_\g e_n,e_n\ran^{}_{\! H^\g}\\
        &=\|x\|^2_{H^\g}\sum_{n\in\N}\lan Q_\g e_n,\Pi_xe_n\ran^{}_{\! H^\g}\\
        &=\sum_{n\in\N}\lan x,e_n\ran^{}_{\! H^\g}\lan e_n,Q_\g x\ran^{}_{\! H^\g}\\
        &=\lan x,Q_\g x\ran^{}_{\! H^\g}\\
        &=\big\|Q_\g^{1/2}x\big\|^2_{H^\g}.
    \end{aligned}
    \end{equation}
\end{proof}

\begin{lemma}
\label{LEM:nonlinear_loc_convergence}
    Assume that $\b>1$.
    \begin{enumerate}[label=\textup{$(\roman*)$}]
        \item
        \label{ITEM:nonlinear_loc_convergence_strong}
        Let $u\in L^2(0,T;H^{\b+1})$ and let $\{u_n\}_{n\in\N}$ be a sequence
        bounded in $L^2(0,T;H^{\b+1})$ and convergent to $u$ in
        $L^2(0,T;H^{\b+1}_{loc})$.
        Then, for every $y\in H^{-\b}$,
        \begin{equation}
        \label{EQ:nonlinear_loc_convergence_strong}
            \lim_{n\to\infty}
            \sup_{t\in[0,T]}
            \bigg|
                \int_0^t
                \big\langle
                    y,
                    B(u_n(s))-B(u(s))
                \big\rangle_{\ss{\!H^{-\b}\!\times\!H^\b}}
                \d s
            \bigg|
            =0.
        \end{equation}

        \item
        \label{ITEM:nonlinear_loc_convergence_dual}
        Let $u\in L^2(0,T;H^\b)$ and let $\{u_n\}_{n\in\N}$ be a sequence
        bounded in $L^2(0,T;H^\b)$ and convergent to $u$ in
        $L^2(0,T;H^\b_{loc})$.
        Then, for every $y\in H^{\b-1}$,
        \begin{equation}
        \label{EQ:nonlinear_loc_convergence_dual}
            \lim_{n\to\infty}
            \sup_{t\in[0,T]}
            \bigg|
                \int_0^t
                \big\langle
                    B(u_n(s))-B(u(s)),
                    y
                \big\rangle_{\ss{\!H^{1-\b}\!\times\!H^{\b-1}}}
                \d s
            \bigg|
            =0.
        \end{equation}
    \end{enumerate}
\end{lemma}

\begin{proof}[Proof of \itref{ITEM:nonlinear_loc_convergence_strong}]
\textit{\textbf{Step $\mathbf{1}$.}}
    Assume first that $y\in C_c^\infty(\R^2;\R^2)\cap H^{-\b}$.
    Choose $m\in\N$ such that $\supp y\subset \bar B_m$.
    Let $x,z\in H^{\b+1}$.
    Since $y$ is compactly supported, divergence free, and $\chi_m=1$ on $\bar B_m$,
    \[
        {\big\lan} y, B(x,z)\big\ran_{\ss{\!H^{-\b}\!\times\!H^\b}}
        =
        \int_{\R^2}[(x\cdot\nabla)z]\cdot y\d\mathscr L^2
        =
        \int_{\R^2}[(\chi_m x\cdot\nabla)(\chi_m z)]\cdot y\d\mathscr L^2.
    \]
    Hence
    \[
        \Big|{\big\lan} y, B(x,z)\big\ran_{\ss{\!H^{-\b}\!\times\!H^\b}}\Big|
        \leq
        \|(\chi_m x\cdot\nabla)(\chi_m z)\|_{H^\b(\R^2;\R^2)}
        \|y\|_{H^{-\b}}.
    \]
    Since $\b>1$, the Sobolev space $H^\b(\R^2)$ is an algebra, and multiplication by $\chi_m$ is bounded on $H^{\b+1}(\R^2;\R^2)$.
    Therefore there exists a constant $C_m>0$, depending only on $\b$ and on $\chi_m$, such that
    \[
        \|(\chi_m x\cdot\nabla)(\chi_m z)\|_{H^\b(\R^2;\R^2)}
        \leq
        C_m
        \|\chi_m x\|_{H^{\b+1}(\R^2;\R^2)}
        \|\chi_m z\|_{H^{\b+1}(\R^2;\R^2)}.
    \]
    Recalling the definition~\eqref{EQ:def_semin_n}, we conclude that
    \begin{equation}
    \label{EQ:local_B_estimate_compact_test_unified_strong}
        \Big|{\big\lan} y, B(x,z)\big\ran_{\ss{\!H^{-\b}\!\times\!H^\b}}\Big|
        \leq
        C_m
        [\![x]\!]_{H^{\b+1}_m}
        [\![z]\!]_{H^{\b+1}_m}
        \|y\|_{H^{-\b}}.
    \end{equation}

    \textit{\textbf{Step $\mathbf{2}$.}}
    Let now $y\in C_c^\infty(\R^2;\R^2)\cap H^{-\b}$ and let $m$ be chosen as above.
    For a.e.\ $s\in[0,T]$, by bilinearity of $B$ and by~\eqref{EQ:local_B_estimate_compact_test_unified_strong},
    \begin{align*}
        \Big|{\big\lan} y, B(u_n(s))-B(u(s))\big\ran_{\ss{\!H^{-\b}\!\times\!H^\b}}\Big|
        &=
        \Big|{\big\lan}
            y,
            B\big(u_n(s)-u(s),u_n(s)\big)
            +
            B\big(u(s),u_n(s)-u(s)\big)
        \big\ran_{\ss{\!H^{-\b}\!\times\!H^\b}}\Big|\\
        &\leq
        C_m
        [\![u_n(s)-u(s)]\!]_{H^{\b+1}_m}
        \Big(
            [\![u_n(s)]\!]_{H^{\b+1}_m}
            +
            [\![u(s)]\!]_{H^{\b+1}_m}
        \Big)
        \|y\|_{H^{-\b}}\\
        &\leq
        C_m
        [\![u_n(s)-u(s)]\!]_{H^{\b+1}_m}
        \big(
            \|u_n(s)\|_{H^{\b+1}}
            +
            \|u(s)\|_{H^{\b+1}}
        \big)
        \|y\|_{H^{-\b}}.
    \end{align*}
    Integrating over $[0,t]$, taking the supremum in $t\in[0,T]$, and using H\"older's inequality, we get
    \begin{gather}
        \sup_{t\in[0,T]}
        \bigg|
            \int_0^t
            {\big\lan} y, B(u_n(s))-B(u(s))\big\ran_{\ss{\!H^{-\b}\!\times\!H^\b}}
            \d s
        \bigg|\\
        \qquad\leq
        C_m
        [\![u_n-u]\!]_{L^2(0,T;H^{\b+1}_m)}
        \Big(
            \|u_n\|_{L^2(0,T;H^{\b+1})}
            +
            \|u\|_{L^2(0,T;H^{\b+1})}
        \Big)
        \|y\|_{H^{-\b}}.
    \end{gather}
    By the hypotheses of \itref{ITEM:nonlinear_loc_convergence_strong}, the right-hand side converges to $0$ as $n\to\infty$.
    Hence~\eqref{EQ:nonlinear_loc_convergence_strong} holds for every
    $y\in C_c^\infty(\R^2;\R^2)\cap H^{-\b}$.

    \textit{\textbf{Step $\mathbf{3}$.}}
    Let now $y\in H^{-\b}$ be arbitrary and fix $\e>0$.
    Since $C_c^\infty(\R^2;\R^2)\cap H^{-\b}$ is dense in $H^{-\b}$, there exists
    $y_\e\in C_c^\infty(\R^2;\R^2)\cap H^{-\b}$ such that
    $\|y-y_\e\|_{H^{-\b}}\leq \e$.
    For every $n\in\N$ and a.e.\ $s\in[0,T]$, using \eqref{EQ:estim_B_1} with $\s=\b+1$,
    \begin{align*}
        &\Big|
            {\big\lan}
                y-y_\e,
                B(u_n(s))-B(u(s))
            \big\ran_{\ss{\!H^{-\b}\!\times\!H^\b}}
        \Big|\\
        &\qquad\leq
        \|y-y_\e\|_{H^{-\b}}
        \Big(
            \|B(u_n(s)-u(s),u_n(s))\|_{H^\b}
            +
            \|B(u(s),u_n(s)-u(s))\|_{H^\b}
        \Big)\\
        &\qquad\leq
        c\,
        \|y-y_\e\|_{H^{-\b}}
        \big(
            \|u_n(s)\|_{H^{\b+1}}
            +
            \|u(s)\|_{H^{\b+1}}
        \big)^2.
    \end{align*}
    Therefore
    \[
        \sup_{t\in[0,T]}
        \bigg|
            \int_0^t
            {\big\lan}
                y-y_\e,
                B(u_n(s))-B(u(s))
            \big\ran_{\ss{\!H^{-\b}\!\times\!H^\b}}
            \d s
        \bigg|
        \leq
        C\,\e
        \Big(
            \|u_n\|_{L^2(0,T;H^{\b+1})}
            +
            \|u\|_{L^2(0,T;H^{\b+1})}
        \Big)^2,
    \]
    for a constant $C>0$ independent of $n$ and $\e$.
    Since $\{u_n\}_{n\in\N}$ is bounded in $L^2(0,T;H^{\b+1})$, we infer that
    \[
        \limsup_{n\to\infty}
        \sup_{t\in[0,T]}
        \bigg|
            \int_0^t
            {\big\lan}
                y-y_\e,
                B(u_n(s))-B(u(s))
            \big\ran_{\ss{\!H^{-\b}\!\times\!H^\b}}
            \d s
        \bigg|
        \leq C'\e
    \]
    for some constant $C'>0$ independent of $\e$.
    On the other hand, by Step~$2$,
    \[
        \lim_{n\to\infty}
        \sup_{t\in[0,T]}
        \bigg|
            \int_0^t
            {\big\lan}
                y_\e,
                B(u_n(s))-B(u(s))
            \big\ran_{\ss{\!H^{-\b}\!\times\!H^\b}}
            \d s
        \bigg|
        =0.
    \]
    Combining the last two estimates and letting $\e\to0^+$, we obtain
    \eqref{EQ:nonlinear_loc_convergence_strong}.
\end{proof}

\begin{proof}[Proof of \itref{ITEM:nonlinear_loc_convergence_dual}.]
    \textit{\textbf{Step $\mathbf{1}$.}}
    Assume first that $y\in C_c^\infty(\R^2;\R^2)\cap H^{\b-1}$.
    Choose $m\in\N$ such that $\supp y\subset \bar B_m$.
    Let $x,z\in H^\b$.
    Since $\chi_m=1$ on $\bar B_m$, we have
    \[
        \big\langle
            B(x,z),y
        \big\rangle_{\ss{\!H^{1-\b}\!\times\!H^{\b-1}}}
        =
        \big\langle
            B(\chi_m x,\chi_m z),y
        \big\rangle_{\ss{\!H^{1-\b}\!\times\!H^{\b-1}}}.
    \]
    By Definition~\ref{DEF:B} with $p=1$, $q=\b-1$, $r=\b-1$, by the embedding
    $H^\b\hookrightarrow H^1$, and by the boundedness of multiplication by
    $\chi_m$ on $H^\b$, there exists a constant $C_m>0$ such that
    \[
        \|B(\chi_m x,\chi_m z)\|_{H^{1-\b}}
        \leq
        C_m \|\chi_m x\|_{H^\b}\|\chi_m z\|_{H^\b}.
    \]
    Therefore
    \begin{equation}
    \label{EQ:local_B_estimate_unified_dual}
        \Big|
            \big\langle
                B(x,z),y
            \big\rangle_{\ss{\!H^{1-\b}\!\times\!H^{\b-1}}}
        \Big|
        \leq
        C_m
        [\![x]\!]_{H^\b_m}
        [\![z]\!]_{H^\b_m}
        \|y\|_{H^{\b-1}}.
    \end{equation}

    \textit{\textbf{Step $\mathbf{2}$.}}
    Let now $y\in C_c^\infty(\R^2;\R^2)\cap H^{\b-1}$.
    For a.e.\ $s\in[0,T]$, by bilinearity of $B$ and by
    \eqref{EQ:local_B_estimate_unified_dual},
    \begin{align*}
        &\Big|
            \big\langle
                B(u_n(s))-B(u(s)),y
            \big\rangle_{\ss{\!H^{1-\b}\!\times\!H^{\b-1}}}
        \Big|\\
        &\qquad=
        \Big|
            \big\langle
                B(u_n(s)-u(s),u_n(s))
                +B(u(s),u_n(s)-u(s)),
                y
            \big\rangle_{\ss{\!H^{1-\b}\!\times\!H^{\b-1}}}
        \Big|\\
        &\qquad\leq
        C_m
        [\![u_n(s)-u(s)]\!]_{H^\b_m}
        \Big(
            [\![u_n(s)]\!]_{H^\b_m}
            +
            [\![u(s)]\!]_{H^\b_m}
        \Big)
        \|y\|_{H^{\b-1}}.
    \end{align*}
    Integrating over $[0,t]$, taking the supremum in $t\in[0,T]$, and using
    H\"older's inequality, we obtain
    \begin{align*}
        &\sup_{t\in[0,T]}
        \bigg|
            \int_0^t
            \big\langle
                B(u_n(s))-B(u(s)),y
            \big\rangle_{\ss{\!H^{1-\b}\!\times\!H^{\b-1}}}
            \d s
        \bigg|\\
        &\qquad\leq
        C_m
        [\![u_n-u]\!]_{L^2(0,T;H^\b_m)}
        \Big(
            \|u_n\|_{L^2(0,T;H^\b)}
            +
            \|u\|_{L^2(0,T;H^\b)}
        \Big)
        \|y\|_{H^{\b-1}}.
    \end{align*}
    By the assumptions of \itref{ITEM:nonlinear_loc_convergence_dual}, the right-hand side converges to $0$ as $n\to\infty$.
    Hence \eqref{EQ:nonlinear_loc_convergence_dual} holds for every
    $y\in C_c^\infty(\R^2;\R^2)\cap H^{\b-1}$.

    \textit{\textbf{Step $\mathbf{3}$.}}
    Let now $y\in H^{\b-1}$ be arbitrary and fix $\e>0$.
    By density, choose $y_\e\in C_c^\infty(\R^2;\R^2)\cap H^{\b-1}$ such that
    $\|y-y_\e\|_{H^{\b-1}}\leq \e$.
    Then, for every $n\in\N$,
    \begin{align*}
        &\sup_{t\in[0,T]}
        \bigg|
            \int_0^t
            \big\langle
                B(u_n(s))-B(u(s)),y-y_\e
            \big\rangle_{\ss{\!H^{1-\b}\!\times\!H^{\b-1}}}
            \d s
        \bigg|\\
        &\qquad\leq
        \|y-y_\e\|_{H^{\b-1}}
        \int_0^T
        \Big(
            \|B(u_n(s))\|_{H^{1-\b}}
            +
            \|B(u(s))\|_{H^{1-\b}}
        \Big)\d s.
    \end{align*}
    By Definition~\ref{DEF:B} with $p=1$, $q=\b-1$, $r=\b-1$ and the embedding
    $H^\b\hookrightarrow H^1$,
    \[
        \|B(v)\|_{H^{1-\b}}
        \leq
        C\|v\|_{H^\b}^2,
        \qquad \forall\, v\in H^\b.
    \]
    Therefore
    \[
        \sup_{t\in[0,T]}
        \bigg|
            \int_0^t
            \big\langle
                B(u_n(s))-B(u(s)),y-y_\e
            \big\rangle_{\ss{\!H^{1-\b}\!\times\!H^{\b-1}}}
            \d s
        \bigg|
        \leq
        C\e
        \Big(
            \|u_n\|_{L^2(0,T;H^\b)}^2
            +
            \|u\|_{L^2(0,T;H^\b)}^2
        \Big).
    \]
    Since $\{u_n\}_{n\in\N}$ is bounded in $L^2(0,T;H^\b)$, we infer that
    \[
        \limsup_{n\to\infty}
        \sup_{t\in[0,T]}
        \bigg|
            \int_0^t
            \big\langle
                B(u_n(s))-B(u(s)),y-y_\e
            \big\rangle_{\ss{\!H^{1-\b}\!\times\!H^{\b-1}}}
            \d s
        \bigg|
        \leq
        C'\e
    \]
    for some constant $C'>0$ independent of $\e$.
    The term with $y_\e$ tends to $0$ by Step~$2$.
    Letting $\e\to0^+$ yields \eqref{EQ:nonlinear_loc_convergence_dual}.
\end{proof}

\section{Gr\"onwall Lemmas}

We recall a weak version of the differential form of the classical Gr\"onwall Lemma, which does not require the differentiability of functions.

\begin{lemma}
\label{LEM:Gronwall_lemma_diff}
    Assume that $f:\R_+\to\R$ is a continuous function and that there exist $a,b\in L^1_{loc}(\R_+)$ such that, in the distributional sense in $\R_+^\ast:=(0,+\infty)$,
    \begin{equation}
    \label{EQ:gronwall_diff}
        f'\leq af+b.
    \end{equation}
    Then
    \begin{equation}
    \label{EQ:gronwall_diff_thesis}
        f(t)\leq e^{A(t)}\bigg(f(0)+\int_0^tb(s) e^{-A(s)}\d s \bigg), \qquad \forall\, t\geq 0,
    \end{equation}
    where
    \[
        A(t):=\int_0^ta(s)\d s, \qquad \forall\, t\geq 0.
    \]
\end{lemma}
\begin{proof}
    Let us define $g:\R_+\to\R$ by
    \[
        g(t):=f(t)e^{-A(t)}-\int_0^tb(s) e^{-A(s)}\d s, \qquad\forall\, t\geq 0.
    \]
    Then the distributional derivative of $g$ on the interval $\R^\ast_+:=(0,+\infty)$, satisfies
    \begin{equation}
    \label{EQ:proof_gronwall_diff}
        g'=f'e^{-A}-afe^{-A}-be^{-A}\leq 0,
    \end{equation}
    where we used the hypothesis in equation~\eqref{EQ:gronwall_diff} for the inequality. 
    Let $\rho\in C_c^\infty(\R^\ast_+)$ be such that $\rho\geq 0$ and $\int_0^{+\infty}\rho(s)\d s=1$. For every $n\in\N$ define $\rho_n(s):=n\rho(ns)$, $s>0$, then  
    \begin{align}
        &\rho_n\in C_c^\infty(\R^\ast_+), \qquad \int_0^{+\infty}\rho_n(s)\d s =1, \qquad \forall\, n\in\N.
    \end{align}
    Let $\varphi$ be a generic non-negative function in $C_c^\infty(\R^\ast_+)$ and fix $n\in\N$. Let us define 
    \[
        \psi_n:=-\varphi+\rho_n\int_0^{+\infty}\varphi(s)\d s.
    \]
    By direct inspection we have $\psi_n\in C_c^\infty(\R^\ast_+)$ and $\int_0^{+\infty}\psi_n(s)\d s=0$, in particular, its primitive function $\Psi_n(t):=\int_0^t\psi_n(s)\d s$, $t>0$, satisfies again $\Psi_n\in C_c^\infty(\R^\ast_+)$.
    Therefore, if we test the inequality~\eqref{EQ:proof_gronwall_diff} againts $\Psi_n$ we reach
    \begin{align}
        0&\geq {\big\lan }g',\Psi_n\big\ran^{}_{{\!\mathcal D'(\R^\ast_+) \times C_c^\infty(\R^\ast_+)}}\\
        &=-{\big\lan }g,\psi_n\big\ran^{}_{{\!\mathcal D'(\R^\ast_+) \times C_c^\infty(\R^\ast_+)}}\\
        &=\int_0^{+\infty}g(t)\varphi(t)\d t -\int_0^{+\infty}g(s)\rho_n(s)\d s\int_0^{+\infty}\varphi(t)\d t\\
        &=\int_0^{+\infty}\varphi(t)\bigg(g(t)-\int_0^{+\infty}g(s)\rho_n(s)\d s\bigg)\d t.
    \end{align}
    This implies, by arbitrariness of $\varphi\in C_c^\infty(\R^\ast_+;\R_+)$ and continuity of $g$, that for all $t\geq 0$
    \[
        g(t)\leq \int_0^{+\infty}g(s)\rho_n(s)\d s=\int_0^{+\infty}g(r/n)\rho(r)\d r\longrightarrow g(0), \qquad \textit{as }{n\to\infty},
    \]
    where we resorted to the Dominated Convergence Theorem to pass to the limit.
    The sought estimate~\eqref{EQ:gronwall_diff_thesis} is derived from $g(t)\leq g(0)$, $t\geq 0$, by recalling the definition of $g$,  and by multiplying by $e^A$.
\end{proof}

We also recall a simple version of the integral Gr\"onwall's Lemma that suits our necessities.
A proof can be found in \cite[proof of Theorem $5$, step $2$]{Flandoli+Romito_2001_Statistically_stationary_solutions_3D_Navier-Stokes_equation_do_not_show_singularities}.

\begin{lemma}
\label{LEM:Gronwall_lemma_negative}
    Assume that $f:\R_+\to\R$ is a non-negative continuous function and that there exist $C_1, C_2>0$ such that
    \begin{equation}
    \label{EQ:gronwall}
        f(t)\leq f(s)-C_1\int_s^tf(r)\d r+ C_2(t-s), \qquad \forall\, t\geq s\geq 0.
    \end{equation}
    Then
    \[
        f(t)\leq \frac{C_2}{C_1}+e^{-C_1t}f(0), \qquad \forall\, t\geq 0.
    \]
\end{lemma}
\end{appendices}


\noindent \textbf{Acknowledgements:} 
Both authors gratefully acknowledge the hospitality and financial support of the Bernoulli Center at EPFL, Lausanne, Switzerland. 
Matteo Ferrari also extends his thanks to the Mathematics Department of the University of York, UK, for their warm hospitality, and to the Erasmus Programme for the funding that supported his stay in York, during which the research for this paper was conducted. 
Moreover, Matteo Ferrari is a member of the Italian "Gruppo Nazionale per l'Analisi Matematica, la Probabilità e le loro Applicazioni (GNAMPA)," which is part of the "Istituto Nazionale di Alta Matematica (INdAM)."\\[3mm]

\printbibliography[heading=bibintoc]

\end{document}